\documentclass[11pt]{amsart}
\pdfoutput=1
\usepackage{amssymb, amsfonts, latexsym, amsthm, amsmath, verbatim, enumerate}
\usepackage{soul} 
\usepackage{paralist}
\usepackage[all]{xy}
\usepackage{amscd}
\usepackage{mathabx}
\usepackage[english]{babel}
\usepackage[margin=1in]{geometry}
\usepackage[pdftex]{graphicx,color} 
\usepackage{tikz-cd}
\usepackage{float}
\floatstyle{boxed}
\restylefloat{figure}
\theoremstyle{plain}
\newtheorem{thm}{Theorem}[section]
\newtheorem{prop}[thm]{Proposition}
\newtheorem{cor}[thm]{Corollary}
\newtheorem*{cor*}{Corollary}
\newtheorem{example}[thm]{Example}
\newtheorem{lem}[thm]{Lemma}
\newtheorem{fact}[thm]{Fact}
\newtheorem*{thm*}{Theorem}

\newtheorem{rmk}[thm]{Remark}
\newtheorem*{rmk*}{Remark}
\theoremstyle{definition}
\newtheorem{defn}[thm]{Definition}
\newtheorem{claim}{Claim}
\newcommand{\G}{\mathit{\mathcal{G}}}

\newcommand{\fr}{\mathit{f_{\mathrm{red}}}}

\newcommand{\f}{\mathit{\mathbb{F}}}
\newcommand{\E}{\mathit{\mathfrak{E}}}
\newcommand{\p}{\mathit{\mathrm{pt}}}
\newcommand{\tH}{\mathit{\tilde{H}}}
\newcommand{\bt}{\mathcal{T}^+}
\newcommand{\bv}{\mathcal{V}^+}
\newcommand{\re}{\mathrm{res}}
\newcommand{\X}{\mathit{\mathcal{X}}}

\newcommand{\q}{\mathit{\mathfrak{q}}}

\newcommand{\ce}{\mathit{\mathfrak{CE}}} 

\newcommand{\CEL}{\mathit{\mathfrak{CLE}}} 
\newcommand{\el}{\mathit{\mathfrak{LE}}}

\newcommand{\hfc}{\mathit{SWFH_{\mathrm{conn}}}}   
\newcommand{\s}{\mathit{\mathfrak{s}}} 
\newcommand{\co}{\mathrm{cor}}
\newcommand{\mbi}{\mathit{\mathbb{I}}}
\newcommand{\ine}{\mathit{Z_{\mathrm{iness}}}}
\newcommand{\hsc}{\mathit{H^{S^1}_{\mathrm{conn}}}}
\newcommand{\gc}{\mathit{\mathrm{gC}}}

\title{Pin(2)-Equivariant Seiberg-Witten Floer Homology of Seifert Fibrations} 
\author{Matthew Stoffregen}
\begin{document}
\begin{abstract}
 We compute the $\mathrm{Pin}(2)$-equivariant Seiberg-Witten Floer homology of Seifert rational homology three-spheres in terms of their Heegaard Floer homology.  As a result of this computation, we prove Manolescu's conjecture that $\beta=-\bar{\mu}$ for Seifert integral homology three-spheres.  We show that the Manolescu invariants $\alpha, \beta,$ and $\gamma$ give new obstructions to homology cobordisms between Seifert fiber spaces, and that many Seifert homology spheres $\Sigma(a_1,...,a_n)$ are not homology cobordant to any $-\Sigma(b_1,...,b_n)$.  We then use the same invariants to give an example of an integral homology sphere not homology cobordant to any Seifert fiber space.  We also show that the $\mathrm{Pin}(2)$-equivariant Seiberg-Witten Floer spectrum provides homology cobordism obstructions distinct from $\alpha,\beta,$ and $\gamma$.  In particular, we identify an $\f[U]$-module called connected Seiberg-Witten Floer homology, whose isomorphism class is a homology cobordism invariant.
\end{abstract}
\maketitle

\section{Introduction}
Let $Y$ be a closed, oriented three-manifold with $b_1=0$ and spin structure $\s$, and let $G=\mathrm{Pin}(2)$, the subgroup $S^1 \cup jS^1$ of the unit quaternions.  Manolescu introduced the $G$-equivariant Seiberg-Witten Floer homology $\mathit{SWFH^G}(Y,\s)$ in \cite{ManolescuPin}, and with it the suite of invariants $\alpha, \beta, \gamma$, defined analogously to the Fr{\o}yshov invariant of the usual, $S^1$-equivariant, Seiberg-Witten Floer homology.  The invariant $\mathit{SWFH^G}(Y,\s)$ is the $G$-equivariant homology of a Conley index coming from finite-dimensional approximation of the Seiberg-Witten equations on $Y$.  As the $G$-equivariant homology of some stable homotopy type, $\mathit{SWFH^G}(Y,\s)$ comes with the structure of a module over $H^*(BG)\simeq\mathbb{F}[q,v]/(q^3),$ where $\mathbb{F}$ is the field of two elements.  The underlying $G$-stable homotopy type, $\mathit{SWF}(Y,\mathfrak{s})$, is also an invariant of the pair $(Y,\mathfrak{s})$.  The invariant $\beta$ was then used to disprove the triangulation conjecture: it is a lift, as a map of sets, from the Rokhlin homomorphism $\theta^H_3 \rightarrow \mathbb{Z}/2$ to $\theta^H_3 \rightarrow \mathbb{Z}$, where $\theta^H_3$ denotes the integral homology cobordism group of integral homology three-spheres.  

Let $Y$ be a Seifert rational homology sphere with spin structure $\mathfrak{s}$, such that the base orbifold of the Seifert fibration of $Y$ has $S^2$ as underlying space\footnote{There are also Seifert fibered rational homology spheres with base orbifold $\mathbb{RP}^2$, and  some of them do not have a Seifert structure over $S^2$. These are not considered in this paper. None of them are integral homology spheres.  Furthermore, in order for a Seifert fiber space $Y$ to be a rational homology sphere, it must fiber over an orbifold with underlying space either $\mathbb{RP}^2$ or $S^2$.}.   In the present paper, we use the description of the Seiberg-Witten moduli space given by Mrowka, Ozsv{\'a}th, and Yu \cite{MOY} to compute $\mathit{SWFH^G}(Y,\mathfrak{s})$, as a module over $\mathbb{F}[q,v]/(q^3)$ (Here, the action of $v$ decreases grading by $4$, and that of $q$ decreases grading by $1$). The description is in terms of the Heegaard Floer homology $\mathit{HF^+}(Y,\s)$, defined in \cite{OzSz1},\cite{OzSz2}.  In particular, this description makes $\mathit{SWFH^G}(Y,\s)$ quickly computable, using the work of Ozsv{\'a}th-Szab{\'o}, N{\'e}methi, and Can-Karakurt \cite{OzSzplumb},\cite{Nemethigr},\cite{CanKarakurt}.  In order to obtain this description, we use both the equivalence of $\mathit{HF^+}$ and $\widecheck{HM}$ due to Kutluhan-Lee-Taubes \cite{KLT1}, and Colin-Ghiggini-Honda \cite{CGH1} and Taubes \cite{Taubes}, and the forthcoming work by Lidman and Manolescu on the equivalence of $\widecheck{HM}$ and $\mathit{SWFH^{S^1}}$ in \cite{LM}.  

Our work in this paper is to relate $\mathit{SWFH^{S^1}}(Y,\s)$ and $\mathit{SWFH^G}(Y,\s)$ when the underlying homotopy type $\mathit{SWF}(Y,\s)$ is simple enough.  This should be compared with \cite{flin2}, in which Lin calculates the $\mathrm{Pin}(2)$-monopole Floer homology in the setting of \cite{flin} for many classes of three-manifolds $Y$ obtained by surgery on a knot.  The approach there is based, similarly, on extracting information from the $S^1$-equivariant theory $\widecheck{HM}(Y,\s)$ of \cite{KM}, when $\widecheck{HM}(Y,\s)$ is simple enough.  

To state the calculation of $\mathit{SWFH^G}(Y,\s)$, let $\mathcal{T}^+$ denote $\mathbb{F}[U,U^{-1}]/U \mathbb{F}[U]$, and $\mathcal{T}^+(i)= \mathbb{F}[U^{-i+1},U^{-i+2},...]/U \mathbb{F}[U]$.  We also introduce the notation $\mathcal{V}^+$ to denote $\mathbb{F}[v,v^{-1}]/v \mathbb{F}[v]$, and $\mathcal{V}^+(i)= \mathbb{F}[v^{-i+1},v^{-i+2},...]/v \mathbb{F}[v]$.  For any graded module $M$, let $M_n$ denote the submodule of homogeneous elements of degree $n$, and define $M[k]$ by $M[k]_n=M_{n+k}$.  Let $\bt_d(n)=\bt(n)[-d]$ and $\bv_d(n)=\bv(n)[-d]$.  The module $\bt_d(n)$ is then supported in degrees from $d$ to $d+2(n-1)$, with the parity of $d$.  Let $Y$ be a Seifert rational homology three-sphere with negative fibration; that is, the orbifold line bundle of $Y$ is of negative degree (see Section \ref{sec:seifmod}).  For example, the Brieskorn sphere $\Sigma(a_1,...,a_n)$, for coprime $a_i$, is of negative fibration.  

Using the graded roots algorithm of N{\'e}methi \cite{Nemethigr}, we may write:
\begin{equation}\label{eq:hf1}
\mathit{HF^+}(Y,\mathfrak{s})=\bt_{s+d_1+2n_1-1}\oplus \bigoplus^{N}_{i=1}\bt_{s+d_i}(\frac{d_{i+1}+2n_{i+1}-d_i}{2})\oplus \bigoplus^N_{i=1} \bt_{s+d_i}(n_i) \oplus J^{\oplus 2}[-s] ,
\end{equation}
for some constants $s,d_i,n_i,N$ and some $\f[U]$-module $J$, all determined by $(Y,\mathfrak{s})$.  Moreover, $d_{i+1} >d_{i}, n_{i+1} <n_i$ for all $i$.  Roughly, in terms of Seiberg-Witten theory, the term $\mathcal{T}^+_{s+d_1+2n_1-1}$ accounts for the reducible critical point, and the modules $\mathcal{T}^{+}_{d_i}(n_i)$ and $\mathcal{T}^+_{d_i}(\frac{d_{i+1}+2n_{i+1}-d_i}{2})$ account for the irreducibles which cancel against the bottom of the infinite $U$-tower.  The term $J^{\oplus 2}$ accounts for the other irreducibles.  

Let us denote by $\re^{\f[U]}_{\f[v]}$ the restriction functor from the map of modules $\f[v] \rightarrow \f[U]$ given by $v \rightarrow U^2$.  The restriction functor converts  $\bt_d(n)$ to $\bv_d(\lfloor \frac{n+1}{2} \rfloor) \oplus \bv_{d+2}(\lfloor \frac{n}{2} \rfloor)$.  

\begin{thm}\label{thm:submain}
 Let $Y$ be a Seifert rational homology three-sphere of negative fibration, fibering over an orbifold with underlying space $S^2$, and let $\s$ be a spin structure on $Y$.  Let $\mathit{HF^+}(Y,\s)$ be as in (\ref{eq:hf1}).  Then there exist constants $(a_i,b_i)$ and an $\f[q,v]/(q^3)$-module $J''$, specified in Corollary \ref{thm:fullswfh} and depending only on the sequence $(d_i,n_i)$, so that, as an $\f[v]$-module:
\begin{eqnarray*} \mathit{SWFH^G}(Y,\mathfrak{s}) & = & \bv_{s+4\lfloor \frac{d_1+2n_1+1}{4} \rfloor} \oplus \bv_{s+1} \oplus \bv_{s+2}\\ && \oplus\bigoplus_{i=1}^{N'} \bv_{s+a_i}(\frac{a_{i+1}+4b_{i+1}-a_i}{4}) \oplus J''[-s] \oplus \re^{\f[U]}_{\f[v]}J[-s].
\end{eqnarray*}

The $q$-action is given by the isomorphism $\bv_{s+2} \rightarrow \bv_{s+1}$ and the map $\bv_{s+1} \rightarrow \bv_{s+4\lfloor \frac{d_1+2n_1+1}{4} \rfloor}$, which is an $\f$-vector space isomorphism in all degrees at least $s+4\lfloor \frac{d_1+2n_1+1}{4} \rfloor$ and vanishes otherwise.  Further, $q$ annihilates $\re^{\f[U]}_{\f[v]}J[-s]$ and $\bigoplus_{i=1}^{N'} \bv_{s+a_i}(\frac{a_{i+1}+4b_{i+1}-a_i}{4})$.  The action of $q$ on $J''$ is specified in Corollary \ref{thm:fullswfh}.
\end{thm}

Theorem \ref{thm:submain} specifies $\alpha, \beta,$ and $\gamma$, which we state as Corollary \ref{cor:betamucor}.  For $Y$ an integral homology three-sphere, let $d(Y)$ be the Heegaard Floer correction term \cite{OzSzgrad}.  Using Theorem \ref{thm:submain} and Theorem \ref{thm:betamu} below we obtain:

\begin{cor}\label{cor:betamucor} 
\begin{enumerate}[(a)]
\item Let $Y$ be a Seifert integral homology sphere of negative fibration.  Then $\beta(Y)=\gamma(Y)=-\bar{\mu}(Y)$, and \[\alpha(Y)= \begin{cases} d(Y)/2, & \mbox{if } d(Y)/2 \equiv -\bar{\mu}(Y) \mbox{ mod } 2 \\
d(Y)/2+1 & \mbox{otherwise.} \end{cases}\]

\item Let $Y$ be a Seifert integral homology sphere of positive fibration.  Then $\alpha(Y)=\beta(Y)=-\bar{\mu}(Y),$ and
 \[\gamma(Y)=\begin{cases} d(Y)/2 & \mbox{if } d(Y)/2 \equiv -\bar{\mu}(Y) \mbox{ mod } 2 \\
d(Y)/2-1 & \mbox{otherwise. } \end{cases} \]
\end{enumerate}
\end{cor}

From Corollary \ref{cor:betamucor}, we see that for Seifert integral homology spheres the Manolescu invariants $\alpha,\beta$, and $\gamma$ are all determined by $d$ and $\bar{\mu}$.  In particular, $\alpha$, $\beta$, and $\gamma$ provide no new obstructions to Seifert spaces bounding acyclic four-manifolds.
  
In \cite{ManolescuPin}, Manolescu also conjectured that for all spin Seifert rational homology spheres  $\beta(Y,\mathfrak{s})=-\bar{\mu}(Y,\mathfrak{s})$, where $\bar{\mu}$ is the Neumann-Siebenmann invariant defined in \cite{NeumannPlumbings}, \cite{Siebenmann}.  We are able to prove part of this conjecture: 
\begin{thm}\label{thm:betamu}
Let $Y$ be a Seifert integral homology three-sphere.  Then $\beta(Y)=-\bar{\mu}(Y)$.  
\end{thm}

We prove Theorem \ref{thm:betamu} by showing that $\beta$ is controlled by the degree of the reducible, and by using a result of Ruberman and Saveliev \cite{RubermanSaveliev10} that gives $\bar{\mu}$ as a sum of eta invariants.  

Fukumoto-Furuta-Ue showed in \cite{FFU} that $\bar\mu$ is a homology cobordism invariant for many classes of Seifert spaces, and Saveliev \cite{Savelievmu} extended this to show that Seifert integral homology spheres with $\bar\mu\neq 0$ have infinite order in $\theta^H_3$.  Theorem \ref{thm:betamu} generalizes the result of Fukumoto-Furuta-Ue, showing that the Neumann-Siebenmann invariant $\bar\mu$, restricted to Seifert integral homology spheres, is a homology cobordism invariant.   

For Seifert spaces with $\mathit{HF^+}(Y,\mathfrak{s})$ of a special form, $\mathit{SWFH^G}(Y,\mathfrak{s})$ may be expressed more compactly than is evident in the statement of Theorem \ref{thm:submain}.  If $Y$ is negative and  
\begin{equation}\label{eq:1proj}
\mathit{HF^{+}}(Y,\mathfrak{s})=\mathcal{T}^+_{d}\oplus \mathcal{T}^{+}_{-2n+1}(m) \oplus \bigoplus_{i \in I} \mathcal{T}^{+}_{a_i}(m_i)^{\oplus 2},
\end{equation}
for some index set $I$, we say that $(Y,\mathfrak{s})$ is \emph{of projective type}.    We will say that $Y$ is of projective type if $Y$ is an integral homology sphere such that (\ref{eq:1proj}) holds.   There are many examples of such Seifert spaces, among them $\Sigma(p,q,pqn \pm 1)$, by work of N{\'e}methi and Borodzik \cite{Nem07},\cite{BN} and Tweedy \cite{Tweedy}.  The condition (\ref{eq:1proj}) also admits a natural expression in terms of graded roots; see Section \ref{subsec:projtyp}.  
\begin{thm}\label{thm:main} If $(Y,\mathfrak{s})$ is of projective type, as in (\ref{eq:1proj}), then:\\If $d \equiv 2n+2$ mod $4$,
\begin{equation}\label{eq:2n2} \mathit{SWFH^G}(Y,\mathfrak{s}) = \mathcal{V}^+_{d+2} \oplus \mathcal{V}^+_{-2n+1} \oplus \mathcal{V}^+_{-2n+2} \oplus \mathcal{V}^+_{-2n+3}(\lfloor \frac{m}{2} \rfloor)\oplus \bigoplus_{i \in I} \mathcal{V}^{+}_{a_i}(\lfloor\frac{m_i+1}{2} \rfloor) \oplus \bigoplus_{i \in I} \mathcal{V}^{+}_{a_i+2}(\lfloor\frac{m_i}{2} \rfloor). \end{equation}
 If $d \equiv 2n$ mod $4$,
\begin{equation}\label{eq:2n} \mathit{SWFH^G}(Y,\mathfrak{s}) = \mathcal{V}^+_{d} \oplus \mathcal{V}^+_{-2n+1} \oplus \mathcal{V}^+_{-2n+2} \oplus \mathcal{V}^+_{-2n+3}(\lfloor \frac{m}{2} \rfloor) \oplus \bigoplus_{i \in I} \mathcal{V}^{+}_{a_i}(\lfloor\frac{m_i+1}{2} \rfloor) \oplus \bigoplus_{i \in I} \mathcal{V}^{+}_{a_i+2}(\lfloor\frac{m_i}{2} \rfloor) .\end{equation}

The $q$-action is given by the isomorphism $\bv_{-2n+2} \rightarrow \bv_{-2n+1}$ and the map $\bv_{-2n+1} \rightarrow \bv_{d+2}$ (if $d \equiv 2n+2 \; \mathrm{mod}\; 4$), or $\bv_{-2n+1} \rightarrow \bv_{d}$ (if $d \equiv 2n \; \mathrm{mod} \; 4$), which is an $\f$-vector space isomorphism in all degrees at least $d+2$ (respectively, $d$), and vanishes otherwise.  In (\ref{eq:2n2}) and (\ref{eq:2n}), $q$ acts on $\bv_{-2n+3}(\lfloor \frac{m}{2} \rfloor)$ as the unique nonzero map $\bv_{-2n+3}(\lfloor \frac{m}{2} \rfloor)\rightarrow \bv_{-2n+2}$.  The action of $q$ annihilates $\bigoplus_{i \in I} \mathcal{V}^{+}_{a_i}(\lfloor\frac{m_i+1}{2} \rfloor) \oplus \bigoplus_{i \in I} \mathcal{V}^{+}_{a_i+2}(\lfloor\frac{m_i}{2} \rfloor)$. 
\end{thm}

To prove Theorem \ref{thm:submain}, we use \cite{MOY} to show that a space representative of the stable homotopy type $\mathit{SWF}(Y,\mathfrak{s})$ naturally splits into two disjoint pieces, which are interchanged by the action of $j \in G$.  Say 
\[X=\mathit{SWF}(Y,\mathfrak{s})/(\mathit{SWF}(Y,\mathfrak{s})^{S^1}).\]
Then 
\begin{equation}\label{eq:jspdef}
X=X_+ \vee jX_+.
\end{equation}
That is, $X$ is a wedge sum of two components related by the action of $j$.  This description in turn connects the CW chain complexes of $EG \wedge_G \mathit{SWF}(Y,\mathfrak{s})$ and $ES^1 \wedge_{S^1} \mathit{SWF}(Y,\mathfrak{s})$.  A careful, but entirely elementary, analysis of the differentials in these two complexes then yields Theorem \ref{thm:submain}.  

Along the way, we provide more information about the $G$-equivariant CW chain complex of $\mathit{SWF}(Y,\mathfrak{s})$ than is reflected in Theorem \ref{thm:submain}.  To store this information, we define in Section \ref{sec:2} the \emph{local equivalence} and \emph{chain local equivalence} classes of pairs $(Y,\mathfrak{s})$ for $Y$ any rational homology three-sphere, not necessarily Seifert.  The local equivalence and chain local equivalence classes, denoted $[\mathit{SWF}(Y,\mathfrak{s})]_l$ and $[\mathit{SWF}(Y,\mathfrak{s})]_{cl}$ respectively, are homology cobordism invariants refining the Manolescu invariants $\alpha,\beta,\gamma$.  The chain local equivalence class takes values in the set $\ce$ of homotopy-equivalence classes of chain complexes of a certain form.  Local and chain local classes reflect how irreducible cells are attached to the reducible; they are so named because they store information about $\mathit{SWF}(Y,\mathfrak{s})$ only near the reducible.  We will compute the chain local equivalence class of negative Seifert rational homology spheres.  One use of local equivalence and chain local equivalence classes is that they behave well under connected sums, while the behavior of $\alpha,\beta,$ and $\gamma$ under connected sum is more complicated.  In particular, it was shown in \cite{Manolescusurv} that $\alpha,\beta$, and $\gamma$ are not homomorphisms.  

We call rational homology three-spheres $Y_1$ and $Y_2$ (integral) \emph{homology cobordant} if there exists a compact oriented four-manifold $W$ with $\partial W = Y_1 \amalg -Y_2$ so that the maps induced by inclusion $H_*(Y_i;\mathbb{Z}) \rightarrow H_*(W; \mathbb{Z})$ are isomorphisms for $i=1,2$.  As a corollary of the calculation of $[\mathit{SWF}(Y,\mathfrak{s})]_{cl}$ for Seifert fiber spaces we obtain:

\begin{cor}\label{cor:dini}
The sets $\{ d_i \}_{i}, \{ n_i \}_{i}$ in Theorem \ref{thm:submain} are integral homology cobordism invariants of negative Seifert fiber spaces.  That is, say $Y_1$ and $Y_2$ are negative Seifert integral homology spheres with $Y_1$ homology cobordant to $Y_2$. Let $S_i$ be the set of isomorphism classes of simple summands of $\mathit{HF^+}(Y_i)$ that occur an odd number of times in the decomposition (\ref{eq:hf1}).  Then $S_1=S_2$.  
\end{cor}

We obtain Corollary \ref{cor:dini} by showing that $\{d_i\}_{i}$ and $\{ n_i \}_i$ determine $[\mathit{SWF}(Y,\mathfrak{s})]_{cl}$.  The following corollaries are further applications of chain local equivalence.
\begin{cor}\label{cor:dinig}
Let $Y$ be a rational homology three-sphere with spin structure $\mathfrak{s}$.  Then there is a homology-cobordism invariant, $\hfc(Y,\mathfrak{s})$, the \emph{connected Seiberg-Witten Floer homology} of $(Y,\mathfrak{s})$, taking values in isomorphism classes of $\f[U]$-modules.  More specifically, $\hfc(Y,\mathfrak{s})$ is the isomorphism class of a summand of $\mathit{HF_{\mathrm{red}}}(Y,\mathfrak{s})$.  
\end{cor}
\begin{cor}\label{cor:dini2}
Let $(Y_1,\mathfrak{s}_1)$ be a negative Seifert rational homology three-sphere with spin structure, with $\mathit{HF^+}(Y_1,\mathfrak{s}_1)$ as in (\ref{eq:hf1}).  Then
\begin{equation}\label{eq:hfcseif}   \hfc(Y_1,\mathfrak{s}_1) =\bigoplus^{N}_{i=1}\bt_{s+d_i}(\frac{d_{i+1}+2n_{i+1}-d_i}{2})\oplus \bigoplus^{N}_{i=1}\bt_{s+d_i}(n_i).
\end{equation}
In particular, if $Y_1$ is an integral homology sphere and $Y_2$ is any integral homology sphere homology cobordant to $Y_1$, then $\widecheck{HM}(Y_2) \simeq \mathit{HF^+}(Y_2)$ contains a summand isomorphic to (\ref{eq:hfcseif}), as $\f[U]$-modules.  
\end{cor}

\begin{rmk}\label{rmk:spincobordism} 
In fact, $\mathit{SWFH_{\mathrm{conn}}}(Y,\s)$ is an invariant of spin rational homology cobordism, for $Y$ a rational homology three sphere.  
\end{rmk}

From Corollary \ref{cor:dini2} and (\ref{eq:hf1}), we see that for Seifert integral homology spheres $Y$, $\hfc(Y,\s)=0$ if and only if $d(Y,\s)/2=-\bar{\mu}(Y,\s)$.  

The connected Seiberg-Witten Floer homology is constructed using the CW chain complex of a  space representative $X$ of $\mathit{SWF}(Y,\mathfrak{s})$.  The CW chain complex $C^{CW}_*(X)$ splits, as a module over $C^{CW}_*(G)$, into a direct sum of two subcomplexes, with one summand attached to the $S^1$-fixed-point set, and the other a free $C^{CW}_*(G)$-module.  Roughly, the $S^1$-Borel homology of the former component is $\hfc(Y,\mathfrak{s})$.  As an application of the Corollaries \ref{cor:dinig} and \ref{cor:dini2}, we have:

\begin{cor}\label{cor:examplecomp}
The spaces $\Sigma(5,7,13)$ and $\Sigma(7,10,17)$ satisfy 
\[d(\Sigma(5,7,13))=d(\Sigma(7,10,17))=2,\]
\[\bar{\mu}(\Sigma(5,7,13))=\bar{\mu}(\Sigma(7,10,17))=0.\]
However, $\hfc(\Sigma(5,7,13))=\bt_1(1)$, while $\hfc(\Sigma(7,10,17))=\bt_{-1}(2) \oplus \bt_{-1}(1)$.  Thus $\Sigma(5,7,13)$ and $\Sigma(7,10,17)$ are not homology cobordant, despite having the same $d$, $\bar{\mu}$, $\alpha$, $\beta$, and $\gamma$ invariants. 
\end{cor}

There are many other examples of homology cobordism classes that are distinguished by $d_i,n_i$, but not by $d$ and $\bar{\mu}$.  As an example, we have the following Corollary.
\begin{cor}\label{cor:kara}
The Seifert space $\Sigma(7,10,17)$ is not homology cobordant to $\Sigma(p,q,pqn \pm 1)$ for any $p,q,n$.  
\end{cor}
This result follows from Corollary \ref{cor:dini}.  Indeed, since $\Sigma(p,q,pqn \pm 1)$ are of projective type, $\hfc(\Sigma(p,q,pqn \pm 1))$ is a simple $\f[U]$-module, using the definition (\ref{eq:1proj}) and  equation (\ref{eq:hfcseif}).  On the other hand, 
\[\hfc(\Sigma(7,10,17))=\bt_{-1}(2) \oplus \bt_{-1}(1),\]
so Corollary \ref{cor:kara} follows.

Moreover, using a calculation from \cite{Manolescusurv}, we are able to show the existence of three-manifolds not homology cobordant to any Seifert fiber space.  (This result was announced earlier by Fr{\o}yshov using instanton homology.)  For example, we have:

\begin{cor}\label{cor:connsums}
The connected sum $\Sigma(2,3,11) \# \Sigma(2,3,11)$ is not homology cobordant to any Seifert fiber space.  
\end{cor}
\begin{proof}
In \cite{Manolescusurv}, Manolescu shows $\alpha(\Sigma(2,3,11) \# \Sigma(2,3,11))=\beta(\Sigma(2,3,11) \# \Sigma(2,3,11))=2$, while $\gamma(\Sigma(2,3,11)\# \Sigma(2,3,11))=0$.  In addition, $d(\Sigma(2,3,11))=2$, so $d(\Sigma(2,3,11) \# \Sigma(2,3,11))=4$.  To obtain a contradiction, say first that $\Sigma(2,3,11) \# \Sigma(2,3,11)$ is homology cobordant to a negative Seifert space $Y$.  Corollary \ref{cor:betamucor} implies 
\[2=\beta(\Sigma(2,3,11) \# \Sigma(2,3,11)) = \beta(Y)=\gamma(Y) = \gamma(\Sigma(2,3,11) \# \Sigma(2,3,11))=0.\]
a contradiction.  Say instead that $\Sigma(2,3,11) \# \Sigma(2,3,11)$ is homology cobordant to a positive Seifert space $Y$.  Then by Corollary \ref{cor:betamucor}, $\gamma(Y)=d(Y)/2=d(\Sigma(2,3,11) \# \Sigma(2,3,11))/2=2$.  However, $\gamma(Y)=0$, again a contradiction, completing the proof.  
\end{proof}
Note that Corollary \ref{cor:connsums} readily implies the following statement for knots.
\begin{cor}\label{cor:knotsums}
There exist knots, such as the connected sum of torus knots $T(3,11) \# T(3,11)$, which are not concordant to any Montesinos knot.
\end{cor}

As an application of $G$-equivariant Seiberg-Witten Floer homology we prove that many Seifert integral homology spheres of negative fibration are not homology cobordant to any Seifert integral homology sphere of positive fibration.  For instance, we have:

\begin{cor}\label{cor:12k+7} The Seifert spaces $\Sigma(2,3,12k+7)$, for $k\geq 0$, are not homology cobordant to $-\Sigma(a_1,a_2,..., a_n)$ for any choice of relatively prime $a_i$.  
\end{cor}

This Corollary is a direct consequence of Corollary \ref{cor:betamucor}, which shows that if $Y$ is a negative Seifert space with $d(Y)/2 \neq -\bar{\mu}(Y)$, then $Y$ is not homology cobordant to any positive Seifert space.  We note $d(\Sigma(2,3,12k+7))=0$ and $\bar{\mu}(\Sigma(2,3,12k+7))=1$, and the Corollary follows.  This should be compared with a result of Fintushel-Stern \cite{FintushelStern85} that gives a similar conclusion: If $R(a_1,...,a_n)>0$, then $\Sigma(a_1,...,a_n)$ is not oriented cobordant to any connected sum of positive Seifert homology spheres by a positive definite cobordism $W$, where $H_1(W;\mathbb{Z})$ contains no $2$-torsion.  However, there are examples with $R<0$, but $d/2 \neq -\bar{\mu}$, so we can apply Corollary \ref{cor:betamucor}.  For instance, $\Sigma(2,3,7)$ has $R$-invariant $-1$, but $d \neq -\bar{\mu}$.  Thus, Corollary \ref{cor:12k+7} is not detected by the $R$-invariant.\\

The organization of the paper is as follows.  In Section \ref{sec:2} we provide the necessary equivariant topology constructions and define local and chain local equivalence.  In Section \ref{sec:jsplit} we compute the $G$-Borel homology of $j$-split spaces.  In Section \ref{sec:gauge} we review the finite-dimensional approximation of \cite{ManolescuPin}.  In Section \ref{sec:seifmod} we recall the results of \cite{MOY} and prove Theorems \ref{thm:submain}, \ref{thm:betamu}, and \ref{thm:main}.  In Section \ref{sec:example} we provide applications and examples of the homology calculation.  Throughout the paper all homology will be taken with $\mathbb{F}=\mathbb{Z}/2$ coefficients, unless stated otherwise.

\section*{Acknowledgements}
The author is particularly grateful to Ciprian Manolescu for his guidance and constant encouragement. The author would also like to thank Cagri Karakurt, Tye Lidman, and Francesco Lin for helpful conversations.  
\section{Spaces of type SWF}\label{sec:2}
\subsection{$G$-CW Complexes}
	In this section we recall the definition of spaces of type SWF from \cite{ManolescuPin}, and introduce local equivalence.  Spaces of type SWF are the output of the construction of the Seiberg-Witten Floer stable homotopy type of \cite{ManolescuPin} and \cite{ManolescuK}; see Section \ref{sec:gauge}. 
	
	First, we recall some basics of equivariant algebraic topology from \cite{tomDieck}.  The reader is encouraged to consult both \cite{ManolescuPin} and \cite{tomDieck} for a fuller discussion.  For now, $G$ will denote a compact Lie group.  We define a $G$-equivariant $k$-cell as a copy of $G/H \times D^k$, where $H$ is a closed subgroup of $G$.  A (finite) equivariant $G$-CW decomposition of a relative $G$-space $(X,A)$, where the action of $G$ takes $A$ to itself, is a filtration $(X_n| n \in \mathbb{Z}_{\geq 0})$ such that 
	\begin{itemize}
	\item $A \subset X_0$ and $X=X_n$ for $n$ sufficiently large.
	\item The space $X_n$ is obtained from $X_{n-1}$ by attaching $G$-equivariant $n$-cells.
	\end{itemize}
When $A$ is a point, we call $(X,A)$ a pointed $G$-CW complex.  

Let $EG$ be the total space of the universal bundle of $G$. For two pointed $G$-spaces $X_1$ and $X_2$, write: 
\[X_1 \wedge_G X_2 = (X_1 \wedge X_2) / (gx_1 \times x_2 \sim x_1 \times gx_2).\]  
The Borel homology of $X$ is given by
\[\tilde{H}_*^G(X) =\tilde{H}_*(EG_+ \wedge_G X),\]
where $EG_+$ is $EG$ with a disjoint basepoint.
Similarly, we define Borel cohomology:
\[\tilde{H}_G^*(X) = \tilde{H}^*(EG_+ \wedge_G X).\]
Additionally, we have a map given by projecting onto the first factor:
\[f: EG_+ \wedge _G X \rightarrow BG_+.\]
From $f$ we have a map $p_G= f^*:H^*(BG) \rightarrow \tilde{H}^*_G(X)$.  Then $H^*(BG)$ acts on $\tilde{H}_*^G(X)$, by composing $p_G$ with the cap product action of $\tilde{H}^*_G(X)$ on $\tilde{H}^G_*(X)$.  We may also define the unpointed version of the above constructions in an apparent way.

As an example, consider the case $G=S^1$.  Here $BS^1=\mathbb{C}P^{\infty}$, so $H^*(BS^1)=\f[U]$, with $\mathrm{deg} \; U = 2$.  Then $\f[U]$ acts on $H^{S^1}(X)$, for $X$ any $S^1$-space.  

From now on we let $G=\mathrm{Pin}(2)$.  The group $G=\mathrm{Pin}(2)$ is the set $S^1 \cup jS^1 \subset \mathbb{H}$, where $S^1$ is the unit circle in the $\langle 1, i \rangle$ plane.  The group action of $G$ is induced from the group action of the unit quaternions.  In order to agree with the conventions of \cite{ManolescuPin} we deal with left $G$-spaces.  Manolescu shows in \cite{ManolescuPin} that $H^*(BG)=\f[q,v]/(q^3)$, where $\mathrm{deg} \; q =1$ and $\mathrm{deg} \; v =4$, so $H^G_*(X)$ is naturally an $\f[q,v]/(q^3)$-module for $X$ a $G$-space.  Moreover $S^\infty=S(\mathbb{H}^\infty)$ has a free action by the quarternions, making $S^\infty$ a free $G$-space.  Since $S^\infty$ is contractible, we identify $EG=S^\infty$.  We may view $EG=S^\infty$ also as $ES^1$ (as an $S^1$-space) by forgetting the action of $j$.  

We will also need to relate $G$-Borel homology and $S^1$-Borel homology.  Consider 
\[f: \mathbb{C}P^\infty = BS^1 \rightarrow BG, \]
the map given by quotienting by the action of $j \in G$ on $BS^1=ES^1/S^1$.  Then we have the following fact (for a proof, see \cite[Example 2.11]{ManolescuPin}):  

\begin{fact}\label{fct:bs1bg}
The natural map
\[f^*=\mathrm{res}^G_{S^1}: H^*(BG) \rightarrow H^*(BS^1) \]
is an isomorphism in degrees divisible by $4$, and zero otherwise.  In particular, $v \rightarrow U^2$.  
\end{fact}

Moreover, for $X$ a $G$-space, we have a natural map
\[g: EG_+ \wedge_{S^1} X \rightarrow EG_+ \wedge_G X.\]
The map $g$ induces a map 
\[g_*=\co^{S^1}_G: \tilde{H}^{S^1}_*(X) \rightarrow \tilde{H}^G_*(X),\]
called the corestriction map.  As a Corollary of Fact \ref{fct:bs1bg}, we have a relationship between the action of $U$ and $v$ (see \cite[\S III.1]{tomDieck}):
\begin{fact}\label{fct:bs1bg2}
Let $X$ be a $G$-space.  Then, for every $x \in H^{S^1}_*(X)$,
\[v(\co^{S^1}_G(x))=\co^{S^1}_G(U^2x).\]
\end{fact}  

We shall use that Borel homology with $\mathbb{F}$ coefficients behaves well with respect to suspension.  If $V$ is a finite-dimensional (real) representation of $G$, let $V^+$ be the one-point compactification, where $G$ acts trivially on $V^+-V$.  Then $\Sigma^VX=V^{+} \wedge X$ will be called the suspension of $X$ by the representation $V$.       

We mention the following representations of $G$:

\begin{itemize}
\item Let $\tilde{\mathbb{R}}^s$ be the vector space $\mathbb{R}^s$ on which $j$ acts by $-1$, and $e^{i\theta}$ acts by the identity, for all $\theta$.
\item We let $\tilde{\mathbb{C}}$ be the representation of $G$ on  $\mathbb{C}$ where $j$ acts by $-1$, and $e^{i\theta}$ acts by the identity for all $\theta$.
\item The quaternions $\mathbb{H}$, on which $G$ acts by multiplication on the left.
\end{itemize}
 
\begin{defn}\label{thm:swfdefn}
Let $s \in \mathbb{Z}_{\geq 0}$.  A \emph{space of type SWF} at level $s$ is a pointed finite $G$-CW complex $X$ with
\begin{itemize}
\item The $S^1$-fixed-point set $X^{S^1}$ is $G$-homotopy equivalent to $(\tilde{\mathbb{R}}^s)^+$.
\item The action of $G$ on $X-X^{S^1}$ is free.  
\end{itemize}
\end{defn}	

\begin{rmk}
We list some examples of spaces of type SWF.  The simplest space of type SWF is $S^0$.  More interesting examples may be produced as follows.  Let $X$ be a free, finite $G$-CW complex.  Then, let \[\tilde{\Sigma}X=X \times [0,1] / (0,x) \sim (0,x')\; \mathrm{and}\; (1,x) \sim (1,x')\; \mathrm{for\;all}\; x, x' \in X.\]
We call $\tilde{\Sigma}X$ the unreduced suspension of $X$.  Here $G$ acts on $\tilde{\Sigma}X$ by multiplication on the first factor.  We fix one of the cone points as the base point.  Then the $S^1$-fixed-point set is precisely $S^0 \subset \tilde{\Sigma}X$, and $\tilde{\Sigma}X$ is a space of type SWF.  As a particular example, let $X=G$, where $G$ acts on $X$ by multiplication on the left, as usual.  Then $\tilde{\Sigma}X$ is, topologically, the suspension of two disjoint circles.
\end{rmk}

We also find it convenient to recall the definition of \emph{reduced Borel homology}, for spaces $X$ of type SWF:
\begin{equation}\label{eq:s1redhom}
\tilde{H}^{S^1}_{*,\mathrm{red}}(X)=\tilde{H}^{S^1}_*(X)/\mathrm{Im} \, U^N,
\end{equation}
for $N \gg 0$.  Indeed, for all $N$ sufficiently large $\mathrm{Im} \, U^N=\mathrm{Im} \, U^{N+1}$, so $\tilde{H}^{S^1}_{*,\mathrm{red}}(X)$ is well-defined.

Associated to a space $X$ of type SWF at level $s$, we take the Borel cohomology $\tilde{H}_G^*(X)$, from which we define $a(X),b(X),$ and $c(X)$ as in \cite{ManolescuPin}:
\begin{equation}\label{eq:adef}
a(X)=\mathrm{min}\{ r \equiv s \;\mathrm{mod}\; 4 \mid \exists \, x \in \tilde{H}^r_G(X), v^lx \neq 0 \; \mathrm{ for\; all} \; l \geq 0 \}, \
\end{equation}
\[
b(X)=\mathrm{min}\{ r \equiv s+1 \;\mathrm{mod}\; 4 \mid \exists \, x \in \tilde{H}^r_G(X), v^lx \neq 0 \; \mathrm{for\; all}\; l \geq 0 \}-1,  
\]
\[
c(X)=\mathrm{min}\{ r \equiv s+2 \;\mathrm{mod}\; 4 \mid \exists \, x \in \tilde{H}^r_G(X), v^lx \neq 0 \;\mathrm{for\; all}\; l \geq 0 \}-2. 
\]
The well-definedness of $a,b$, and $c$ follows from the Equivariant Localization Theorem (see \cite{tomDieck}).  We also list an equivalent definition of $a,b,$ and $c$ from \cite{ManolescuPin}, using homology: 
\begin{equation}\label{eq:aaltdef}
a(X)=\mathrm{min}\; \{ r \equiv s \; \mathrm{mod}\; 4 \mid \exists \, x \in \tH^G_r(X), x \in \mathrm{Im}\, v^l \;\mathrm{for\; all}\; l \geq 0 \},
\end{equation}
\[
b(X)=\mathrm{min}\; \{ r \equiv s+1 \; \mathrm{mod}\; 4 \mid \exists \, x \in \tH^G_r(X), x \in \mathrm{Im}\, v^l \;\mathrm{for\; all}\; l \geq 0 \}-1,
\]\[
c(X)=\mathrm{min}\; \{ r \equiv s+2 \; \mathrm{mod}\; 4 \mid \exists \, x \in \tH^G_r(X), x \in \mathrm{Im}\, v^l \;\mathrm{for\; all}\; l \geq 0 \}-2.
\]
\begin{rmk}  The Manolescu invariants of \cite{ManolescuPin} are defined in terms of $a,b$, and $c$, as we will review in Section \ref{sec:gauge}. \end{rmk}
\begin{defn}[see \cite{ManolescuK}]
Let $X$ and $X'$ be spaces of type SWF, $m,m' \in \mathbb{Z}$, and $n,n' \in \mathbb{Q}$.  We say that the triples $(X,m,n)$ and $(X',m',n')$ are \emph{stably equivalent} if $n-n' \in \mathbb{Z}$ and there exists a $G$-equivariant homotopy equivalence, for some $r \gg 0$ and some nonnegative $M \in \mathbb{Z}$ and $N \in \mathbb{Q}$:
\begin{equation}\label{eq:gstabdef} \Sigma^{r\mathbb{R}} \Sigma^{(M-m)\tilde{\mathbb{R}}}\Sigma^{(N-n)\mathbb{H}}X \rightarrow \Sigma^{r\mathbb{R}} \Sigma^{(M'-m')\tilde{\mathbb{R}}}\Sigma^{(N-n')\mathbb{H}}X'.\end{equation}
\end{defn}
Let $\mathfrak{E}$ be the set of equivalence classes of triples $(X,m,n)$ for $X$ a space of type $\mathrm{SWF}$, $m\in \mathbb{Z}$, $n \in \mathbb{Q}$, under the equivalence relation of stable $G$-equivalence\footnote{This convention is slightly different from that of \cite{ManolescuK}.  The object $(X,m,n)$ in the set of stable equivalence classes $\mathfrak{E}$, as defined above, corresponds to $(X,\frac{m}{2},n)$ in the conventions of \cite{ManolescuK}.}.  The set $\mathfrak{E}$ may be considered as a subcategory of the $G$-equivariant Spanier-Whitehead category \cite{ManolescuPin}, by viewing $(X,m,n)$ as the formal desuspension of $X$ by $m$ copies of $\tilde{\mathbb{R}}^+ $ and $n$ copies of $\mathbb{H}^+$.  For $(X,m,n),(X',m',n') \in \mathfrak{E}$, a map $(X,m,n) \rightarrow (X',m',n')$ is simply a map as in (\ref{eq:gstabdef}) that need not be a homotopy equivalence.  We define Borel homology for $(X,m,n) \in \E$ by 
\begin{equation}\label{eq:borelgeneral}
\tilde{H}^{G}_*((X,m,n))=\tilde{H}^G_*(X)[m+4n].
\end{equation}

\begin{defn}\label{def:loceq}
We call $X_1,X_2 \in \E$ \emph{locally equivalent} if there exist $G$-equivariant (stable) maps
\[\phi: X_1 \rightarrow X_2, \]
\[\psi: X_2 \rightarrow X_1, \]
 which are $G$-homotopy equivalences on the $S^1$-fixed-point set.  For such $X_1,X_2$, we write $X_1 \equiv_{l} X_2$, and let $\mathfrak{LE}$ denote the set of local equivalence classes.
 
\end{defn} 
 Local equivalence is easily seen to be an equivalence relation.  The set $\mathfrak{LE}$ comes with an abelian group structure, with multiplication given by smash product.  One may check that inverses are given by Spanier-Whitehead duality.

\subsection{$G$-CW decompositions of $G$-spaces}\label{subsec:23}
Throughout this section $X$ will denote a space of type SWF.  Here we will give example $G$-CW decompositions and construct a $G$-CW structure on smash products of $G$-spaces.

For $W$ a CW complex, we write $C^{CW}_*(W)$ for the corresponding cellular (CW) chain complex.  We fix a convenient CW decomposition of $G$.  The $0$-cells are the points $1,j,j^2,j^3$ in $G$, and the $1$-cells are $s,js,j^2s,j^3s$, where $s=\{e^{i\theta} \in S^1 \mid \theta \in (0,\pi)\}$.  We identify each of the cells of this CW decomposition with its image in $C^{CW}_{*}(G)$, the corresponding $CW$ chain complex of $G$.  Then $\partial(s)=1+j^2$.  To ease notation, we will refer to $C^{CW}_{*}(G)$ by $\mathcal{G}$.  

We will use that this CW decomposition also induces a CW decomposition of $S^1$, for which $C^{CW}_*(S^1)$ is the subcomplex of $\G$ generated by $1,j^2,s,j^2s$.  

A $G$-CW decomposition of $X$ also induces a CW decomposition of $X$, using the decomposition of $G$ into cells as above, which we will call a $G$-compatible CW decomposition of $X$. 

\begin{example}\label{ex:rsusp} Note that the representation $(\tilde{\mathbb{R}}^s)^+$ admits a $G$-CW decomposition with $0$-skeleton a copy of $S^0$ on which $G$ acts trivially, and an $i$-cell $c_i$ of the form $D^i \times \mathbb{Z}/2$ for $i \leq s$.  One of the two points of the $0$-skeleton of $(\tilde{\mathbb{R}}^s)^+$ is fixed as the basepoint.  \end{example}  
In particular, any space of type SWF has a $G$-CW decomposition with a subcomplex as in Example \ref{ex:rsusp}.  
\begin{example}\label{ex:hsusp}We find a CW decomposition for $\mathbb{H}^+$ as a $G$-space.  We write elements of $\mathbb{H}$ as pairs of complex numbers $(z,w)=(r_1e^{i\theta_1},r_2e^{i\theta_2})$ in polar coordinates.  The action of $j$ is then given by $j(z,w)=(-\bar{w},\bar{z})$.  Fix the point at infinity as the base point.  We let $(0,0)$ be the ($G$-invariant) $0$-cell labelled $r_0$.  We let $y_1$ be the $G$-$1$-cell given by the orbit of $\{(r_1,0) \mid r_1 > 0\}$:
\[\{(r_1e^{i\theta},r_2e^{i\theta})\mid \, r_1r_2=0\}.\]
We take $y_2$ the $G$-$2$-cell given by the orbit of $\{(r_1,r_2) \mid r_1r_2 \neq 0.\}$:
\[\{(r_1e^{i\theta_1},r_2e^{i\theta_2}) \mid \, \theta_1=\theta_2 \;\mathrm{or}\; \theta_1=\theta_2+\pi \;\mathrm{mod}\; 2\pi\}.\] 
Finally, $y_3$ consists of the orbit of $\{(r_1e^{i\theta_1},r_2) \mid \theta_1\in (0,\pi),\;r_1r_2 \neq 0.\}:$
\[\{(r_1e^{i\theta_1},r_2e^{i\theta_2}) \mid \, \theta_1 \neq \theta_2 \; \mathrm{mod} \; \pi\}.\] 
\end{example}

We now give $X_1 \wedge X_2$ a $G$-CW structure for $X_1$ and $X_2$ spaces of type SWF.  To do so, we proceed cell by cell on both factors, so we need only find a $G$-CW structure on $G \times G$, $\mathbb{Z}/2 \times G$, and $\mathbb{Z}/2 \times \mathbb{Z}/2$, each with the diagonal $G$-action.  The space $\mathbb{Z}/2 \times G$ has a $G$-CW decomposition as $G \amalg G$, as may be seen directly, and $\mathbb{Z}/2 \times \mathbb{Z}/2$ may be written as a disjoint union of $G$-$0$-cells $\mathbb{Z}/2 \amalg \mathbb{Z}/2$.  
 
\begin{example}\label{ex:gsmag}The $G$-CW structure on $G\times G$ is more complicated.  We choose a homotopy $\phi_t: G \times G \rightarrow G \times G$ as in Figure \ref{fig:gtg}, with $t \in [0,1]$, $\phi_0=\mathrm{Id},$ and $\phi_1(G \times G)$ shown.  The arrows depict the action of $S^1$.  On the left, the diagonal lines show the $G$-action before homotopy.  For example, as shown in Figure \ref{fig:gtg}, the homotopy $\phi$ takes the line $\ell=\{(e^{i\theta} \times e^{i\theta} \mid \, \theta \in (0,\pi) \}$, the first half of the diagonal in $S^1 \times S^1$, to the sum of CW cells:
   \[s \otimes 1 + j^2 \otimes s. \]The arrows on the right show the $G$-action on $G \times G$ given by 
\begin{equation}\label{eq:gconj}g(g_1 \times g_2) = \phi_1(g \phi_1^{-1}(g_1 \times g_2)).\end{equation} The action (\ref{eq:gconj}) is clearly cellular with respect to the product CW structure of $G \times G$.  \begin{figure}
  \input{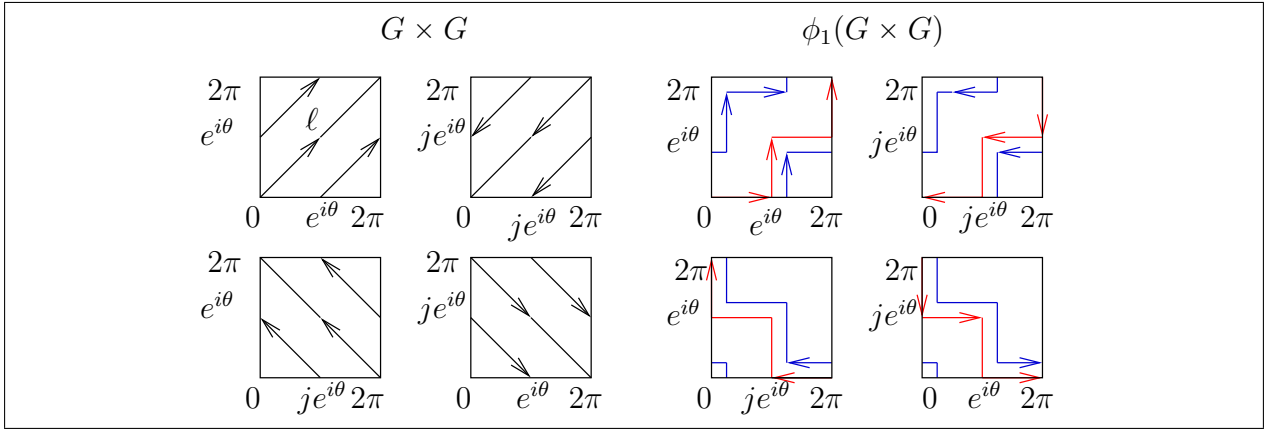}
  \caption{Homotopy of the action of $G$ on $G \times G$.}
  \label{fig:gtg}
\end{figure} 

   Since the $G$-equivariant homotopy type of the $G$-CW complex of $G \times G$ is invariant under homotopy, we let $G$ act on $G \times G$ by (\ref{eq:gconj}).  With this action, $G$ acts cellularly on the product CW decomposition of $G \times G$.  Then $G \times G$ admits the product CW-decomposition as a $G$-CW decomposition.  
\end{example}
Now, let $X_1$ and $X_2$ be spaces of type SWF.  We then give $X_1 \wedge X_2$ a $G$-CW decomposition proceeding cell-by-cell.  That is, for $G$-cells $e_1 \subseteq X_1, e_2 \subseteq X_2$ we give $e_1 \wedge e_2$ the appropriate $G$-CW decomposition as constructed above.  This is possible because the cells $e_i$ are necessarily of the form: $D^k, D^k \times \mathbb{Z}/2$, or $D^k \times G$.  In particular, the construction of a $G$-CW structure on $X_1 \wedge X_2$ gives us a $G$-CW structure for suspensions.  For $V$ a finite-dimensional $G$-representation which is a direct sum of copies of $\mathbb{R},\tilde{\mathbb{R}},$ and $\mathbb{H}$, we have $\Sigma^VX=V^+ \wedge X$, and so give $\Sigma^VX$ the smash product $G$-CW decomposition.  

Finally, we construct a CW structure for the $G$-smash product $X_1 \wedge_G X_2= (X_1 \wedge X_2)/G$.  More generally, we describe a CW structure for the quotient $W/G$ for $W$ a $G$-CW complex.  Indeed, let $W=\bigcup e_i$ a $G$-CW complex, where $e_i=D^{k(i)} \times G/H_i$ are equivariant $G$-cells for some function $k$, and $H_i \subseteq G$ are subgroups.  Then $W/G$ admits a CW decomposition given by $W= \bigcup e_i/G=\bigcup D^{k(i)}$.  
\subsection{Modules from $G$-CW decompositions.}\label{subsec:modfromg}
Throughout this section $X$ will denote a space of type SWF.  Here we will show that the CW chain complex of $X$ inherits a module structure from the action of $G$, and we will define chain local equivalence.

From the group structure of $G$, $C^{CW}_{*}(G)=\G$ acquires an algebra structure.  The algebra structure is determined by the multiplication map $G \times G\rightarrow G$ from which we obtain a map $C^{CW}_*(G)\otimes_\f C^{CW}_*(G) \rightarrow C^{CW}_*(G)$.  Here, we have used the product $G$-CW decomposition of $G\times G$, from Example \ref{ex:gsmag}, for which the multiplication map is cellular.  One may then directly compute the algebra structure of $\G$.  Explicitly, 
\[C^{CW}_{*}(G) \cong \mathbb{F}[s,j]/(sj=j^3s,s^2=0,j^4=1).\]

For any $G$-compatible decomposition of $X$, the relative CW chain complex $C^{CW}_*(X,\p)$ inherits the structure of a $\G$-chain complex.  That is, $C^{CW}_*(X,\p)$ is a module over $\G$, such that, for $z \in C^{CW}_*(X,\p)$, and $a\in\G$, $\partial(az)=a\partial(z)+\partial(a)z$.  The $\G$-module structure comes from the maps:
\begin{equation}\label{eq:modstrucprf}C^{CW}_*(G) \otimes C^{CW}_*(X) \rightarrow C^{CW}_*(X),\end{equation}
coming from the $G$-action.  The condition that the CW decomposition of $X$ be $G$-compatible ensures that the map $G \times X \rightarrow X$ is cellular, so that (\ref{eq:modstrucprf}) is well-defined.

  We find the module structure for the Examples \ref{ex:rsusp}-\ref{ex:gsmag} of Section \ref{subsec:23}.  
\begin{example}\label{ex:rsusp2}
Consider the $\G$-chain complex structure of $C^{CW}_*((\tilde{\mathbb{R}}^s)^+,\p)$ from Example \ref{ex:rsusp}.  Identifying $c_i$ with its image in $C^{CW}_*((\tilde{\mathbb{R}}^s)^+,\p)$, we have $\partial (c_0)=0, \partial (c_1)=c_0,$ and $\partial(c_i)=(1+j)c_{i-1}$ for $i\geq 2$.  One may easily check that the action of $\G$ is given by the relations $jc_0=c_0$, $j^2c_i=c_i$ for $i\geq 1$, and $sc_i=0$ for all $i$ (in particular, the CW cells of $((\tilde{\mathbb{R}}^s)^+,\p)$ are precisely $c_0,c_1,...c_s$ and $jc_1,...,jc_s$, and all of these are distinct).  
\end{example}
\begin{example}\label{ex:hsusp2}
We also find the $\G$-chain complex structure of $C^{CW}_*(\mathbb{H}^+,\p)$ from Example \ref{ex:hsusp}.  One may check that the differentials are given by 
\begin{equation}\partial(r_0)=0,\; \partial y_1 = r_0, \;\partial y_2 = (1+j)y_1, \; \mathrm{and} \;\partial y_3 = sy_1+(1+j)y_2.\end{equation}  The $\G$-action on the fixed-point set, $r_0$, is necessarily trivial.  However, elsewhere the $G$-action on $(\mathbb{H}^+,\p)$ is free, and so the submodule (not a subcomplex, however) of $C^{CW}_*(\mathbb{H}^+,\p)$ generated by $y_1,y_2,y_3$ is $\G$-free, specifying the $\G$-module structure of $C^{CW}_*(\mathbb{H}^+,\p)$.
\end{example}
\begin{example}\label{ex:gtimesg2}
The CW chain complex of the usual product CW structure on $G \times G$ becomes a $\G$-module via:
  \[C^{CW}_*(G \times G)= C^{CW}_*(G) \otimes_\f C^{CW}_*(G),\]
   where the action of $\G$ is given by: 
\begin{equation}\label{eq:gac} s(a \otimes b) =(sa \otimes b) + (j^2a \otimes sb), \end{equation}
\[j(a \otimes b)=(ja \otimes jb).\]   
The differentials are induced by those of the usual product CW structure on $G \times G$.
\end{example}  
For $X_1 \wedge X_2$ with the $G$-CW decomposition described in Section \ref{subsec:23}, we have:
\begin{equation}\label{eq:smashform} C^{CW}_*(X_1 \wedge X_2, \p)=C^{CW}_*(X_1,\p) \otimes_\f C^{CW}_*(X_2,\p),\end{equation}
as $\G$-chain complexes.

  Furthermore the CW chain complex for the $G$-smash product $X_1 \wedge_G X_2 $ is given by:
  \begin{equation}\label{eq:smashdec2} C^{CW}_*(X_1 \wedge_G X_2,\p) \simeq C^{CW}_*(X_1\wedge X_2,\p) /\G. \end{equation}
We will write elements of the latter as $x_1 \otimes_\G x_2$.  Note that Borel homology $\tilde{H}^G_*(X)$ is calculated using a $G$-smash product, and so may be computed from the following chain complex:
\begin{equation}\label{eq:defofborel}
\tilde{H}^G_*(X)=H(C^{CW}_*(EG) \otimes_\G C^{CW}_*(X,\p),\partial).
\end{equation}
In (\ref{eq:defofborel}), we choose some fixed $G$-CW decomposition of $EG$ to define $C^{CW}_*(EG)$.  Equation (\ref{eq:defofborel}) shows that we may define $G$-Borel homology for $\G$-chain complexes $Z$ by:
\begin{equation}\label{eq:chainborel} H^G_*(Z)=H(C^{CW}_*(EG) \otimes_\G Z,\partial),\end{equation}
and similarly for $S^1$-Borel homology:
\begin{equation}\label{eq:chains1borel}
H^{S^1}_*(Z)=H(C^{CW}_*(EG) \otimes_{C^{CW}_*(S^1)} Z,\partial),
\end{equation}  
where $C^{CW}_*(S^1)$ is viewed as a subcomplex of $\G$.  

For $R$ a ring, $M$ an $R$-module with a fixed basis $\{B_i\}$, we say that an element $m \in M$ \emph{contains} $b \in \{ B_i \}$ if when $m$ is written in the basis $\{B_i\}$ it contains a nontrivial $b$ term. 

\begin{defn}\label{def:typswf}
We call a $\G$-chain complex $Z$ \emph{a chain complex of type SWF} at level $s$ if $Z$ is isomorphic to a chain complex (perhaps with a grading shift) generated by 
\begin{equation}\label{eq:gmoddef}  \{ c_0,c_1,c_2,...,c_s \} \cup \bigcup_{i \in I} \{x_i\},\end{equation}  
subject to the following conditions.  The element $c_i$ is of degree $i$, and $I$ is some finite index set. The only relations are $j^2c_i=c_i,sc_i=0,jc_0=c_0$.  The differentials are given by $\partial{c_1}=c_0$, and $\partial{c_i}=(1+j)c_{i-1}$ for $2\leq i \leq s-1$.  Further, $\partial(c_s)$ contains $(1+j)c_{s-1}$.  The submodule generated by $\{ x_i \}_{i \in I}$ is free under the action of $\G$.  We call the submodule generated by $\{c_i\}_i$ the \emph{fixed-point set} of $Z$.
\end{defn}

Chain complexes of type SWF are to be thought of as reduced $G$-CW chain complexes of spaces of type SWF.  Indeed, all spaces $X$ of type SWF have a $G$-CW decomposition with reduced $G$-CW chain complex a complex of type SWF.  To see this, we first decompose $X^{S^1}\simeq (\tilde{\mathbb{R}}^s)^+$ using the CW decomposition of Example \ref{ex:rsusp} for $(\tilde{\mathbb{R}}^s)^+$.  We note next that $X^{S^1}$ is a $G$-CW subcomplex of $X$, and all cells of $(X,X^{S^1})$ are free $G$-cells, since $X$ is a space of type SWF.  Label these cells $\{x_i\}$ for $i$ in some index set.  Then it is clear that the corresponding CW chain complex is as in Definition \ref{def:typswf}.    

To introduce chain local equivalence, we will consider the CW chain complexes coming from suspensions.  For a module $M$ and a subset $S \subseteq M$, we let $\langle S \rangle \subseteq M$ denote the subset generated by $S$. 

Note that, by Example \ref{ex:rsusp2} and the CW decomposition constructed in Section \ref{subsec:23} for suspensions, for $X$ a complex of type SWF: 
\begin{equation}\label{eq:rtilsuspcw}\Sigma^{\tilde{\mathbb{R}}}C^{CW}_*(X,\p)=\langle c_0,c_1 \rangle \otimes_\f C^{CW}_*(X,\p), \end{equation}
with relations $\partial c_1 = c_0$, $j^2c_1=c_1, jc_0=c_0,sc_0=sc_1=0$.  The differential on the right is  given by $\partial (a \otimes b)=\partial(a)\otimes b+ a\otimes\partial(b)$.   
Similarly, using Example \ref{ex:hsusp2}:
\[C^{CW}_*(\Sigma^\mathbb{H}X,\p)=\langle r_0,y_1,y_2,y_3 \rangle \otimes_\mathbb{F} C^{CW}_*(X,\p),\]
with the product differential on the right, and differentials for the $y_i$ given as in Example \ref{ex:hsusp2}.    

For $V=\mathbb{H}$, $\tilde{\mathbb{R}}$, or $\mathbb{R}$, we define $\Sigma^VZ$ so that if $Z=C^{CW}_*(X,\p)$ for $X$ a space of type SWF, then $\Sigma^VZ = C^{CW}_*(\Sigma^VX,\p)$.   We set:
\begin{equation}\label{eq:vsusp} \Sigma^{V}Z = C^{CW}_*(V^+,\p)\otimes_\f Z, \end{equation}
with $\G$-action given by:
\begin{equation}\label{eq:gactformula2}
s(a \otimes b)=(sa \otimes b)+(j^2a \otimes sb), \end{equation}
\[j(a \otimes b)=ja \otimes jb,\]
for $V=\mathbb{H},\tilde{\mathbb{R}}$, or $\mathbb{R}$.  The chain complexes $C^{CW}_*(\mathbb{H}^+,\p)$ and $C^{CW}_*(\tilde{\mathbb{R}}^+,\p)$ were given in Examples \ref{ex:rsusp2} and \ref{ex:hsusp2}, respectively.  Also, $C^{CW}_*(\mathbb{R}^+,\p)=\langle c_1 \rangle $, where $jc_1=c_1,sc_1=0$, and $\mathrm{deg} \; c_1=1$.  Hence, for example:
\begin{equation}\label{eq:rsusp} \Sigma^{\mathbb{R}}Z= Z[-1]. \end{equation}

\begin{lem}\label{lem:cwsus}
Let $V=\mathbb{H}$, $\tilde{\mathbb{R}}$, or $\mathbb{R}$.  If $Z=C^{CW}_*(X,\p)$ for $X$ a space of type SWF, then $\Sigma^VZ = C^{CW}_*(\Sigma^VX,\p)$. 
\end{lem}
\begin{proof}
This follows from the CW chain complex structure given for suspensions in Section \ref{subsec:23}, and (\ref{eq:smashform}).    
\end{proof}

For $V=\mathbb{H}^i\oplus \tilde{\mathbb{R}}^j\oplus \mathbb{R}^k$ for some constants $i,j,k$, we define $\Sigma^V Z$ by:
\begin{equation}\label{eq:iter}
\Sigma^V Z = (\Sigma^{\mathbb{H}})^i (\Sigma^{\tilde{\mathbb{R}}})^j (\Sigma{\mathbb{R}})^k Z.
\end{equation} 
where $(\Sigma^{\mathbb{H}})^i$ denotes applying $\Sigma^{\mathbb{H}}$ $i$ times, and so for $\tilde{\mathbb{R}}$ and $\mathbb{R}$.  It is then clear that:
\begin{equation}\label{eq:repcommu}
\Sigma^V\Sigma^W Z \cong \Sigma^W \Sigma^V Z,
\end{equation}
for two $G$-representations $V,W$.  
\begin{defn}\label{def:compst}
Let $Z_i$ be chain complexes of type SWF, $m_i \in \mathbb{Z}, n_i \in \mathbb{Q}$, for $i=1,2$.  We call $(Z_1,m_1,n_1)$ and $(Z_2,m_2,n_2)$ \emph{chain stably equivalent} if $n_1-n_2 \in \mathbb{Z}$ and there exist $M \in \mathbb{Z}, N \in \mathbb{Q}$ and maps
\begin{equation}\label{eq:compst}
\Sigma^{(N-n_1)\mathbb{H}}\Sigma^{(M-m_1)\tilde{\mathbb{R}}} Z_1 \rightarrow 
\Sigma^{(N-n_2)\mathbb{H}}\Sigma^{(M-m_2)\tilde{\mathbb{R}}} Z_2
\end{equation}
\begin{equation}\label{eq:compst2}
\Sigma^{(N-n_1)\mathbb{H}}\Sigma^{(M-m_1)\tilde{\mathbb{R}}} Z_1 \leftarrow 
\Sigma^{(N-n_2)\mathbb{H}}\Sigma^{(M-m_2)\tilde{\mathbb{R}}} Z_2,
\end{equation}
which are chain homotopy equivalences.    
\end{defn}
\begin{rmk} We do not consider suspensions by $\mathbb{R}$, unlike in the case of stable equivalence for spaces, since by (\ref{eq:rsusp}), no new maps are obtained by suspending by $\mathbb{R}$.  
\end{rmk}
Chain stable equivalence is an equivalence relation, and we denote the set of chain stable equivalence classes by $\ce$.  

\begin{lem}\label{lem:assocchains}
Associated to an element $(X,m,n) \in \mathfrak{C}$ there is a well-defined element $(C^{CW}_*(X,\p),m,n)\in \mathfrak{CE}$.
\end{lem}
\begin{proof}
Say that $[(X_1,m_1,n_1)]=[(X_2,m_2,n_2)] \in \mathfrak{C}$ with $G$-CW decompositions $C_i$ of $X_i$.  We will show that 
\begin{equation}\label{eq:cwident}[(C^{CW}_*(X_1,\p),m_1,n_1)]=[(C^{CW}_*(X_2,\p),m_2,n_2)] \in \mathfrak{CE},\end{equation}
where $C^{CW}_*(X_i,\p)$ is the $\G$-chain complex associated to the $G$-CW decomposition $C_i$ of $X_i$.  (In the case $X_1\simeq X_2$, and $m_1=m_2,n_1=n_2$, we are showing that the corresponding element in $\mathfrak{CE}$ does not depend on the choice of $G$-CW decomposition).
By hypothesis, there are homotopy equivalences $f$ and $g$:
\[ f: \Sigma^{(N-n_1)\mathbb{H}}\Sigma^{(M-m_1)\tilde{\mathbb{R}}}X_1 \rightarrow \Sigma^{(N-n_2)\mathbb{H}}\Sigma^{(M-m_2)\tilde{\mathbb{R}}}X_2,\]
\[ g: \Sigma^{(N-n_2)\mathbb{H}}\Sigma^{(M-m_2)\tilde{\mathbb{R}}}X_2 \rightarrow \Sigma^{(N-n_1)\mathbb{H}}\Sigma^{(M-m_1)\tilde{\mathbb{R}}}X_1.\]
By the Equivariant Cellular Approximation Theorem (see \cite{Waner}), we may homotope $f$ and $g$ to cellular maps (where the cell structures of suspensions are given as in (\ref{eq:vsusp})):
\[ f^{CW}: \Sigma^{(N-n_1)\mathbb{H}}\Sigma^{(M-m_1)\tilde{\mathbb{R}}}C_1 \rightarrow \Sigma^{(N-n_2)\mathbb{H}}\Sigma^{(M-m_2)\tilde{\mathbb{R}}}C_2,\]
\[ g^{CW}: \Sigma^{(N-n_2)\mathbb{H}}\Sigma^{(M-m_2)\tilde{\mathbb{R}}}C_2 \rightarrow \Sigma^{(N-n_1)\mathbb{H}}\Sigma^{(M-m_1)\tilde{\mathbb{R}}}C_1.\]
Since $f$ and $g$ are homotopy equivalences, so are $f^{CW}$ and $g^{CW}$.  The cellular maps $f^{CW}$ and $g^{CW}$ induce maps $f_*$ and $g_*$:
\[ f_*: \Sigma^{(N-n_1)\mathbb{H}}\Sigma^{(M-m_1)\tilde{\mathbb{R}}}C^{CW}_*(X_1,\p) \rightarrow \Sigma^{(N-n_2)\mathbb{H}}\Sigma^{(M-m_2)\tilde{\mathbb{R}}}C^{CW}_*(X_2,\p),\]
\[ g_*: \Sigma^{(N-n_2)\mathbb{H}}\Sigma^{(M-m_2)\tilde{\mathbb{R}}}C^{CW}_*(X_2,\p) \rightarrow \Sigma^{(N-n_1)\mathbb{H}}\Sigma^{(M-m_1)\tilde{\mathbb{R}}}C^{CW}_*(X_1,\p).\]
These are chain homotopy equivalences, by construction, and so we obtain (\ref{eq:cwident}), as needed.  
\end{proof}

\begin{defn}\label{def:compstl}
Let $Z_i$ be chain complexes of type SWF, $m_i \in \mathbb{Z}, n_i \in \mathbb{Q}$, for $i=1,2$.  We call $(Z_1,m_1,n_1)$ and $(Z_2,m_2,n_2)$ \emph{chain locally equivalent} if there exist $M \in \mathbb{Z}, N \in \mathbb{Q}$ and maps
\begin{equation}\label{eq:compstl}
\Sigma^{(N-n_1)\mathbb{H}}\Sigma^{(M-m_1)\tilde{\mathbb{R}}} Z_1 \rightarrow 
\Sigma^{(N-n_2)\mathbb{H}}\Sigma^{(M-m_2)\tilde{\mathbb{R}}} Z_2
\end{equation}
\begin{equation}\label{eq:compstl2}
\Sigma^{(N-n_1)\mathbb{H}}\Sigma^{(M-m_1)\tilde{\mathbb{R}}} Z_1 \leftarrow 
\Sigma^{(N-n_2)\mathbb{H}}\Sigma^{(M-m_2)\tilde{\mathbb{R}}} Z_2,
\end{equation}
which are chain homotopy equivalences on the fixed-point sets.  
\end{defn}
We call a map as in (\ref{eq:compstl}) or (\ref{eq:compstl2}) a \emph{chain local equivalence}.  Elements $Z_1,Z_2 \in \ce$ are chain locally equivalent if and only if there are chain local equivalences $Z_1 \rightarrow Z_2$ and $Z_2 \rightarrow Z_1$.  There are pairs of chain complexes with a chain local equivalence in one direction but not the other; these are not chain locally equivalent complexes.    
Chain local equivalence is an equivalence relation, and we write $(Z_1,m_1,n_1) \equiv_{cl} (Z_2,m_2,n_2)$ if $(Z_1,m_1,n_1)$ is chain locally equivalent to $(Z_2,m_2,n_2)$.  Denote the set of chain local equivalence classes by $\mathfrak{CLE}$.  The set $\mathfrak{CLE}$ is naturally an abelian group, with multiplication given by the tensor product (over $\f$, with $\G$-action as above).  As in Lemma \ref{lem:assocchains}, associated to an element $(X,m,n) \in \mathfrak{LE}$ there is a well-defined element $(C^{CW}_*(X,\p),m,n) \in \mathfrak{CLE}$.  

It follows from the definitions that the $\G$-module $C^{CW}_*(X,\p)$ determines $\tilde{H}^G_*(X)$ for $X$ a space of type SWF.  

\subsection{Calculating the chain local equivalence class}
In this section we will obtain a description of $\CEL$ more amenable to calculations than the definition.  Throughout this section $Z$ will denote a chain complex of type SWF.  

In order to determine if $(Z_1,m_1,n_1)$ and $(Z_2,m_2,n_2)$ are chain locally equivalent without checking all possible $M,N$, we will use the forthcoming Lemma \ref{lem:stabeqch}.  To prove Lemma \ref{lem:stabeqch} we will first need Lemma \ref{lem:eqfreu}, a result on chain homotopy classes of maps between fixed-point sets.  For two $\G$-chain complexes $Z'_1$ and $Z'_2$, let $[Z'_1,Z'_2]$ denote the set of chain homotopy classes of maps from $Z'_1$ to $Z'_2$.  We have an algebraic anologue of the Equivariant Freudenthal Suspension Theorem (Theorem 3.3 of \cite{Adams}), as follows.  We recall that for $Z$ a chain complex of type SWF at level $s$, the fixed point set $R \subset Z$ is isomorphic, as a $\G$-chain complex, to 
\begin{equation}\label{eq:rsuspcomplex}C^{CW}_*(\tilde{\mathbb{R}}^s,\p)\cong \langle c_0,...,c_s\rangle , \end{equation}
with relations  $jc_0=sc_0=0$ and, for $i>0$, $j^2c_i=c_i$, while $sc_i=0$.  The differentials in (\ref{eq:rsuspcomplex}) are given by $\partial(c_i)=(1+j)c_{i-1}$ for $2 \leq i \leq s$, and $\partial(c_1)=c_0$, $\partial(c_0)=0$.  
\begin{lem}\label{lem:eqfreu} 
Let $R_1 \cong R_2 \cong C^{CW}_*(\tilde{\mathbb{R}}^s,\p)$, where $\cong$ denotes isomorphism of $\G$-chain complexes.  Then the map
 \begin{equation}\label{eq:equivfre}[R_1,R_2] \rightarrow [\Sigma^\mathbb{H} R_1 , \Sigma^\mathbb{H} R_2],\end{equation}
obtained by suspension by $\mathbb{H}$ is an isomorphism.
\end{lem}
\begin{proof}
To show that the map in (\ref{eq:equivfre}) is an isomorphism, we consider the commutative diagram:
\begin{equation}\label{eq:suspdiag} \begin{CD}
[\Sigma^{\mathbb{H}}C^{CW}_*(S^0,\p),\Sigma^{\mathbb{H}}C^{CW}_*(S^0,\p)] @>\Sigma^{\tilde{\mathbb{R}}^s}>> [\Sigma^\mathbb{H}R_1,\Sigma^\mathbb{H}R_2] \\
@A\Sigma^\mathbb{H}AA @A\Sigma^{\mathbb{H}}AA \\
[C^{CW}_*(S^0,\p),C^{CW}_*(S^0,\p)] @>\Sigma^{\tilde{\mathbb{R}}^s}>> [R_1,R_2]
\end{CD} \end{equation}
We have used the isomorphisms $R_1 \cong R_2 \cong \Sigma^{\tilde{\mathbb{R}}^s}C^{CW}_*(S^0,\p)$ in writing the right column.  In (\ref{eq:suspdiag}), the composition is precisely \[\Sigma^{\mathbb{H}\oplus \tilde{\mathbb{R}}^s}:[C^{CW}_*(S^0,\p),C^{CW}_*(S^0,\p)] \rightarrow [\Sigma^{\mathbb{H}}R_1,\Sigma^\mathbb{H}R_2].\]  In writing (\ref{eq:suspdiag}), we have also used that for $Z$ a $\G$-chain complex, and $V, W$ representations of $G$, $\Sigma^V\Sigma^WZ \cong \Sigma^W\Sigma^VZ$, as in (\ref{eq:repcommu}).  We will show that the maps:
\begin{equation}\label{eq:firstmapiso}
\Sigma^{\mathbb{H}}:[C^{CW}_*(S^0,\p),C^{CW}_*(S^0,\p)] \rightarrow [\Sigma^{\mathbb{H}}C^{CW}_*(S^0,\p),\Sigma^{\mathbb{H}}C^{CW}_*(S^0,\p)],
\end{equation}
\begin{equation}\label{eq:thirdmapiso}
\Sigma^{\tilde{\mathbb{R}}^s}:[C^{CW}_*(S^0,\p),C^{CW}_*(S^0,\p)] \rightarrow [R_1,R_2],
\end{equation}
and
\begin{equation}\label{eq:secondmapiso}
\Sigma^{\tilde{\mathbb{R}}^s}: [\Sigma^{\mathbb{H}}C^{CW}_*(S^0,\p),\Sigma^{\mathbb{H}}C^{CW}_*(S^0,\p)] \rightarrow [\Sigma^\mathbb{H}R_1,\Sigma^\mathbb{H}R_2]
\end{equation}
are isomorphisms.  Then, since three of the four maps in (\ref{eq:suspdiag}) are isomorphisms, so is the fourth, which is exactly the map from (\ref{eq:equivfre}), proving the Lemma.

We show that (\ref{eq:firstmapiso}) is an isomorphism.  We use the notation of Example \ref{ex:hsusp2} for $\Sigma^\mathbb{H}C^{CW}_*(S^0,\p)$, writing $c_0$ for the generator of $C^{CW}_*(S^0,\p)$.  Let $f\in [\Sigma^\mathbb{H}C^{CW}_*(S^0,\p),\Sigma^\mathbb{H}C^{CW}_*(S^0,\p)]$.  Then $f(r_0 \otimes c_0)=r_0 \otimes c_0$ or $f(r_0 \otimes c_0)=0$, for degree reasons.  In the former case, $f(y_1 \otimes c_0)=u_1y_1 \otimes c_0$ where $u_1$ is a unit in $\G$.  Indeed, this follows from the requirement:
\[\partial(f(y_1 \otimes c_0))=f(\partial(y_1 \otimes c_0))=f(r_0 \otimes c_0)=r_0\otimes c_0.\]
Similarly, we obtain, perhaps after a homotopy, 
\begin{equation}\label{eq:posthom}
f(y_i \otimes c_0)=u_iy_i \otimes c_0,
\end{equation} where $u_i$ is a unit in $\G$ for $i=1,2,3$.  Indeed, this follows from $H_*(\Sigma^{\mathbb{H}}\langle c_0 \rangle)$ being concentrated in grading $4$.  For instance, $f(y_2\otimes c_0)+u_1y_2\otimes c_0$ must be a cycle in $\Sigma^{\mathbb{H}}\langle c_0 \rangle$, since $\partial(f(y_2\otimes c_0))=f(\partial(y_2\otimes c_0))=(1+j)u_1y_1\otimes c_0$.  Then, by $H_2(\Sigma^{\mathbb{H}}\langle c_0 \rangle)=0$, the element $f(y_2\otimes c_0)+u_1y_2\otimes c_0$ is a boundary, and we may choose a homotopy $h$, vanishing in grading $1$, so that $(\partial h +h \partial)(y_2\otimes c_0)=\partial h(y_2\otimes c_0) =f(y_2\otimes c_0)+u_1y_2\otimes c_0$.  This establishes (\ref{eq:posthom}) for $i=2$, and $i=3$ follows similarly.  

We show that $f \simeq \, \mathrm{Id}_{\Sigma^\mathbb{H}C^{CW}_*(S^0,\p)}$.  We define a homotopy $h:\Sigma^{\mathbb{H}}C^{CW}_*(S^0,\p) \rightarrow \Sigma^{\mathbb{H}}C^{CW}_*(S^0,\p)$ from $f$ to $\mathrm{Id}_{\Sigma^\mathbb{H}C^{CW}_*(S^0,\p)}$, proceeding by defining it in each grading.  First, let $h(r_0 \otimes c_0)=0$.  Then choose $h$ in grading $1$ so that $\partial h(y_1 \otimes c_0)=(1+u_1)y_1 \otimes c_0$, and extend $\G$-linearly.  This is possible, because $(1+u_1)y_1\otimes c_0$ is a boundary in $\Sigma^\mathbb{H}C^{CW}_*(S^0,\p)$ for any unit $u_1$.  An elementary calculation shows that $h$ may be extended over degree $2$ and degree $3$. 
In the case that $f(r_0 \otimes c_0)=0$, an explicit homotopy as above shows that $f$ is homotopic to the zero map.  This shows that (\ref{eq:firstmapiso}) is surjective.

To show that (\ref{eq:firstmapiso}) is injective, we note that $[C^{CW}_*(S^0,\p),C^{CW}_*(S^0,\p)]=[\langle c_0\rangle,\langle c_0 \rangle]$ is exactly $\mathbb{Z}/2$ as there is precisely one nontrivial map, $c_0 \rightarrow c_0$.  Then we need only show the identity map has nontrivial suspension.  But $\Sigma^{\mathbb{H}}\mathrm{Id}_{C^{CW}_*(S^0,\p)}=\mathrm{Id}_{\mathbb{H}^+}$, which induces an isomorphism in homology, and so is not null-homotopic.  Then, indeed, we obtain that the map in (\ref{eq:firstmapiso}) is an isomorphism.  

The proof of the isomorphism (\ref{eq:thirdmapiso}) is parallel to the proof of (\ref{eq:firstmapiso}), and is left to the reader. 

We show that the map in (\ref{eq:secondmapiso}) is an isomorphism.  Note that $\Sigma^{\mathbb{H}}C^{CW}_*(S^0,\p)\simeq C^{CW}_*(\mathbb{H}^+,\p)$.  We let $\tilde{\oplus}$ denote a direct sum of $\G$-modules that is not necessarily a direct sum of chain complexes (i.e. there may be differentials between the summands).  Then $C^{CW}_*(\mathbb{H}^+,\p)=\langle c_0 \rangle \tilde{\oplus} F$, for $F$ a $\G$-free submodule, which is not a subcomplex.  We have:
\begin{equation}\label{eq:pregam2}
\Sigma^{\tilde{\mathbb{R}}^s}C^{CW}_*(\mathbb{H}^+,\p) = \Sigma^{\tilde{\mathbb{R}}^s}\langle c_0 \rangle \tilde{\oplus} \Sigma^{\tilde{\mathbb{R}}^s}F.
\end{equation}
However, $\Sigma^{\tilde{\mathbb{R}}^s}F \simeq F[-s]$.  Indeed, we have a map $\gamma:F[-s] \rightarrow \Sigma^{\tilde{\mathbb{R}}^s}F$ defined by $\gamma(x[-s])=C \otimes x$, where $C$ is the fundamental class of $(\tilde{\mathbb{R}}^s)^+$.  If $Z$ is of type SWF at level $0$, then $C=c_0$, while if $Z$ is of type SWF at level $s>0$, we have $C=(1+j)c_s$, where we use the notation from Example \ref{ex:rsusp2}.  Also, $\gamma$ is a chain map, as the reader may verify.  Furthermore, it is clear that $\gamma$ induces a quasi-isomorphism.  We show it is, in fact, a homotopy equivalence.  We construct a homotopy inverse 
\begin{equation}\label{eq:taudef}
\tau: \Sigma^{\tilde{\mathbb{R}}^s}F \rightarrow F[-s],
\end{equation}
so that $\tau(C \otimes x)=x[-s]$ for $x \in F$.  We treat the case $s=1$; for $s>1$ we apply:
\begin{equation}\label{eq:tauinduc}
\Sigma^{\tilde{\mathbb{R}}^{s}}F =(\Sigma^{\tilde{\mathbb{R}}})^sF\simeq F[-s].
\end{equation}
Fix a $\G$-basis $x_i$ of $F$.  Assume we have defined $\tau(c_k \otimes x_i)$ for $k=0,1$, for all $x_i$ such that $\mathrm{deg}\;x_i \leq m-1$, for some $m$.  For generators $x_i$ of degree $m$ we define:
\begin{equation}\label{eq:moretau}
\tau(c_0 \otimes x_i) = \tau(c_1 \otimes \partial x_i),
\end{equation}
\begin{equation*}
\tau(c_1 \otimes x_i) =0,
\end{equation*}
\begin{equation*}
\tau(jc_1 \otimes x_i) =x_i[-1],
\end{equation*}
and extend by linearity.  The reader may check that the extension (\ref{eq:moretau}) is a chain map.  Further, $(1+j)c_1 \otimes x \rightarrow x[-1]$ for all $x \in F$ by definition, so $\tau\gamma=1_{F[-s]}$, where $1_{F[-s]}$ is the identity on $F[-s]$.  

We find a homotopy $H$ from $\gamma \tau$ to $\mathrm{Id}_{\Sigma^{\tilde{\mathbb{R}}}F}$, to show that $\gamma$ is a homotopy equivalence.  Fix generators $x_i$ as in the definition of $\tau$. Define $H$ by $H(c_0 \otimes x_i) =c_1 \otimes x_i$, for all $x_i$, and by $H(c_1 \otimes x)=0=H(jc_1 \otimes x)$ for all $x \in F$, and extending linearly.  We check that $H$ is a chain homotopy between $\gamma \tau$ and $\mathrm{Id}_{\Sigma^{\tilde{\mathbb{R}}}F}$, on the basis of generators $x_i$.  That is, we must check:
\begin{equation}\label{eq:hcheck1}
(\partial H + H\partial) (c_0 \otimes x_i) = \gamma \tau (c_0 \otimes x_i) + c_0 \otimes x_i,
\end{equation}
\begin{equation}\label{eq:hcheck2}
(\partial H + H\partial) (c_1 \otimes x_i)= \gamma \tau (c_1 \otimes x_i ) + c_1 \otimes x_i,
\end{equation}
and
\begin{equation}\label{eq:hcheck3}
(\partial H + H\partial) (jc_1 \otimes x_i)=\gamma \tau (jc_1 \otimes x_i ) + jc_1 \otimes x_i.
\end{equation}
We suppose inductively that (\ref{eq:hcheck1})-(\ref{eq:hcheck3}) are true for all $x_i$ with $\mathrm{deg}\; x_i \leq N$ for some $N$.  The inductive hypothesis is true (vacuously) for $N$ sufficiently small, since $F$ is a bounded-below complex.  Fix $x_i$ of degree $N+1$.  We show that (\ref{eq:hcheck1})-(\ref{eq:hcheck3}) hold for $x_i$.

First, consider:
\begin{equation}\label{eq:hz1}
(\partial H + H \partial )(c_0 \otimes x_i)=\partial (c_1 \otimes x_i) +H(c_0 \otimes \partial x_i),
\end{equation}
where we have used the definition $H(c_0 \otimes x_i ) = c_1 \otimes x_i$.  Also:
\begin{equation}\label{eq:hz2}
(\partial H +H\partial) (c_1 \otimes \partial x_i)= \partial H(c_1 \otimes \partial x_i) + H(c_0 \otimes \partial x_i) = \gamma\tau(c_1 \otimes \partial x_i ) +c_1 \otimes \partial x_i,
\end{equation}
by the inductive hypothesis.  Rearranging (\ref{eq:hz2}), we have:
\begin{equation}\label{eq:hz3}
H(c_0 \otimes \partial x_i)=\gamma \tau (c_1 \otimes \partial x_i ) +c_1 \otimes \partial x_i + \partial H(c_1 \otimes \partial x_i).
\end{equation}
By the definition of $\tau$, we have $\tau (c_1 \otimes \partial x_i) = \tau(c_0 \otimes x_i)$, so, using (\ref{eq:hz1}), we obtain:
\begin{equation}\label{eq:hz4}
(\partial H+H \partial)(c_0\otimes x_i ) = \gamma \tau (c_0 \otimes x_i)+c_0 \otimes x_i + \partial H (c_1 \otimes \partial x_i).
\end{equation}
But $H(c_1 \otimes \partial x_i)=0$ by definition, so:
\begin{equation}\label{eq:hz5}
(\partial H+H \partial)(c_0 \otimes x_i)=\gamma \tau (c_0 \otimes x_i)+c_0 \otimes x_i,
\end{equation}
verifying (\ref{eq:hcheck1}).  

Next, we investigate $(\partial H+ H\partial)(c_1 \otimes x_i)$:
\begin{equation}\label{eq:hz6}
(\partial H+ H \partial)(c_1 \otimes x_i )= \partial H (c_1 \otimes x_i ) + H(c_0 \otimes x_i) + H(c_1\otimes \partial x_i)=H(c_0 \otimes x_i)=c_1 \otimes x_i,
\end{equation}
using $H(c_1 \otimes x_i)=0$ and $H(c_1 \otimes \partial x_i)=0$.  Using $\tau (c_1 \otimes x_i ) =0$, we obtain (\ref{eq:hcheck2}) from (\ref{eq:hz6}).  

We also check $(\partial H+H\partial)(jc_1 \otimes x_i):$
\begin{equation}\label{eq:hz7}
(\partial H+H\partial)(jc_1 \otimes x_i)=\partial H(jc_1 \otimes x_i)+H(c_0 \otimes x_i) + H(jc_1 \otimes \partial x_i)=H(c_0 \otimes x_i)=c_1 \otimes x_i,
\end{equation}
since $H(jc_1 \otimes x_i)=0$ and $H(jc_1 \otimes \partial x_i)=0$ by definition.  Additionally, $\tau(jc_1 \otimes x_i)= x_i[-1]$, and $\gamma(x_i[-1])=(1+j)c_1 \otimes x_i$.  Then $c_1 \otimes x_i =\gamma\tau(jc_1 \otimes x_i)+jc_1 \otimes x_i$, and (\ref{eq:hcheck3}) follows from (\ref{eq:hz7}).  

Then $H$ is a chain homotopy between $\gamma \tau$ and $\mathrm{Id}_{\Sigma^{\tilde{\mathbb{R}}}F}$, as needed, and so $\gamma$ and $\tau$ are homotopy equivalences.

We let $\mbi$ denote the identity map on $\Sigma^{\tilde{\mathbb{R}}^s}\langle c_0\rangle$.  We have a homotopy equivalence: 
\begin{equation}\label{eq:gamma2}
\Sigma^{\tilde{\mathbb{R}}^s}\langle c_0 \rangle \tilde{\oplus} \Sigma^{\tilde{\mathbb{R}}^s}F \xrightarrow{\mbi \tilde{\oplus} \tau} \Sigma^{\tilde{\mathbb{R}}^s}\langle c_0 \rangle \tilde{\oplus} F[-s].
\end{equation}
Further, there is an isomorphism 
\[
[\Sigma^{\tilde{\mathbb{R}}^s}C^{CW}_*(\mathbb{H}^+,\p),\Sigma^{\tilde{\mathbb{R}}^s}C^{CW}_*(\mathbb{H}^+,\p)]\rightarrow [\Sigma^{\tilde{\mathbb{R}}^s}\langle c_0\rangle \tilde{\oplus} F[-s],\Sigma^{\tilde{\mathbb{R}}^s}\langle c_0 \rangle \tilde{\oplus} F[-s]],
\]
given by \[f \rightarrow (\mbi \tilde{\oplus} \tau) f (\mbi \tilde{\oplus} \gamma).\]  Here, the map $(\mbi \tilde{\oplus} \gamma)$ acts by the identity on the first summand, and by $\gamma$ on the second.  

We first prove surjectivity of (\ref{eq:secondmapiso}).  Fix an element $f \in [\Sigma^{\tilde{\mathbb{R}}^s}C^{CW}_*(\mathbb{H}^+,\p),\Sigma^{\tilde{\mathbb{R}}^s}C^{CW}_*(\mathbb{H}^+,\p)]$.  Let $f'=(\mbi \tilde{\oplus} \tau) f (\mbi \tilde{\oplus} \gamma)$.  We find $g: C^{CW}_*(\mathbb{H}^+,\p) \rightarrow C^{CW}_*(\mathbb{H}^+,\p)$ so that $\Sigma^{\tilde{\mathbb{R}}^s}g \simeq f$.  We define $g$ separately on the two summands $C^{CW}_*(S^0,\p)$ and $F$.  

Let $g_1 \in [C^{CW}_*(S^0,\p),C^{CW}_*(S^0,\p)]$ so that $\Sigma^{\tilde{\mathbb{R}}^s}g_1 \simeq f|_{\langle c_0,...,c_s \rangle}$.  Such a $g_1$ exists by (\ref{eq:thirdmapiso}).  Further, note that there is a natural isomorphism $[F,F]=[F[-s],F[-s]]$, and let $g_2 \in [F,F]$ be the element corresponding to $f'|_{F[-s]} \in [F[-s],F[-s]]$.  Define $g: \langle c_0 \rangle \tilde{\oplus} F \rightarrow \langle c_0 \rangle \tilde{\oplus} F$ by \begin{equation}\label{eq:lit}g=g_1 \tilde{\oplus} g_2. \end{equation}  We need to check that $g$ is a chain map.  This is clear for $g|_{F}$ and $g|_{\langle c_0 \rangle}$, by the construction of $g$, so we need only check $x \in F$ so that $\partial x \in \langle c_0 \rangle$.  In this case, we have, where $C$ is the fundamental class of $(\tilde{\mathbb{R}}^s)^+$:
\begin{equation}
(1_{\Sigma^{\tilde{\mathbb{R}}^s}} \otimes g)(C \otimes \partial(x))=f'|_{\Sigma^{\tilde{\mathbb{R}}^s}}(C \otimes \partial(x)),
\end{equation}
by (\ref{eq:lit}).  Moreover,
\begin{equation}\label{eq:f'}
f'|_{\Sigma^{\tilde{\mathbb{R}}^s}}(C \otimes \partial(x))=\partial(f(C \otimes x))=\partial(\gamma f' \tau(C \otimes x)).
\end{equation}
Here, we have used that $x \in F$, and the definition of $f'$.  Additionally, $\gamma f' \tau(C \otimes x) \in F[-s]$ for degree reasons.  Then,   
\begin{align}
\gamma f' \tau(C \otimes x)&=\gamma(f'(x[-s])) \label{eq:complic33}\\ &=\gamma(g_2(x)[-s]) \nonumber \\&=C \otimes g_2(x)\nonumber \\
&=(1_{\Sigma^{\tilde{\mathbb{R}}^s}} \otimes g)(C \otimes x). \nonumber
\end{align}
Composing these equalities, that is, substituting (\ref{eq:complic33}) into (\ref{eq:f'}), we obtain:
\begin{equation}
(1_{\Sigma^{\tilde{\mathbb{R}}^s}} \otimes g)(C \otimes \partial(x))=\partial((1_{\Sigma^{\tilde{\mathbb{R}}^s}} \otimes g)(C \otimes x))=C \otimes \partial(g(x)).
\end{equation}
It follows that $g(\partial(x))=\partial(g(x))$, establishing that $g$ is a chain map.  By construction, \[(\mbi \tilde{\oplus}\tau )\Sigma^{\tilde{\mathbb{R}}^s}g(\mbi \tilde{\oplus}\gamma) \simeq f',\] and so $\Sigma^{\tilde{\mathbb{R}}^s}g \simeq f$, as needed. 

Finally, we check injectivity of (\ref{eq:secondmapiso}).  We have $[\langle c_0\rangle,\langle c_0 \rangle]=[\Sigma^{\mathbb{H}}\langle c_0\rangle,\Sigma^{\mathbb{H}}\langle c_0 \rangle]$ is $\mathbb{Z}/2$, with nontrivial map given by the identity $\mathrm{Id}_{\mathbb{H}^+}$.  We need only show then that the map $\Sigma^{\tilde{\mathbb{R}}^s}\mathrm{Id}_{\mathbb{H}^+}$ is not null-homotopic.  Indeed, it induces a nontrivial map on homology by construction, so is not null-homotopic.  
Then (\ref{eq:secondmapiso}) is an isomorphism, as needed.  
\end{proof}
\begin{rmk}
We have $[C^{CW}_*(S^0,\p),C^{CW}_*(S^0,\p)]=\mathbb{Z}/2$, as remarked in the proof.  Hence Lemma \ref{lem:eqfreu} implies $[\Sigma^{\mathbb{H}}R_1,\Sigma^{\mathbb{H}}R_2]\cong \mathbb{Z}/2$.  
\end{rmk}
\begin{lem}\label{lem:stabeqch}
Let $Z_1$ and $Z_2$ be locally equivalent chain complexes of type SWF.  Let $R_i \subset Z_i$ be the corresponding fixed-point sets.  
Additionally, for all nonzero homogeneous $r \in R_i$, we require $\mathrm{deg} \, r < \mathrm{deg}\, x$ for all nonzero homogeneous 
\[x \in Z_i/R_i,\]
for $i=1,2$.  Then there exist chain maps
\begin{equation}\label{eq:chaincalcequation}
Z_1 \rightarrow  Z_2,
\end{equation}\[
Z_1 \leftarrow Z_2,\]
that are chain homotopy equivalences on the fixed-point sets.  
\end{lem}
\begin{proof}
Let $Z_i(N,M)$ denote $\Sigma^{N\mathbb{H}}\Sigma^{M\tilde{\mathbb{R}}} Z_i$.  
By hypothesis there exist maps which are homotopy equivalences on the fixed-point sets:
\begin{equation}\label{eq:stabeqch}
Z_1(N,M)\rightarrow Z_2(N,M), \end{equation}
\[ 
Z_1(N,M)\leftarrow Z_2(N,M), \]
for $M,N$ sufficiently large.  \\
\begin{claim}\label{clm:vsusp} Let $V=\mathbb{H}$ or $\tilde{\mathbb{R}}$.  Let $Z_i$, for $i=1,2$, as in the Lemma, and take $\phi$ a map which is a chain homotopy equivalence on fixed-point sets:
\[ \phi: \Sigma^{V}
Z_1 \rightarrow 
\Sigma^{V}Z_2.\]
Then $\phi$ is chain homotopic to the suspension of a map, also a chain homotopy equivalence on fixed-point sets:
\[ 
\phi_0: Z_1\rightarrow 
Z_2. \]
\end{claim}
Since $\Sigma^VZ_i$ also satisfy the conditions of the Lemma, it follows from Claim \ref{clm:vsusp} that any map which is a homotopy equivalence on fixed-point sets, for $M_0,N_0 \geq 0$:
\[ \phi: 
Z_1(N_0,M_0) \rightarrow 
Z_2(N_0,M_0)\]
is homotopic to the suspension of a map:
\[ \phi_0: 
Z_1 \rightarrow 
Z_2,\]
which implies the existence of the maps as in (\ref{eq:chaincalcequation}).  

We prove Claim \ref{clm:vsusp} for $V=\mathbb{H}$; the case of $V=\tilde{\mathbb{R}}$ is similar, but easier.  
 
We let $\tilde{\oplus}$ denote a direct sum of $\G$-modules that is not necessarily a direct sum of chain complexes.

Let $F_i$ be the $\G$-free submodule (not necessarily a subcomplex) of $Z_i$ generated by elements $x$ of degree greater than $\mathrm{deg} \; r$ for all $r$ in the fixed-point set $R_i$.  We will also consider $F_i$ as a $\G$-chain complex so that the projection 
\[ Z_i \rightarrow Z_i/R_i \simeq F_i \]
is a map of complexes.  Then we have $Z_i =R_i \tilde{\oplus} F_i$.  For a given local equivalence $\phi:\Sigma^\mathbb{H}Z_1 \rightarrow \Sigma^\mathbb{H}Z_2$, we have the diagram: 
\[ \begin{CD} 
\Sigma^\mathbb{H}(R_1 \tilde{\oplus} F_1) @>\phi>> \Sigma^\mathbb{H}(R_2 \tilde{\oplus} F_2) \\
@VVV   @VVV \\
(\Sigma^\mathbb{H}R_1)\tilde{\oplus}(\Sigma^{\mathbb{H}}F_1) @>>> (\Sigma^\mathbb{H}R_2) \tilde{\oplus}(\Sigma^{\mathbb{H}}F_2) 
\end{CD} \]
However, $\Sigma^{\mathbb{H}}F_i$ is homotopy equivalent to $\Sigma^{\mathbb{R}^4}F_i = F_i[-4]$.  To see this, we use the notation for suspension by $\mathbb{H}$ as in Example \ref{ex:hsusp2} and write $ \gamma:  F_i[-4] \rightarrow \Sigma^{\mathbb{H}}F_i$, where $\gamma(x[-4])=s(1+j)^3y_3 \otimes x$.  The term $s(1+j)^3y_3$ appears as it is the fundamental class of $S^4\simeq \mathbb{H}^+$.  Furthermore, $\gamma$ is a chain map, as the reader may verify.  It is clear that $\gamma$ is a quasi-isomorphism, and it is, in fact, a homotopy equivalence.  There is a homotopy inverse $\tau$, whose construction is analogous to that in (\ref{eq:moretau}), so that $\tau(s(1+j)^3y_3 \otimes x)=x[-4]$.  We obtain a map:
\[ \phi '=(1_{\Sigma^{\mathbb{H}}R_2} \tilde{\oplus}\tau)\phi (1_{\Sigma^{\mathbb{H}}R_1} \tilde{\oplus}\gamma): (\Sigma^{\mathbb{H}}R_1) \tilde{\oplus} (F_1[-4]) \rightarrow (\Sigma^{\mathbb{H}}R_2) \tilde{\oplus} (F_2[-4]).
\]
Here we write $1_{\Sigma^{\mathbb{H}}R_1} \tilde{\oplus} \gamma$ for the map that acts as the identity on the $\Sigma^{\mathbb{H}}R_1$ factor, and as $\gamma$ on the free factor, and similarly for $1_{\Sigma^{\mathbb{H}}R_2} \tilde{\oplus} \tau$.
For degree reasons, $\phi ' $ sends $\Sigma^{\mathbb{H}}R_1 \rightarrow \Sigma^{\mathbb{H}}R_2$ and $F_1[-4] \rightarrow F_2[-4]$.  By (\ref{eq:equivfre}), we have:
\begin{equation}\label{eq:equivfre2}[R_1,R_2] \stackrel{\Sigma^{\mathbb{H}}}{\longrightarrow} [\Sigma^\mathbb{H} R_1 , \Sigma^\mathbb{H} R_2]\end{equation}
is an isomorphism.  Also, $[F_1[-4],F_2[-4]]=[F_1,F_2]$, clearly.  Define $\phi_0|_{R_1}$ by the element of $[R_1,R_2]$ corresponding to $\phi'|_{\Sigma^\mathbb{H}R_1} \in [\Sigma^\mathbb{H}R_1,\Sigma^\mathbb{H}R_2]$.  Similarly, define $\phi_0|_{F_1}$ by the element of $[F_1,F_2]$ corresponding to $\phi'|_{F_1[-4]}\in [F_1[-4],F_2[-4]]$.  Then we have a map, of $\G$-modules:
\[\phi_0: R_1 \oplus F_1 \rightarrow R_2 \oplus F_2.\]
We show that $\phi_0$ is actually a map of $\G$-chain complexes:
\[\phi_0: Z_1 = R_1 \tilde{\oplus} F_1 \rightarrow R_2 \tilde{\oplus} F_2=Z_2,\]
by showing $\phi_0(\partial x)=\partial \phi_0 x $ for $x \in F_1$.  We proceed as in the proof of Lemma \ref{lem:eqfreu}.  Indeed, we need only check those $x\in F_1$ with $\partial x \in R_1$, since $\phi_0|_{F_1}$ is a chain map by construction.  For $x \in F_1$ such that $\partial x \in R_1$, we have \begin{equation}\label{eq:hchain}(1_{\Sigma^\mathbb{H}} \otimes \phi_0)(s(1+j)^3y_3 \otimes\partial x)=\phi'|_{\Sigma^\mathbb{H}R_1}(s(1+j)^3y_3 \otimes \partial x),\end{equation} by the definition of $\phi_0$.  Here we write $1_{\Sigma^\mathbb{H}}\otimes \phi_0$ for the map obtained by suspending $\phi_0$ by $\mathbb{H}$, acting as the identity on $\mathbb{H}$.  Also, $\phi'$ is a chain map, so
\[\phi'|_{\Sigma^\mathbb{H}R_1}(s(1+j)^3y_3 \otimes \partial x)=\phi(s(1+j)^3y_3 \otimes\partial x)=\partial \gamma\phi'\tau(s(1+j)^3y_3 \otimes x).\]
However, $\gamma\phi'(\tau(s(1+j)^3y_3 \otimes x)) \in F_2[-4]$ for degree reasons, so 
\begin{align}\gamma \phi'(\tau(s(1+j)^3y_3 \otimes x))&=\gamma(\phi'(x[-4])) \label{eq:lem3trickeq} \\&=\gamma(\phi '(x)[-4])\nonumber\\&=s(1+j)^3y_3 \otimes \phi '(x) \nonumber \\ & = (1_{\Sigma^{\mathbb{H}}} \otimes \phi_0)(s(1+j)^3y_3 \otimes x) \nonumber.
\end{align}  Composing the equalities, we have:
\[((1_{\Sigma^\mathbb{H}} \otimes \phi_0)(s(1+j)^3y_3 \otimes \partial x)) =  \partial((1_{\Sigma^\mathbb{H}} \otimes \phi_0) (s(1+j)^3y_3 \otimes x))=s(1+j)^3y_3 \otimes \partial \phi_0(x),\]
which gives:
\[\phi_0(\partial x) = \partial \phi_0 (x).\]
This shows that $\phi_0$ is a chain map.  By construction, $(1_{\Sigma^{\mathbb{H}}R_2} \tilde{\oplus}\tau) \Sigma^\mathbb{H} \phi_0 \simeq \phi'(1_{\Sigma^{\mathbb{H}}R_1} \tilde{\oplus}\tau)$.  Finally, $(1_{\Sigma^{\mathbb{H}}R_2} \tilde{\oplus}\tau) \phi (1_{\Sigma^{\mathbb{H}}R_1} \tilde{\oplus}\gamma)=\phi'$, so $\phi=(1_{\Sigma^{\mathbb{H}}R_2} \tilde{\oplus}\gamma) \phi' (1_{\Sigma^{\mathbb{H}}R_1} \tilde{\oplus}\tau)= \Sigma^{\mathbb{H}}\phi_0$, as needed.    
\end{proof}

For $Z$ a chain complex of type SWF, we will let $Z$ also denote the element $(Z,0,0) \in \ce$.  
\begin{defn}Let $R$ be the fixed-point set of $Z$.  If $\mathrm{deg} \, r < \mathrm{deg}\, x$ for all nonzero homogeneous $x \in (Z/R)$ and $r \in R$, we say that the chain complex $Z$ is a \emph{suspensionlike complex}.  
\end{defn} 
\begin{rmk}
Let $X$ be a free, finite $G$-CW complex.  Then the reduced $G$-CW chain complex of $\tilde{\Sigma}X$, the unreduced suspension of $X$, is a suspensionlike chain complex.  Conversely, any suspensionlike chain complex with fixed-point set $R=\langle c_0 \rangle$ may be realized as the $G$-CW chain complex of an unreduced suspension.  Further, any suspensionlike chain complex of type SWF is chain stably equivalent to $C^{CW}_*(X,\p)$ for some space $X$ of type SWF.
\end{rmk}
Lemma \ref{lem:stabeqch} states that if $\Sigma^{(N_0-n_i)\mathbb{H}}\Sigma^{(M_0-m_i)\tilde{\mathbb{R}}}Z_i$ are suspensionlike, then all local (stable) maps between $(Z_1,m_1,n_1)$ and $(Z_2,m_2,n_2)$ are realized as genuine chain maps by suspending the complexes $Z_i$ by $N_0 \mathbb{H} \oplus M_0 \tilde{\mathbb{R}}$.

\subsection{Inessential subcomplexes and connected quotient complexes}\label{subsec:iness}
\begin{defn}\label{def:ines}
Take $Z$ a chain complex of type SWF, and let $R \subset Z$ be the fixed-point set.  For any subcomplex $M \subset Z$ such that $M \cap R = \{0\}$, the projection $Z \rightarrow Z/M$ is a chain homotopy equivalence on $R$.  If there exists a map of chain complexes $Z/M \rightarrow Z$ that is a chain homotopy equivalence on $R$, we say that $M$ is an \emph{inessential subcomplex}.
\end{defn}
If $M$ is inessential, then $Z/M \equiv_{cl} Z$.  
We order inessential subcomplexes by inclusion, $N \leq M$ if $N \subseteq M$.  We show that there is a unique ``minimal" model $Z/N$ locally equivalent to $Z$.   
\begin{lem}\label{lem:inj}
Let $M \subset Z$ be an inessential subcomplex, maximal with respect to inclusion.  Then a map $f:Z/M \rightarrow Z$ which is a homotopy equivalence on fixed-point sets is injective.
\end{lem}
\begin{proof}
Indeed, say $f:Z/M \rightarrow Z$ is a local equivalence with nonzero kernel.  Let $R_1$ denote the fixed-point set of $Z/M$ and $R_2$ denote the fixed-point set of $Z$.  Since $f$ restricts to a homotopy equivalence on the fixed-point sets, $(\mathrm{ker}\, f) \cap R_1 =\{ 0\}$.  Let $\pi: Z \rightarrow Z/M$ be the projection map.  Then $f$ induces a map $Z/(\pi^{-1}(\mathrm{ker}\, f)) \rightarrow Z$, and by $(\mathrm{ker}\, f) \cap R_1 =\{ 0 \}$, this map is a homotopy-equivalence on fixed-point sets.  Additionally, we have $\pi^{-1}(\mathrm{ker}\, f) \cap R_2 =\{ 0 \}$.  Then $\pi^{-1}(\mathrm{ker}\, f)$ is an inessential submodule, and it (strictly) contains $M$, contradicting that $M$ was maximal.  Then $f$ was injective, as needed.  
\end{proof}
\begin{lem}\label{lem:conniso}
Let $Z$ be a chain complex of type SWF and let $M,N \subset Z$ be inessential subcomplexes, with $M$ and $N$ maximal with respect to inclusion.  Then $Z/M \cong Z/N$, where $\cong$ denotes isomorphism of $\G$-chain complexes.
\end{lem}
\begin{proof}
Indeed, if there exist maps $\alpha: Z/M \rightarrow Z,$ and $\beta: Z/N \rightarrow Z$ as above, consider the composition:
\[\phi: Z/N \rightarrow Z \rightarrow Z/M.\]
In particular, we have a map $\alpha \phi: Z/N \rightarrow Z$, which is injective by Lemma \ref{lem:inj}.  It then follows that $\phi$ is injective.  We also have:
\[\psi: Z/M \rightarrow Z \rightarrow Z/N. \]
As before, $\psi$ is injective.  Then, since we have injective chain maps between $Z/N$ and $Z/M$, finite-dimensional $\f$-complexes, the two chain complexes must have the same dimension.  An injective map between complexes of the same dimension is bijective, and, finally, a bijective $\G$-chain complex map is a $\G$-chain complex isomorphism.   
\end{proof}
\begin{lem}\label{lem:homsums}
Let $Z$ be a chain complex of type SWF and $M$ a maximal inessential subcomplex of $Z$.  We have a (noncanonical) decomposition of $Z$:
\begin{equation}\label{eq:connsubdef}Z=(Z/M) \oplus M,\end{equation}
where the isomorphism class of $Z/M$ is an invariant of $Z$, independent of the choice of maximal inessential subcomplex $M \subseteq Z$.  
\end{lem} 
\begin{proof}
Let $\beta: Z/M \rightarrow Z$ be a homotopy equivalence on fixed-point sets.  The map $\beta$ is injective by Lemma \ref{lem:inj}.  Let $\pi$ be the projection $Z \rightarrow Z/M$.  We note that $\beta\pi\beta$ is a map $Z/M \rightarrow Z$ which is a homotopy equivalence on the fixed point set, and so by Lemma \ref{lem:inj}, $\beta\pi\beta$ is injective.  Then $\pi\beta$ is also injective.   

We have a map $\beta\oplus i: (Z/M) \oplus M \rightarrow Z$, where $i$ is the inclusion $i:M \rightarrow Z$.  We check that $\beta \oplus i$ is injective.  Indeed, if $(\beta \oplus i)(z\oplus m)=0$, we have $\beta(z)=m$.  By definition, $\pi(m)=0$, while $\pi\beta$ is injective.  It follows that $m=z=0$, and $\beta \oplus i$ is injective.  Then $Z/M \oplus M \rightarrow Z$ is an injective map of $\f$-vector spaces of the same dimension, and so is an isomorphism (of $\G$-chain complexes).  Since, by Lemma \ref{lem:conniso}, the isomorphism class of $Z/M$ is independent of $M$, we obtain that the isomorphism class of $Z/M$ is a well-defined invariant of $Z$.  
\end{proof}
\begin{defn}\label{def:connplex}
For $Z$ a chain complex of type SWF, let $Z_{\mathrm{conn}}$ denote $Z/\ine$, for $\ine\subseteq Z$ a maximal inessential subcomplex.  We call $Z_{\mathrm{conn}}$ the \emph{connected complex} of $Z$.  
\end{defn}
\begin{thm}\label{thm:decomp}
Let $Z$ be a suspensionlike chain complex of type SWF.  Then for $W$ another suspensionlike complex of type SWF, $Z \equiv_{cl} W$ if and only if $Z_{\mathrm{conn}}\cong W_{\mathrm{conn}}$.  
\end{thm}
\begin{proof}
By Lemma \ref{lem:homsums}, we may write $Z=Z_{\mathrm{conn}} \oplus Z_{\mathrm{iness}}, W=W_{\mathrm{conn}}  \oplus W_{\mathrm{iness}}$, with $Z_{\mathrm{iness}},W_{\mathrm{iness}}$ maximal inessential subcomplexes.  Say we have local equivalences (we need not consider suspensions, by Lemma \ref{lem:stabeqch})
\[ \phi: Z_{\mathrm{conn}}  \oplus Z_{\mathrm{iness}} \rightarrow W_{\mathrm{conn}}  \oplus W_{\mathrm{iness}},\]
\[ \psi:W_{\mathrm{conn}}  \oplus W_{\mathrm{iness}} \rightarrow Z_{\mathrm{conn}}  \oplus Z_{\mathrm{iness}}.\] 
We restrict $\phi$ and $\psi$ to $Z_{\mathrm{conn}} $ and $W_{\mathrm{conn}} $, since it is clear that $Z_{\mathrm{conn}}  \oplus Z_{\mathrm{iness}}$ is chain locally equivalent to $Z_{\mathrm{conn}} $, and likewise for $W_{\mathrm{conn}} $.  Further, we project the image of $\phi$ and $\psi$ to $W_{\mathrm{conn}} $ and $Z_{\mathrm{conn}} $, respectively.  Call the resulting maps $\phi_0$ and $\psi_0$.  If $\phi_0$ had a nontrivial kernel, then we would obtain by composition a local equivalence:
\[\psi_0 \phi_0: Z_{\mathrm{conn}} /\mathrm{ker}\, \phi_0 \rightarrow Z_{\mathrm{conn}} .\]
Composing with the inclusion $\iota: Z_\mathrm{conn} \rightarrow Z$ gives a chain local map $\iota \psi_0\phi_0:Z_{\mathrm{conn}}/\mathrm{ker} \; \phi_0 \rightarrow Z$, so by Lemma \ref{lem:inj}, $\iota \psi_0 \phi_0$ is injective.  Thus, $\phi_0$ is injective. Similarly $\psi_0$ is injective, so we obtain an isomorphism of chain complexes $Z_{\mathrm{conn}}  \cong W_{\mathrm{conn}} $.  Conversely, a homotopy equivalence $Z_{\mathrm{conn}}  \rightarrow W_{\mathrm{conn}} $ yields a local equivalence $Z \rightarrow W$ by the composition \[Z \xrightarrow{\pi} Z_{\mathrm{conn}} \rightarrow W_{\mathrm{conn}} \rightarrow W,\] 
where $\pi:Z \rightarrow Z_{\mathrm{conn}}$ is projection to the first summand.  
\end{proof}

The next Corollary allows us to view the chain local equivalence type of a space of type SWF in $\ce$ instead of $\CEL$.
\begin{cor}\label{cor:reducetohe}
In the language of Theorem \ref{thm:decomp}, there is an injection $B: \CEL \rightarrow \ce$ given by $[(Z,m,n)] \rightarrow [(Z_{\mathrm{conn}},m,n)]$, for $(Z,m,n)$ a representative of the class $[(Z,m,n)]$ with $Z$ suspensionlike.
\end{cor}
\begin{proof} 
Fix $[(Z,m,n)]=[(Z',m',n')] \in \CEL$ with $Z$ and $Z'$ suspensionlike; we will show that $[(Z_{\mathrm{conn}},m,n)]=[(Z'_{\mathrm{conn}},m',n')]$ in $\ce$.  First, we observe that, for $V=\mathbb{H},\tilde{\mathbb{R}}:$ 
\begin{equation}\label{eq:concom} \Sigma^V Z_{\mathrm{conn}} \simeq (\Sigma^V Z)_{\mathrm{conn}}.\end{equation}
 We have, for $M,N$ sufficiently large:
\[\Sigma^{(M-m)\tilde{\mathbb{R}}}\Sigma^{(N-n)\mathbb{H}}Z \leftrightarrows \Sigma^{(M-m')\tilde{\mathbb{R}}}\Sigma^{(N-n')\mathbb{H}}Z'.\]

Here the maps in both directions are local equivalences.   Choosing $M \geq \mathrm{max}\{m,m'\}\,$ and $N \geq \mathrm{max} \, \{n,n'\}$ guarantees that both
\[\Sigma^{(M-m)\tilde{\mathbb{R}}}\Sigma^{(N-n)\mathbb{H}}Z \mathrm{\; and \;} \Sigma^{(M-m')\tilde{\mathbb{R}}}\Sigma^{(N-n')\mathbb{H}}Z'\]
are suspensionlike.  Then, by Theorem \ref{thm:decomp}, we have a homotopy equivalence:

\[(\Sigma^{(M-m)\tilde{\mathbb{R}}}\Sigma^{(N-n)\mathbb{H}}Z)_{\mathrm{conn}} \rightarrow  (\Sigma^{(M-m')\tilde{\mathbb{R}}}\Sigma^{(N-n')\mathbb{H}}Z')_{\mathrm{conn}}.\]

However, by (\ref{eq:concom}), we obtain a homotopy equivalence:

\[\Sigma^{(M-m)\tilde{\mathbb{R}}}\Sigma^{(N-n)\mathbb{H}}(Z_{\mathrm{conn}}) \rightarrow  \Sigma^{(M-m')\tilde{\mathbb{R}}}\Sigma^{(N-n')\mathbb{H}}(Z'_{\mathrm{conn}}).\]

Then $[(Z_\mathrm{conn},m,n)]=[(Z'_{\mathrm{conn}},m',n')] \in \ce$, as needed.  Finally, we show $B$ is injective.  If $(Z_\mathrm{conn},m,n)$ is stably equivalent to $(Z'_\mathrm{conn},m',n')$, then $(Z,m,n)$ and $(Z',m',n')$ are locally equivalent, by Theorem \ref{thm:decomp} and (\ref{eq:concom}).
\end{proof} 
By Corollary \ref{cor:reducetohe}, instead of considering the relation given by chain local equivalence, we need only consider chain homotopy equivalences.  

\begin{defn}
The \emph{connected $S^1$-homology} of $(Z,m,n)\in \ce$, denoted by $\hsc ((Z,m,n))$, for $Z$ a suspensionlike chain complex of type SWF, is the quotient $(H^{S^1}_*(Z)/(H_*^{S^1}(Z^{S^1})+H^{S^1}_*(\ine)))[m+4n]$, where $\ine \subseteq Z$ is a maximal inessential subcomplex.  By Theorem \ref{thm:decomp}, the graded $\f[U]$-module isomorphism class of $\hsc ((Z,m,n))$ is an invariant of the chain local equivalence class of $(Z,m,n)$.  
\end{defn}
\begin{rmk}
We could have instead considered the quotient $(H^{S^1}_*(Z)/H^{S^1}_*(\ine))[m+4n]$, which is isomorphic to $\hsc((Z,m,n))\oplus \bt_d$, for some $d$.  As defined above, $\hsc((Z,m,n))$ has no infinite $\f[U]$-tower, because of the quotient by $H^{S^1}_*(Z^{S^1})$.  
\end{rmk}

\section{$j$-split spaces}\label{sec:jsplit}
In this section we introduce $j$-split spaces of type SWF, and compute their $G$-Borel homology.  We will see in Lemma \ref{lem:crossflows} that the Seiberg-Witten Floer spectra of Seifert spaces are $j$-split.  The computation of this section will then provide the $G$-equivariant Seiberg-Witten Floer homology of Seifert spaces.   

\begin{defn}\label{def:split}
We call a space $X$ of type SWF \emph{$j$-split} if $X/X^{S^1}$ may be written:
\[ X/X^{S^1} \simeq X_+ \vee X_- ,\]
for some $S^1$-space $X_+$, where $\simeq$ denotes $G$-equivariant homotopy equivalence, and $j$ acts on the right-hand side by interchanging the factors (that is, $jX_+=X_-$).  
Similarly, we call a $\G$-chain complex $(Z,\partial)$ of type SWF \emph{$j$-split} if $(1)-(3)$ below are satisfied.
\begin{enumerate}
\item There exists $\fr \in Z$ such that $\langle \fr \rangle$ is the fixed-point set, $Z^{S^1}$, of $Z$.  Furthermore $s\fr=0,j\fr=\fr$.  In particular, $Z$ is of type SWF at level $0$.  \item The fixed-point set $Z^{S^1}$ is a subcomplex of $Z$ (that is, $\partial(\fr)=0$).
\item We have:
\[Z/Z^{S^1}=(Z_+ \oplus jZ_+) ,\]
where $Z_+$ is a $C^{CW}_*(S^1)$ chain complex, and $j$ acts on the right-hand side by interchanging the factors.  
\end{enumerate}
\end{defn}
Recall that $\tilde{\oplus}$ denotes a direct sum of $\G$-modules that is not necessarily a direct sum of chain complexes.  For a $j$-split chain complex $Z$ we may write, referring to $jZ_+$ by $Z_-$:
\[Z=(Z_+ \oplus Z_-) \tilde{\oplus} \langle \fr \rangle.\]
In the above, $Z$ is to be thought of as the reduced CW chain complex of a $G$-space $X$, and $\fr$ is to be thought of as the chain corresponding to the $S^1$-fixed subset of $X$.  The requirement that $Z$ be a chain complex of type SWF at level $0$ will be used in Section \ref{subsec:chain} to calculate the chain local equivalence class of $j$-split chain complexes.

A $j$-split space $X$ with $X^{S^1} \simeq S^0$ admits a CW chain complex which is a $j$-split chain complex.  For $X$ a $j$-split space of type SWF at level $s$, we use the following Lemma to relate the CW chain complex of $X$ to $j$-split complexes.  

\begin{lem}\label{lem:jsplitcomplexes}
Let $X$ be a $j$-split space of type SWF at level $s$.  Then
\[[C^{CW}_*(X,\p)]=[(Z,-s,0)] \in \ce,
\]
for some $j$-split chain complex $Z$.
\end{lem}
\begin{proof}
The chain complex $C^{CW}_*(X,\p)$ may be written 
\begin{equation}\label{eq:jrfspl} C^{CW}_*(X,\p)=R \tilde{\oplus} F,\end{equation}
where $R=C^{CW}_*(X^{S^1},\p)\cong C^{CW}_*((\tilde{\mathbb{R}}^s)^+,\p)$ is a subcomplex and $F$ is a free $\G$-chain complex.  Since $X$ is $j$-split, the decomposition (\ref{eq:jrfspl}) may be chosen so that
\begin{equation}\label{eq:fsplit}
F=F_+ \oplus jF_+,
\end{equation}
where $F_+$ is a $C^{CW}_*(S^1)$-chain complex, and $j$ acts on $F$ by interchanging $F_+$ and $jF_+$.

We first show that we may choose $F$ satisfying (\ref{eq:jrfspl}) and (\ref{eq:fsplit}) and so that, for $x \in F$ homogeneous, 
\begin{equation}\label{eq:fspl2}(\partial x) |_R=0,\end{equation} if $\mathrm{deg} \; x \neq s+1$.  

Indeed, fix some $F$ satisfying (\ref{eq:jrfspl}) and (\ref{eq:fsplit}), and let $\{x_i\}$ be a homogeneous basis for $F$.  Let $F(n)$ denote the $\G$-chain complex generated by $x_i$ of degree less than or equal to $n$.  We define new chain complexes $F'(n)$ so that $R \tilde{\oplus}F'(n)=R\tilde{\oplus}F(n)$, and so that $F'=\bigcup_n F'(n)$ satisfies (\ref{eq:jrfspl})-(\ref{eq:fspl2}).  Let $\pi_n$ denote projection $\pi_n: R \tilde{\oplus}F'(n) \rightarrow R$ onto the first factor.  Set $F'(0)=0$.  Assume we have defined $F'(n)$ for $n \leq N < s$, so that (\ref{eq:fspl2}) holds for all $x \in F'(n)$.  

We define $F'(N+1)$ by defining generators $x_i'$ of $F'(N+1)/F'(N)$ corresponding to the generators $x_i$ of $F(N+1)/F(N)$.  For each $x_i$ of degree $N+1$ so that $\pi_N(\partial x_{i})=0$, let $x_i' =x_i$.  If instead $x_i$ is of degree $N+1$ and $\pi_N(\partial x_{i})\neq 0$, then 
\[\partial(\pi_N(\partial x_i))=\pi_N(\partial^2(x_i))=0.\]
So, $\pi_N(\partial x_i) = (1+j)c_N$, since $(1+j)c_N$ is the only nonzero cycle of $R$ in grading $N$ (or, when $N=0$, $\pi_N(\partial x_i)=c_0$).  However, by assumption, $N<s$, so $\pi_N(\partial x_i)=\partial c_{N+1}$.  Then, we let $x_i'=x_i+c_{N+1}$.  

Let \[F'(N+1)=\langle F'(N), \bigcup_{\{i \mid \mathrm{deg}\; x_i=N+1\}} x_i' \rangle.\]  By construction $R \tilde{\oplus}F'(N+1)=R \tilde{\oplus}F'(N)$, and (\ref{eq:fspl2}) holds for all $x \in F'(N+1)$.  

For $N \geq s$, define $F'(N+1)$ by $F'(N+1)=\langle F'(N), \bigcup_{\{i \mid \mathrm{deg}\; x_i=N+1\}} x_i \rangle$.  

From the construction, it is clear that $F'$ satisfies (\ref{eq:jrfspl})-(\ref{eq:fspl2}), as needed.  

Take $F$ satisfying (\ref{eq:jrfspl})-(\ref{eq:fspl2}).  Consider the $\G$-chain complex $Z=C^{CW}_*(S^0,\p) \tilde{\oplus} F[s]$, where $C^{CW}_*(S^0,\p)=\langle c_0 \rangle$ is a subcomplex.  To define the differentials from $F[s]$ to $C^{CW}_*(S^0,\p)$ in $Z$, we set, for $x[s] \in F[s]$:
\begin{equation}\label{eq:rfdel1}(\partial x[s] )|_{C^{CW}_*(S^0,\p)} = c_0,\end{equation}  if $(\partial x)|_R=(1+j)c_s,$ and
\begin{equation}\label{eq:rfdel2}(\partial x[s] )|_{C^{CW}_*(S^0,\p)} = 0\end{equation} if $(\partial x)|_R=0$.

By the construction of $F$, (\ref{eq:rfdel1}) and (\ref{eq:rfdel2}) determine the differential on $Z$.  

Finally, consider the suspension: \[\Sigma^{\tilde{\mathbb{R}}^s}Z=\Sigma^{\tilde{\mathbb{R}}^s}(C^{CW}_*(S^0,\p))\tilde{\oplus}\Sigma^{\tilde{\mathbb{R}}^s}(F[s])\simeq R \tilde{\oplus}\Sigma^{\tilde{\mathbb{R}}^s}F[s].\]

We note, as in the proof of Lemma \ref{lem:eqfreu}, that $\Sigma^{\tilde{\mathbb{R}}^s}F[s] \simeq F[0]=F.$  Then, there is a homotopy equivalence, constructed exactly as in the proofs of Lemmas \ref{lem:eqfreu} and \ref{lem:stabeqch}:
\begin{equation}\label{eq:ctw}
\Sigma^{\tilde{\mathbb{R}}^s}Z\simeq R \tilde{\oplus}F.
\end{equation} 
It follows that $[(Z,-s,0)] = [C^{CW}_*(X,\p)] \in \ce$, as needed.  
\end{proof}
Note also that any $j$-split chain complex occurs as the CW chain complex of some $j$-split space.  

\begin{rmk} $j$-splitness is not the same as Floer $K_G$-splitness of \cite{ManolescuK}.  
\end{rmk}
\subsection{Calculation of $\tilde{H}^G_*(X)$}\label{subsec:61}
In this section we will compute the $G$-equivariant homology of a $j$-split space in terms of its $S^1$-homology.     

Let $X$ be a $j$-split space of type SWF at level $m$ with $X/X^{S^1}=X_+ \vee X_-$.  
The Puppe sequence 
\[ X^{S^1} \rightarrow X \rightarrow X/X^{S^1} \rightarrow \Sigma X^{S^1}\]
leads to a commutative diagram, where the rows are exact:
\begin{equation}\label{eq:pretrans}
\begin{CD}
EG_+ \wedge_{S^1} X^{S^1} @>>> EG_+ \wedge_{S^1} X @>>> EG_+ \wedge_{S^1} (X_+ \vee X_-) @>>> EG_+ \wedge_{S^1} \Sigma X^{S^1} \\
@VVV @VVV @VVV @VVV \\
EG_+ \wedge_G X^{S^1} @>>> EG_+ \wedge_G X @>>> EG_+ \wedge_G X/X^{S^1} @>>> EG_+ \wedge_G \Sigma X^{S^1}. \\
\end{CD}
\end{equation}
In (\ref{eq:pretrans}) the vertical maps are obtained by taking the quotient by the action of $j \in G$.  The diagram (\ref{eq:pretrans}) itself yields a commutative diagram for Borel homology, where the rows are exact:
\begin{equation}\label{eq:s1gtrans}
\begin{CD}
\tilde{H}^{S^1}_*(X^{S^1}) @>>> \tilde{H}^{S^1}_*(X) @>>> \tilde{H}^{S^1}_*(X_+) \oplus \tilde{H}^{S^1}_*(X_-) @>d_{S^1}[-1]>> \tilde{H}^{S^1}_{*}(\Sigma X^{S^1})\\
@V\phi_1VV @V\phi_2VV @V\phi_3VV @V\Sigma\phi_1VV \\
\tilde{H}^G_*(X^{S^1}) @>\iota_G>> \tilde{H}^G_*(X) @>\pi_G>> \tilde{H}^G_*(X/X^{S^1}) @>d_G[-1]>> \tilde{H}^G_{*}(\Sigma X^{S^1}). \\
\end{CD}
\end{equation}
Specifically, we view (\ref{eq:s1gtrans}) as a diagram of $\f[q,v]/(q^3)$ modules, where $v$ acts on the top row by $U^2$ and $q$ annihilates the top row.  An $\f[U]$-module $M$ viewed as an $\f[q,v]/(q^3)$-module this way is denoted $\re^{\f[U]}_{\f[q,v]/(q^3)}M$.  More precisely, let $\phi: \f[q,v]/(q^3) \rightarrow \f[U]$ be $v \rightarrow U^2,\, q \rightarrow 0$, and let $\re^{\f[U]}_{\f[q,v]/(q^3)}$ be the corresponding restriction functor.  The restriction takes the simple $\f[U]$-module $\bt_d(n)$ to 
\begin{equation}\label{eq:restri}   \re^{\f[U]}_{\f[q,v]/(q^3)}\bt_d(n)=\bv_d(\lfloor \frac{n+1}{2} \rfloor ) \oplus \bv_{d+2}(\lfloor \frac{n}{2} \rfloor). \end{equation}
We define the maps $d_{S^1}:\tilde{H}^{S^1}_*(X_+) \rightarrow \tilde{H}^{S^1}_*(X^{S^1})$ and $d_G: \tilde{H}^G_*(X/X^{S^1}) \rightarrow \tilde{H}^G_*(X^{S^1})$ by shifting by $1$ the degree of the horizontal maps on the right of diagram (\ref{eq:s1gtrans}).  (So that $d_{S^1}$ and $d_G$ are maps of degree $-1$.)  
\begin{fact}\label{fct:phi1iso}The map $\phi_1$ in (\ref{eq:s1gtrans}) is precisely the corestriction map $\co^{S^1}_G$, and is an isomorphism in degrees congruent to $m \, \mathrm{mod} \, 4$, and vanishes otherwise. \end{fact}\begin{proof} This follows from the construction of the $\phi_i$ and the dual of Fact \ref{fct:bs1bg}. \end{proof}
\begin{fact}\label{fct:phi3iso}
\begin{equation}\label{eq:phi3s}\phi_3|_{\tilde{H}^{S^1}_*(X_+)}: \tilde{H}^{S^1}_*(X_+) \rightarrow \tilde{H}^G_*(X/X^{S^1})\end{equation}
is an isomorphism (of $\f[q,v]/(q^3)$-modules).\end{fact}\begin{proof} Since $X$ is $j$-split, both domain and target are isomorphic, as vector spaces, to $H_*(X_+/S^1)$.  The map $\phi_3$ is a bijection and an $\f[q,v]/(q^3)$-module map, and so is an isomorphism.  \end{proof}  In particular, Fact \ref{fct:phi3iso} shows that the $q$-action on $\tilde{H}^G_*(X/X^{S^1})$ is trivial. 
     Since $\phi_3|_{\tilde{H}^{S^1}(X_+)}$ is an isomorphism, we have:
 \begin{equation}\label{eq:quotrel}
 \re^{\f[U]}_{\f[q,v]/(q^3)}\tilde{H}^{S^1}_*(X_+)=\tilde{H}^G_*(X/X^{S^1}).
\end{equation}  
\begin{fact}\label{fct:uqvequiv}  The maps $d_{S^1}$ and $d_G$ are $\f[U]$ and $\f[q,v]/(q^3)$-equivariant, respectively. \end{fact}\begin{proof}
The fact follows since the maps $d_{S^1}$ and $d_G$ are induced on Borel homology, respectively, from $S^1$ and $G$-equivariant maps.
\end{proof} By (\ref{eq:s1gtrans}), 
\begin{equation}\label{eq:dgrho2}
d_G \phi_3 = \phi_1 d_{S^1}.
\end{equation}
\begin{lem}\label{lem:3s1}
We have: 
\begin{equation}\label{eq:s1split} \tilde{H}^{S^1}_*(X)=\mathrm{coker}\; d_{S^1} \oplus \mathrm{ker} \; d_{S^1}.\end{equation}
\end{lem}
\begin{proof}
Using the top row of (\ref{eq:s1gtrans}), we have an exact sequence:
\[ 0 \rightarrow\mathrm{coker}\; d_{S^1} \rightarrow \tilde{H}^{S^1}_*(X) \rightarrow \mathrm{ker} \; d_{S^1}\rightarrow 0,\]
so $\tilde{H}^{S^1}_*(X)$ is an extension of $\mathrm{ker} \; d_{S^1}$ by $\mathrm{coker}\;d_{S^1}$.  Note that $\mathrm{coker}\; d_{S^1}$ is isomorphic to $\bt_d$ for some integer $d$.  A calculation shows $\mathrm{Ext}^1_{\f[U]}(\bt_{d_i}(n_i),\bt_d)=0$ for all $d, d_i,n_i$.  Thus, any extension of $\mathrm{ker} \; d_{S^1}$ by $\mathrm{coker} \; d_{S^1}$ is trivial, and we obtain the Lemma.  
\end{proof}
We also write (\ref{eq:s1split}) as the homology of the complex $\tilde{H}^{S^1}_*(X^{S^1}) \oplus \tilde{H}^{S^1}_*(X/X^{S^1})$ with differential $d_{S^1}$.   
\begin{lem}\label{lem:3g}  
We have:
\begin{equation}\label{eq:firstgdecomp}\tilde{H}^G_*(X)\cong\mathrm{coker}\, d_G \oplus \mathrm{ker}\, d_G.\end{equation}
as $\f[v]$-vector spaces.  The subspace $\mathrm{coker}\, d_G$ is a $\f[q,v]/(q^3)$-submodule, and $q$ acts on $x\in \mathrm{ker}\, d_G$ by $qx=0$ if $x\in \mathrm{Im}\, \phi_2|_{\mathrm{ker} d_{S^1}}$ (using the decomposition of $\tilde{H}^{S^1}_*(X)$ in Lemma \ref{lem:3s1}).  Also, $qx\neq 0 \in \mathrm{coker}\, d_G$ if $x\in \mathrm{ker} \, d_G$ but $x\not\in \mathrm{Im} \, \phi_2|_{\mathrm{ker} \, d_{S^1}}$.  As there is at most one homogeneous element of each degree in $\mathrm{coker} \, d_G$, $qx$ is uniquely specified for all $x \in \mathrm{ker}\, d_G$ in the decomposition (\ref{eq:firstgdecomp}).  \end{lem}
\begin{proof}
As in the proof of Lemma \ref{lem:3s1}, we see that $\tilde{H}^G_*(X)$ is an extension of \[\mathrm{ker}\, d_G \subseteq \mathrm{res}^{\f[U]}_{\f[q,v]/(q^3)}\tilde{H}^{S^1}_*(X_+)\] by $\mathrm{coker} \, d_G=\tilde{H}^G_*(X^{S^1})/(\mathrm{Im}\, d_G)$.  We will first show that the extension is trivial as an $\f[v]$-extension.  
  
We construct $M \subset \tilde{H}^G_*(X)$ a vector space lift of $\mathrm{ker} \; d_G \subset \tilde{H}^G_*(X/X^{S^1})$, so that $\phi_2(\mathrm{ker} \; d_{S^1})\subseteq M$, using the decomposition of $\tilde{H}^{S^1}_*(X)$ in (\ref{eq:s1split}).

Specifically, we define $M$ in each degree $i$ by: \[M_i=\begin{cases}(\phi_2(\mathrm{ker} \; d_{S^1}))_i \; \mathrm{for} \; i \not\equiv 3 +m\;\mathrm{mod}\;4,\\
\tilde{H}^G_i(X) \; \mathrm{for} \; i \equiv 3+m \; \mathrm{mod} \; 4. \end{cases}\]
We next show that $\pi_G|_M:M\rightarrow \mathrm{ker} \, d_G$ is an isomorphism.  

We have $(\mathrm{coker} \; d_G)_i=0$ for $i\equiv 3+m \; \mathrm{mod} \; 4$, since $\tilde{H}^G_*(X^{S^1})\cong H_*(BG)[-m]$, so
\begin{equation}\label{eq:deg3surj}\pi_G:\tilde{H}^G_i(X) \rightarrow (\mathrm{ker} \; d_G)_i\end{equation}
is an isomorphism for all $i \equiv 3+m \; \mathrm{mod} \; 4$.  

We now show that $\pi_G:(\mathrm{Im}\, \phi_2|_{\mathrm{ker} \; d_{S^1}})_i\rightarrow (\mathrm{ker} \; d_G)_i$ is an isomorphism for $i \not\equiv 3+m\; \mathrm{mod}\;4.$  It suffices to show $\mathrm{ker}\, d_G \subseteq \mathrm{Im} \, \phi_3|_{\mathrm{ker}\, d_{S^1}}$ in degrees not congruent to $3+m \,\mathrm{mod} \, 4$.  Indeed, $\phi_3$ is surjective by (\ref{eq:phi3s}).  Furthermore, by Fact \ref{fct:phi1iso}, $\phi_1$ is injective in degrees not congruent to $2+m \, \mathrm{mod} \, 4$.  By (\ref{eq:dgrho2}), if $y \in \mathrm{ker} \, d_G$ with $\mathrm{deg}\, (y) \not\equiv 3+m \; \mathrm{mod} \; 4$, and $y=\phi_3(x)$, for $x \in \tilde{H}^{S^1}_*(X/X^{S^1})$, then $\phi_1(d_{S^1}x)=0$.  By the injectivity of $\phi_1$, we have $d_{S^1}x=0$, and we obtain:
\[y \in \mathrm{Im}\,(\phi_3|_{\mathrm{ker}\, d_{S^1}}).\]  That is, $(\mathrm{Im} \, \phi_3|_{\mathrm{ker} d_{S^1}})_i=(\mathrm{ker}\; d_G)_i$ for $i \not\equiv 3+m \; \mathrm{mod} \; 4$.  Then, $\pi_G(\mathrm{Im}\,\phi_2|_{\mathrm{ker}\;d_{S^1}})_i=(\mathrm{ker} \; d_G)_i$, as needed.

We have then established that $\tilde{H}^G_*(X)= \mathrm{coker}\; d_G \oplus M$ as $\f$-vector spaces.  

We next determine the $\f[q,v]/(q^3)$-action on $M \subset \tilde{H}^G_*(X)$.  Since $\mathrm{ker} \; d_{S^1} \subset \tilde{H}^{S^1}_*(X)$ is an $\f[q,v]/(q^3)$-submodule, so is its image in $\tilde{H}^G_*(X)$.  Then, for $x \in M$ homogeneous of degree not congruent to $3 +m \; \mathrm{mod} \; 4$, we have $qx,vx \in M$.  In fact, $qx=0$, since $q$ acts trivially on $\tilde{H}^{S^1}_*(X)$.  Moreover, for $x \in M$ of degree congruent to $3+m \; \mathrm{mod} \; 4$, $vx \in \tilde{H}^G_*(X)$ is also of degree congruent to $3+m$, and, in particular, we see $vx \in M$.  So we need only determine $qx$ for $x \in M$ with $\mathrm{deg} \; x \equiv 3+m \; \mathrm{mod} \; 4$.  

As in \cite{tomDieck}[III.2] there exists a Gysin sequence: 
\begin{equation}\begin{tikzcd}\tilde{H}^*_G(X) \arrow{r} &\tilde{H}^*_{S^1}(X)\arrow{r} &\tilde{H}^*_G(X)\arrow{r}{q\cup -} &\tilde{H}^{*+1}_G(X)\arrow{r}&\dots,\end{tikzcd}\end{equation}
where $q\cup -$ denotes cup product with $q$.  Dualizing, we obtain an exact sequence:
\begin{equation}\label{eq:gysin}\begin{tikzcd}
\tilde{H}^G_*(X)\arrow{r}{(1+j)\cdot -} &\tilde{H}^{S^1}_*(X) \arrow{r}{\phi_2} &\tilde{H}^{G}_*(X) \arrow{r}{q\cap -}& \tilde{H}^G_{*-1}(X)\arrow{r} & \dots,
\end{tikzcd}\end{equation}
where $(1+j)\cdot -$ denotes the map obtained from multiplication (on the chain level) by $1+j \in \G$, and $q\cap -$ denotes cap product with $q$.   

From (\ref{eq:gysin}), we have that if $x \in M \subset \tilde{H}^G_*(X)$ is not in $\mathrm{Im} \, \phi_2|_{\mathrm{ker} \, d_{S^1}}$, then $qx \neq 0$.  We will show that $qx \in \mathrm{coker}\; d_G$.  

First, we see \begin{equation}\label{eq:intermedj}(1+j)\cdot \mathrm{coker}\; d_G\subset \mathrm{coker}\;d_{S^1}.\end{equation}  Indeed, (\ref{eq:intermedj}) follows from the commutativity of the diagram
\[ \begin{tikzcd}
\tilde{H}^G_*(X) \arrow{r}{(1+j)\cdot} & \tilde{H}^{S^1}_*(X)\\
\tilde{H}^G_*(X^{S^1})\arrow{u} \arrow{r}{(1+j)\cdot} & \tilde{H}^{S^1}_*(X^{S^1})\arrow{u}.
\end{tikzcd}\]Additionally, we see that \[\begin{tikzcd}\mathrm{ker}\, d_G \arrow{r}{(1+j)\cdot -}& \mathrm{ker}\, d_{S^1} \end{tikzcd}\] is injective by the $j$-splitness condition (Definition \ref{def:split}).  Then $\mathrm{ker}\; (1+j) \subset \tilde{H}^{G}_*(X)$ is, in fact, a subset of $\mathrm{coker}\; d_G$.  Thus, if $x \not \in \mathrm{Im} \; \phi_2 |_{\mathrm{ker}\; d_{S^1}}$, $qx$ must be the unique nonzero element in grading $\mathrm{deg} \, x -1$ in $\mathrm{coker}\, d_G$, completing the proof.  \end{proof}   

Our goal will be to relate (\ref{eq:s1split}) and (\ref{eq:firstgdecomp}), relying on (\ref{eq:quotrel}) and (\ref{eq:dgrho2}).  From this relationship we will be able to show that the $S^1$-homology (\ref{eq:s1split}) determines the $G$-homology (\ref{eq:firstgdecomp}).  In Lemmas \ref{lem:deco1} and \ref{lem:deco2} we compute $\tilde{H}^{S^1}_{*}(X)$ from $\tilde{H}^{S^1}_*(X/X^{S^1})$ and $d_{S^1}$.  In Lemmas \ref{lem:deco3}-\ref{lem:deco4}, we show how to compute $\tilde{H}^G_*(X)$ from the same information.  Then in Theorem \ref{thm:onjspaces} we compute $\tilde{H}^G_*(X)$ directly from $\tilde{H}^{S^1}_*(X)$.  

We begin by noting that any finite graded $\f[U]$-module may be written as a direct sum of copies of $\bt_{d_i}(n_i)$, as $\f[U]$ is a principal ideal domain.  In particular, $\tilde{H}^{S^1}_*(X/X^{S^1})$, since it has finite rank as an $\f$-module, is a direct sum of copies of the $\bt_{d_i}(n_i)$. 

\begin{lem}\label{lem:dimoftows}
On $\bt_d(n) \subset \tilde{H}^{S^1}_*(X/X^{S^1})$, the differential $d_{S^1}$ vanishes unless $2n+d \geq 3+m$ and $d \leq m+1$.  
\end{lem}
\begin{proof}  Let $U^{-k}$ denote the unique nonzero element of $\bt_m$ in degree $m+2k$.  Let $x_{d+2n-2}$ be an $\f[U]$-module generator of $\bt_d(n)$, with $\mathrm{deg}\; (x_{d+2n-2}) =d+2n-2$.  Then either $d_{S^1}$ vanishes on $\bt_d(n)$ or $d_{S^1}(x_{d+2n-2})$ is nonzero.  In this latter case, because of the grading, $d_{S^1}(x_{d+2n-2})=U^{-\frac{d+2n-m-3}{2}}$.  If $2n+d<3+m$, then $\bt_d(n)$ has no elements in degree greater than $m$, and so has no nontrivial maps to $\bt_m$.  Similarly, for $d>m+1$, $d_{S^1}(\bt_d(n))=0$.  Indeed, if $d_{S^1}(\bt_d(n)) \neq 0$, then \[d_{S^1}x_{d+2n-2}=U^{-\frac{d+2n-m-3}{2}}.\]  Then, by Fact \ref{fct:uqvequiv}, $d_{S^1}(U^{\frac{d+2n-m-3}{2}}x_{d+2n-2}) = U^{0} \neq 0 \in \bt_m$.  However, if $d>m+1$, then $U^{\frac{d+2n-m-3}{2}}x_{d+2n-2}=0$, a contradiction.  
\end{proof}

\begin{lem}\label{lem:deco1}
There exists a decomposition 
\begin{equation}\label{eq:deco1}
\tilde{H}^{S^1}_*(X_+) = J_1  \oplus J_2,
\end{equation} as a direct sum of $\f[U]$-modules $J_1$ and $J_2$, where $d_{S^1}$ vanishes on $J_2$ and
\begin{equation*}
J_1 = \bigoplus^N_{i=1} \bt_{d_i}(n_i),
\end{equation*}
with $2n_i +d_i> 2n_{i+1}+d_{i+1}$, and $d_{i+1} > d_i$, for some $N$.  Moreover, $d_N \leq 1+m, 2n_N+d_N \geq 3+m$, and $d_{S^1}$ is nonvanishing on each summand $\bt_{d_i}(n_i)$.
\end{lem}

\begin{proof}
To begin, set $\tilde{H}^{S^1}_*(X_+)=J_1 \oplus J_2$ for some choices of $J_1$ and $J_2$ so that $d_{S^1}|_{J_2}=0$, possibly by setting $J_2=0$.  We introduce a partial ordering $\succeq$ of (graded) $\f[U]$-modules.  We say 
\begin{equation*}
T_{d_1}(n_1) \succeq T_{d_2}(n_2)
\end{equation*}
if $2n_1+d_1 \geq 2n_2+d_2$ and $d_1 \geq d_2$.  Our goal is to arrange that the summands of $J_1$ are not comparable under this relation.  Suppose we have $\bt_{d_1}(n_1)\oplus \bt_{d_2}(n_2) \subset J_1$, and $\bt_{d_1}(n_1) \succeq \bt_{d_2}(n_2)$.  If one of the $\bt_{d_i}(n_i)$ has $d_{S^1}|_{\bt_{d_i}(n_i)}=0$, we move it to $J_2$.  Otherwise, we have that $d_{S^1}$ is nontrivial on both $\bt_{d_i}(n_i)$.  Let $\bt_{d_i}(n_i)$ be generated by $x_i$ for $i=1,2$.  Then $\langle x_1,U^{n_1-n_2+(d_1-d_2)/2}x_1+x_2 \rangle$ are new $\f[U]$-generators for $\bt_{d_1}(n_1) \oplus \bt_{d_2}(n_2) \subset J_1$, such that $d_{S^1}$ vanishes on $U^{n_1-n_2+(d_1-d_2)/2}x_1+x_2$, i.e. so that $d_{S^1}$ vanishes on the $\bt_{d_2}(n_2)$ submodule.  So we may choose a new decomposition $\tilde{H}^{S^1}_*(X_+)=J_1' \oplus J_2'$, where $J_2' \simeq J_2 \oplus \bt_{d_2}(n_2)$.  Thus, we may choose $J_1$ such that there is no submodule $X \oplus Y$ of $J_1$ with $X \succeq Y$.  Say $J_1=\bigoplus_{i=1}^N\bt_{d_i}(n_i)$ has been chosen so that all its summands are incomparable under $\succeq$ (and so that $d_{S^1}$ is nonvanishing on each $\bt_{d_i}(n_i)$).  Perhaps by reordering, let $d_{i+1} \geq d_{i}$.  If $d_{i+1}=d_{i}$, $\bt_{d_i}(n_i)$ and $\bt_{d_{i+1}}(n_{i+1})$ would be comparable, contradicting our choice of $J_1$.  Thus $d_{i+1}>d_i$.  Again using that the $\bt_{d_i}(n_i)$ are incomparable, we obtain $2n_i +d_i>2n_{i+1}+d_{i+1}$.  Finally, we saw in Lemma \ref{lem:dimoftows} that $d_{S^1}$ vanishes on any summand $\bt_d(n)$ with $d>1+m$ or $2n+d<3+m,$ so by the condition that $d_{S^1}$ is nonvanishing, we have $d_N \leq 1+m, 2n_N+d_N \geq 3+m$.  
\end{proof}

\begin{lem}\label{lem:deco2}
Let $\tilde{H}^{S^1}_*(X_+) = J_1 \oplus J_2$, with $J_1$ as in Lemma \ref{lem:deco1}.  Then
\begin{equation}\label{eq:deco2}
\tilde{H}^{S^1}_*(X)=\bt_{d_1+2n_1-1}\oplus\bigoplus^{N}_{i=1}\bt_{d_i}(\frac{d_{i+1}+2n_{i+1}-d_i}{2})\oplus \bigoplus^{N}_{i=1}\bt_{d_i}(n_i) \oplus J_2^{\oplus 2}.
\end{equation}
We interpret $d_{N+1}=m+1,n_{N+1}=0$.  The expression $\frac{d_{N+1}+2n_{N+1}-d_N}{2}$ may vanish, in which case $\bt_{d_N}(\frac{d_{N+1}+2n_{N+1}-d_N}{2})$ is the zero module.
\end{lem}
\begin{proof}
In the decomposition of Lemma $\ref{lem:deco1}$, we write $x_i$ for the generator of $\bt_{d_i}(n_i)$.  We choose a basis for $\mathrm{ker} \, d_{S^1}$, given by $\{y_i \}_i$ for $y_i=x_{i+1}+U^{n_i-n_{i+1}+(d_i-d_{i+1})/2}x_i$ for $i=1,...,n-1,$ and $y_N=U^{(d_N+2n_N-1)/2}x_N$.  Note that $y_N$ may be zero.

We have seen that $J_2 \subset \mathrm{ker}\; d_{S^1}$, and also $jJ_2 \subset \mathrm{ker}\; d_{S^1}$, giving the two copies of the $J_2$ summand in (\ref{eq:deco2}).  We see that $\f[U]  U^{-\frac{d_1+2n_1-m-3}{2}} = \mathrm{Im} \; d_{S^1}\subset \bt_m$, by Lemma \ref{lem:deco1}.  Then $\bt_{d_1+2n_1-1}=\mathrm{coker} \; d_{S^1}$.  Further, $(1+j)J_1$ contributes the summand $\bigoplus_{i=1}^N \bt_{d_i}(n_i)$, since $d_{S^1}$ is $j$-invariant, and so vanishes on multiples of $(1+j)$.  Finally, the set $\{y_i\}$ generates the $\bigoplus^N_{i=1} \bt_{d_i}(\frac{d_{i+1}+2n_{i+1}-d_{i}}{2})$ summand.  

For an example of how the new basis gives the Lemma, see Figures \ref{fig:fig62} and \ref{fig:62pt2}. 
\begin{figure}
\input{exampu.pdftex_t}
\caption{An example of $\tilde{H}^{S^1}_*(X)$ as in Lemma \ref{lem:deco2}.  The first four (finite) towers are $\bt_{-1}(3)^{\oplus 2} \oplus \bt_{1}(1)^{\oplus 2}$.  Then $J_1= \bt_{-1}(3) \oplus \bt_{1}(1)$ and $J_2 = \bt_{-1}(1)$ in (\ref{eq:deco1}) (keeping in mind that the action of $j$ interchanges the pairs of copies $\bt_{d_i}(n_i)$, so $\tilde{H}^{S^1}_*(X/X^{S^1})\simeq J_1 \oplus J_2 \oplus J_1 \oplus J_2$ as an $\f[U]$-module).  In particular, $d_1=-1,n_1=3,d_2=1,n_2=1$.  Here $m=0$.  The shaded-head arrows denote differentials while the open-head arrows denote $U$-actions.  }
\label{fig:fig62}
\end{figure}
\begin{figure}
\input{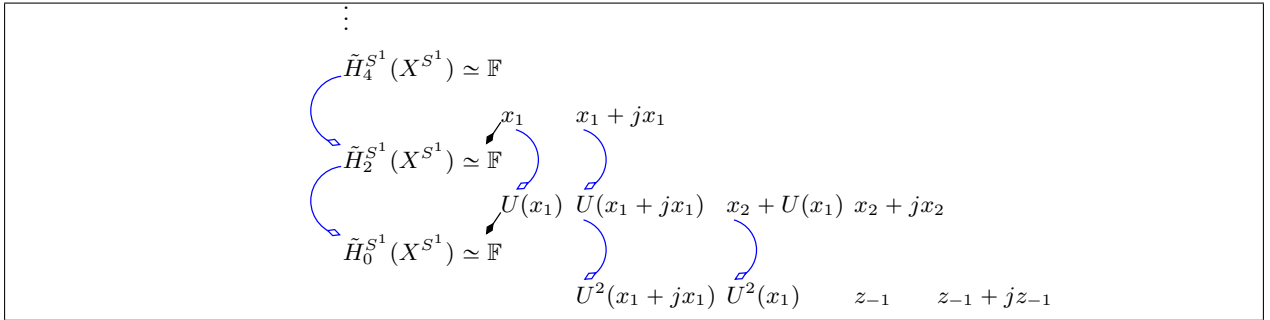}
\caption{Using the basis in the proof of Lemma \ref{lem:deco2} for the complex of Figure \ref{fig:fig62}.  Here the generator of $J_2$ is written $z_{-1}$.  The $x_i$ are generators of $\bt_{d_i}(n_i)$ for $i=1,2$.}
\label{fig:62pt2}
\end{figure} \end{proof} 
We now compute $\tilde{H}^{G}_*(X/X^{S^1})$.  To find $\mathrm{ker}\, d_G$, we write $\tilde{H}^{G}_*(X/X^{S^1}) = J_1' \oplus J_2'$, where $d_G$ vanishes on $J_2'$ ($J_2'$ need not be maximal, currently).  To find $J_1'$ and $J_2'$ in terms of $J_1$ and $J_2$, we use: 

\begin{lem}\label{lem:deco3}
Let $J_1,J_2$ and $d_i,n_i$ be as in Lemma \ref{lem:deco1}.  Then we may set $\tilde{H}^{G}_*(X/X^{S^1})=J_1' \oplus J_2'$, where 
\begin{equation*}
J_1' = \bigoplus_{\{i \mid d_i \equiv m+1\, \mathrm{mod} \,4 \}} \bv_{d_i}(\lfloor\frac{n_i+1}{2}\rfloor) \oplus \bigoplus_{\{i \mid d_i \equiv m+3 \,\mathrm{mod}\, 4 \}} \bv_{d_i+2}( \lfloor \frac{n_i}{2} \rfloor ),
\end{equation*}
\begin{equation*}
J_2' = \re^{\f[U]}_{\f[v]} J_2 \oplus \bigoplus_{\{i \mid d_i \equiv m+1 \,\mathrm{mod}\, 4 \}} \bv_{d_i+2}(\lfloor \frac{n_i}{2} \rfloor) \oplus \bigoplus_{\{i \mid d_i \equiv m+3 \,\mathrm{mod}\, 4 \}} \bv_{d_i}(\lfloor \frac{n_i+1}{2} \rfloor).
\end{equation*}
Moreover, $d_G$ is nonvanishing on each nontrivial summand of $J_1'$, and $d_G(J_2')=0$.
\end{lem}
\begin{proof}
We use (\ref{eq:restri}) and (\ref{eq:quotrel}) to conclude that 
\[\phi_3J_1=\bigoplus^N_{i=1}\bv_{d_i}(\lfloor \frac{n_i+1}{2} \rfloor)\oplus \bigoplus^N_{i=1}\bv_{d_i+2}(\lfloor \frac{n_i}{2} \rfloor). \]  We also use 
\begin{equation*}
\co^{S^1}_{G} d_{S^1} =d_G \phi_3 ,
\end{equation*}
as in (\ref{eq:dgrho2}) to obtain that $d_G$ is nonvanishing on each of $\bv_{d_i}(\lfloor \frac{n_i+1}{2} \rfloor)$, with $d_i \equiv m+1 \,\mathrm{mod}\,4$ and $\bv_{d_i+2}(\lfloor \frac{n_i}{2} \rfloor)$ with $d_i \equiv m+3 \,\mathrm{mod}\,4$.  To find $J_2'$ we apply (\ref{eq:quotrel}) again, to $J_2$, and we observe that $d_G$ is vanishing on each of $\bv_{d_i}(\lfloor \frac{n_i+1}{2} \rfloor)$, with $d_i \equiv m+3 \,\mathrm{mod}\,4$ and $\bv_{d_i+2}(\lfloor \frac{n_i}{2} \rfloor)$ with $d_i \equiv m+1 \,\mathrm{mod}\,4$.
\end{proof}
\begin{fact}\label{fct:onimageofphi2}
The $\f[v]$-submodule \[\bigoplus_{\{i \mid d_i \equiv m+1 \,\mathrm{mod}\, 4 \}} \bv_{d_i+2}(\lfloor \frac{n_i}{2} \rfloor) \oplus \bigoplus_{\{i \mid d_i \equiv m+3 \,\mathrm{mod}\, 4 \}} \bv_{d_i}(\lfloor \frac{n_i+1}{2} \rfloor)\]
in Lemma \ref{lem:deco3} is the component of $\tilde{H}^G_*(X/X^{S^1})$ not in the image of $\phi_2|_{\mathrm{ker} \, d_{S^1}}$.  
\end{fact}
For an example of Lemma \ref{lem:deco3}, see Figure \ref{fig:63}.
\begin{figure}
\input{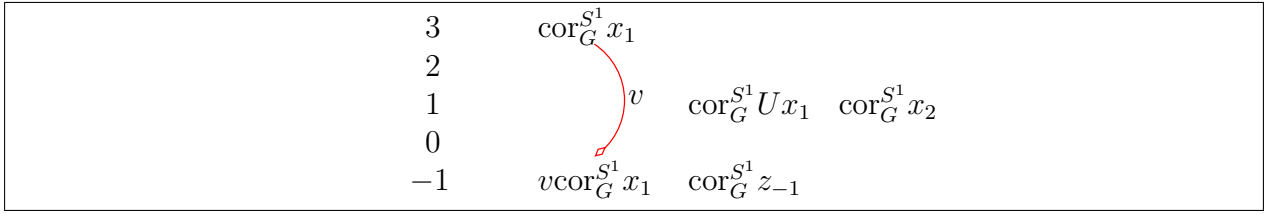}
\caption{Computing $\tilde{H}^G_*(X/X^{S^1})$ for the complex of Figures \ref{fig:fig62} and \ref{fig:62pt2}.  Here $J_1'=\bv_1(1)^{\oplus 2}$, and $J_2'=\bv_{-1}(2)\oplus \bv_{-1}(1).$  }
\label{fig:63}
\end{figure}

We define an order $\succeq$ on modules $\bv_d(n)$ with $d \equiv m+1 \, \mathrm{mod}\, 4$.  Note that all simple submodules $\bv_{d}(n)$ of $J_1'$ in Lemma \ref{lem:deco3} have $d \equiv m+1\, \mathrm{mod}\, 4$.  Let $\bv_{d_1}(n_1) \succeq \bv_{d_2}(n_2)$ if $d_1 \geq d_2$ and $d_1 +4n_1 \geq d_2+4n_2$.  Let $\mathcal{J}$ denote the set of distinct pairs $(a,b)$ for which $\bv_{a}(b)$ is a maximal summand of $J_1'$ as in Lemma \ref{lem:deco3}.  If $(a,b) \in \mathcal{J}$, set $m(a,b)+1$ to be the multiplicity with which $\bv_a(b)$ occurs as a summand of $J_1'$.  If $(a,b) \not\in \mathcal{J}$, set $m(a,b)$ to be the multiplicity with which $\bv_a(b)$ occurs in $J_1'$.  Then we define:
\begin{equation}\label{eq:jrep}
J_{\mathrm{rep}} = \bigoplus_{(a,b)} \bv_a(b)^{\oplus m(a,b)}  , 
\end{equation}
where summands of multiplicity $0,-1$ do not contribute to the sum. 
That is, $J_{\mathrm{rep}}$ counts the repeated summands (whence the ``rep") in $J_1'$, as well as those which are not contributing ``new" differentials targeting the reducible.  In the example of Figure \ref{fig:63}, $J_{\mathrm{rep}}=\bv_1(1)$.

 Arguing as in Lemma \ref{lem:deco1}, we obtain the following.    
\begin{lem}\label{lem:deco3.5}  
Let $\tilde{H}^{S^1}_*(X_+)$ be decomposed as in Lemma \ref{lem:deco1}, and let $\mathcal{J}$ be as in the preceding paragraphs.  Then we may set $\tilde{H}^{G}_*(X/X^{S^1})=J_1''\oplus J_2''$ with
\begin{equation*}
J_1''\simeq \bigoplus_{(a_i,b_i) \in \mathcal{J}} \bv_{a_i}(b_i),
\end{equation*}
\begin{equation*}
J_2'' \simeq \re^{\f[U]}_{\f[v]} J_2 \oplus \bigoplus_{\{i \mid d_i \equiv m+1 \,\mathrm{mod}\, 4 \}} \bv_{d_i+2}(\lfloor \frac{n_i}{2} \rfloor) \oplus \bigoplus_{\{i\mid d_i \equiv m+3 \,\mathrm{mod}\, 4 \}} \bv_{d_i}(\lfloor \frac{n_i+1}{2} \rfloor)\oplus J_{\mathrm{rep}}.
\end{equation*}
Moreover, $d_G$ is nonvanishing on each nontrivial summand of $J_1''$, and $d_G(J_2'')=0$.  Further, $a_i<a_{i+1}$ and $a_i+4b_i>a_{i+1}+4b_{i+1}$ for $i=1,...,N_0-1$, where $N_0=|\mathcal{J}|$.   
\end{lem}
\begin{proof}
We argue as in Lemma \ref{lem:deco1}, starting with the decomposition 
\[\tilde{H}^{G}_*(X/X^{S^1})=J_1'\oplus J_2'\]
given in Lemma \ref{lem:deco3}.  We will show that we may choose $J_1''=\bigoplus_{(a_i,b_i) \in \mathcal{J}} \bv_{a_i}(b_i)$, so that $\tilde{H}^G_*(X/X^{S^1})=J_1'' \oplus J_2''$ with $d_GJ_2''=0$.  Fix a direct sum decomposition $J_1'=\bigoplus_i \bv_{a_i}(b_i)$, for some $a_i,b_i$.  Say that $\bv_{e_1}(f_1) \subseteq J_1' $, where $(e_1,f_1) \not\in \mathcal{J}$ and choose $(e_2,f_2) \in \mathcal{J}$, with $\bv_{e_2}(f_2) \succeq \bv_{e_1}(f_1)$ and $\bv_{e_1}(f_1)\oplus \bv_{e_2}(f_2) \subseteq J_2'$.  Further, assume that $d_G$ is nontrivial on $\bv_{e_1}(f_1)$; if it were trivial, then we enlarge $J_2'$ by setting $J_2''=J_2' \oplus \bv_{e_1}(f_1)$.  Let $x_i$ be the generator of $\bv_{e_i}(f_i)$.  We choose new $\f[v]$-generators, $x_2$ of $\bv_{e_2}(f_2)$ and $v^{f_2-f_1+(e_2-e_1)/4}x_2+x_1$ of $\bv_{e_1}(f_1)$ so that $d_G$ vanishes on $\bv_{e_1}(f_1)$.  Again, then we may enlarge $J_2'$ by adding the $\bv_{e_1}(f_1)$ factor.  This shows that we can remove all summands $\bt_a(b)$ with $(a,b) \not \in \mathcal{J}$ from $J_1'$.  
Similarly, if $\bv_a(b)\oplus \bv_a(b) \subseteq J_1'$, with $(a,b)\in \mathcal{J}$ and with generators $x_1$ and $x_2$ such that $d_G(x_1)=d_G(x_2) \neq 0$, we choose the new basis $\langle x_1, x_2+x_1 \rangle$.  The differential $d_G$ is nonzero on the copy of $\bv_a(b)$ generated by $x_1$, while $d_G$ vanishes on the copy of $\bv_a(b)$ generated by $x_1+x_2$, and $J_2'$ may be enlarged.  Then we may choose $J_1''\simeq \bigoplus_{(a,b) \in \mathcal{J}} \bv_a(b)$.  The formula for $J_2''$ also follows once $J_1''$ is specified.  
\end{proof} 

In Figures \ref{fig:64} and \ref{fig:65}, we provide an example illustrating the proof of Lemma \ref{lem:deco3.5}.
\begin{figure}
\input{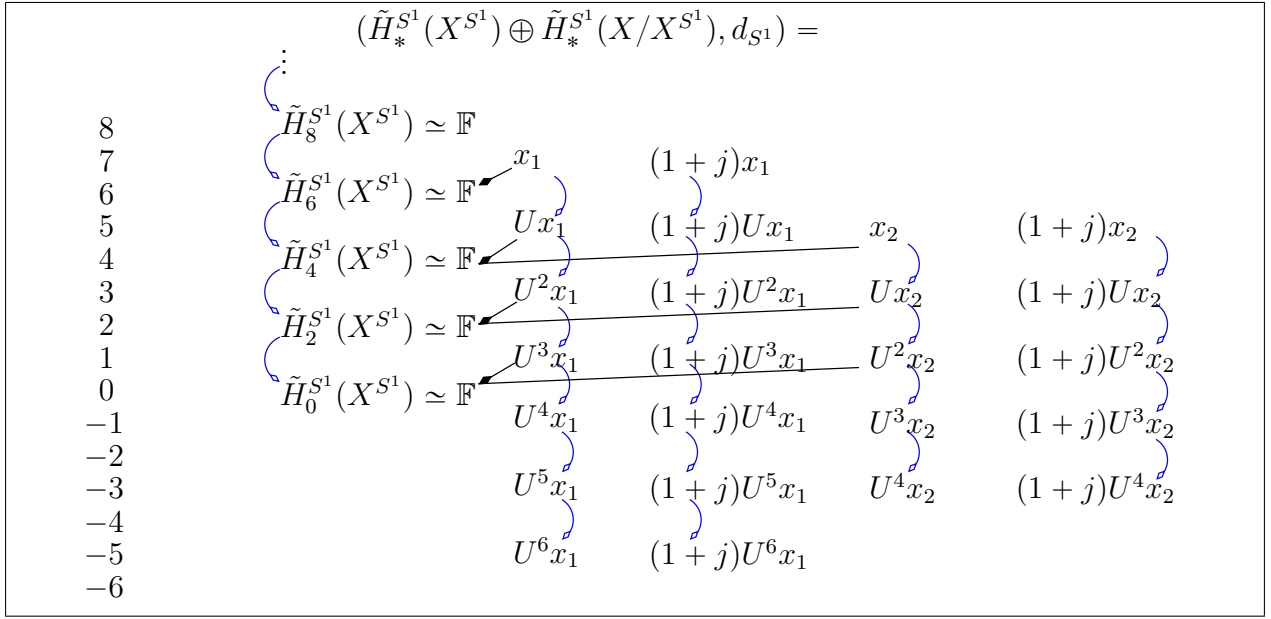}
\caption{ An example $\f[U]$-module $\tilde{H}^{S^1}_*(X^{S^1}) \oplus \tilde{H}^{S^1}_*(X/X^{S^1})$ for $X$ with $m=0$.  Here $d_1=-5,n_1=7$ and $d_2=-3,n_2=5$, and $J_2=0$.}
\label{fig:64}
\end{figure}

\begin{figure}
\input{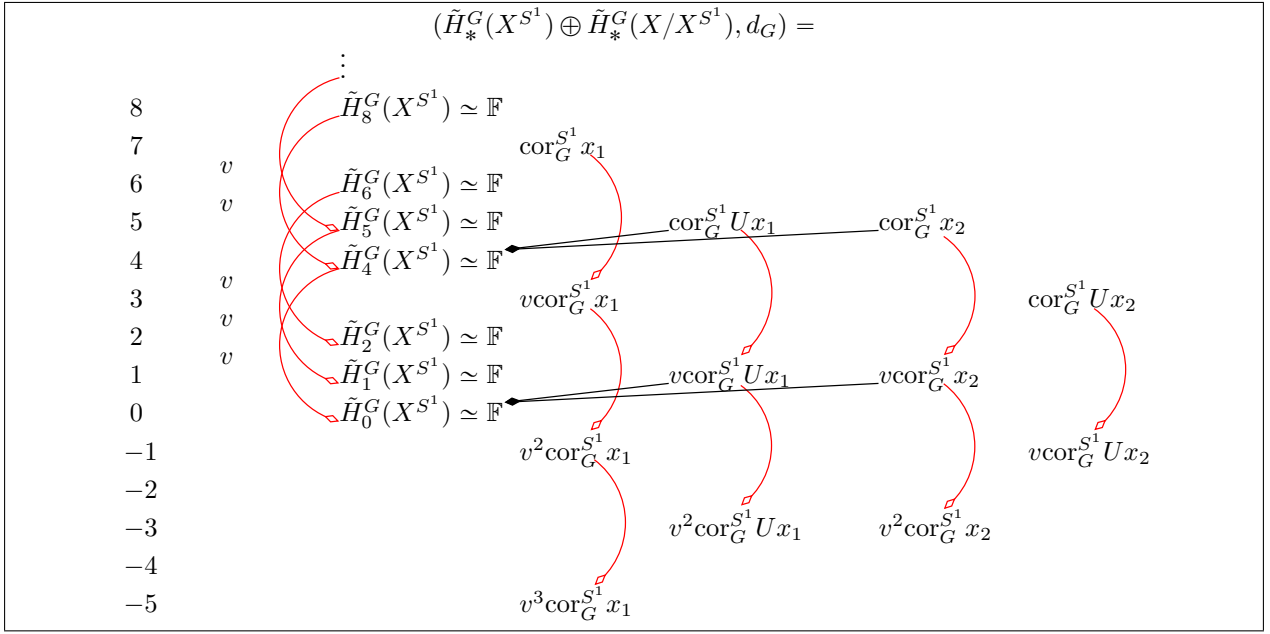}
\caption{ Here we show how to compute $(\tilde{H}^G_*(X^{S^1}) \oplus \tilde{H}^G_*(X/X^{S^1}),d_G)$, given $(\tilde{H}^{S^1}_*(X^{S^1}) \oplus \tilde{H}^{S^1}_*(X/X^{S^1}),d_{S^1})$, for the example complex given in Figure \ref{fig:64}.  The curved arrows denote the $v$-action.  Here, $J_\mathrm{rep}$ is $\bv_{-3}(3)$, and $J''_1=\bv_{-3}(3)$.  Then we have also $J_2''=\bv_{-3}(3) \oplus \bv_{-1}(2) \oplus \bv_{-5}(4)$.  If we have a basis of $\co^{S^1}_G Ux_1$, $\co^{S^1}_G x_2$ for $J_1'$, then $\co^{S^1}_GUx_1+\co^{S^1}_Gx_2$ would be a basis for $J_\mathrm{rep}$ produced by Lemma \ref{lem:deco3.5}.}
\label{fig:65}
\end{figure}
We may now compute $\tilde{H}^{G}_*(X)$ in terms of $\tilde{H}^{S^1}_*(X/X^{S^1})$ and the map $d_{S^1}$.
\begin{lem}\label{lem:deco4}
Let $\tilde{H}^{S^1}_*(X_+)$ be decomposed as in Lemma \ref{lem:deco1} and let $J_1'', J_2''$ be as in Lemma \ref{lem:deco3.5}.  Then:
\begin{eqnarray}\label{eq:splitcalc}
\tilde{H}^{G}_*(X) &=& \bv_{a_1+4b_1-1} \oplus \bv_{1+m} \oplus \bv_{2+m} \\ &&\oplus \bigoplus_{i=1}^{N_0} \bv_{a_i}(\frac{a_{i+1}+4b_{i+1}-a_i}{4})\oplus J_2'', \nonumber
\end{eqnarray}
as an $\f[v]$-module.  The $q$-action is given by the isomorphism $q: \bv_{2+m} \rightarrow \bv_{1+m}$ and the map $\bv_{1+m} \rightarrow \bv_{a_1+4b_1-1}$, which is an $\f$-vector space isomorphism in all degrees at least $a_1+4b_1-1$.  The action of $q$ annihilates $\bigoplus_{i=1}^{N_0} \bv_{a_i}(\frac{a_{i+1}+4b_{i+1}-a_i}{4})$ and $\re^{\f[U]}_{\f[v]}J_2 \oplus J_{\mathrm{rep}}\subseteq J_2''$.  

To finish specifying the $q$-action, let $x_i$ be a generator of $\bv_{d_i+2}(\lfloor \frac{n_i}{2} \rfloor)$ for $i$ such that $d_i \equiv m+1 \, \mathrm{mod}\, 4$ (respectively, let $x_i$ be a generator of $\bv_{d_i}(\lfloor \frac{n_i+1}{2} \rfloor)$ if $d_i \equiv m+3 \, \mathrm{mod}\, 4$).  Then $qx_i$ is the unique nonzero element of $H^G_*(X/X^{S^1})$ in grading $\mathrm{deg} \, x_i -1$, for all $i$.  In particular, $\tilde{H}^{S^1}_*(X/X^{S^1})$ and $d_{S^1}$ determine $\tilde{H}^{G}_*(X)$.  Here $a_{N_0+1}=m+1,b_{N_0+1}=0$.     
\end{lem}
\begin{proof}
The proof is analogous to that of Lemma \ref{lem:deco2}.  We choose a basis for $\mathrm{ker}\, d_G$ as follows.   Write the generator of $\bv_{a_i}(b_i)$ as $x_i$.  Then set $y_i=x_{i+1}+v^{b_i-b_{i+1}+(a_i-a_{i+1})/4}x_i$ for $i=1,...,N_0-1,$ and $y_{N_0}=v^{(a_{N_0}+4b_{N_0}-1)/4}x_{N_0}$.  It is clear that $y_i \in \mathrm{ker} \, d_G$ for all $i$, and it is straightforward to check that $\{y_i\}$ generates $\mathrm{ker} \, d_G \cap J_1''$.  The $y_i$ generate the term $\bigoplus^{N_0}_{i=1} \bv_{a_i}(\frac{a_{i+1}+4b_{i+1}-a_{i}}{4})$ in (\ref{eq:splitcalc}).  Since $d_G$ is $q$-equivariant and $q$ annihilates $\tilde{H}^G_*(X/X^{S^1})$, the modules $\bv_1$ and $\bv_2 \subset H_*(BG)$ are disjoint from the image of $d_G$.  Moreover, $v^{-\frac{a_1+4b_1-5-m}{4}}=d_G(x_1)$, where $v^{-k}$ is the unique element $x$ of $H_*(BG)[-m]$ with $v^kx$ an $\f$-generator of $H_0(BG)[-m]$.  Since there are no elements $x \in J_1''$ with grading greater than $a_1+4b_1-4$, the maximal $k$ for which $v^{-k} \in \mathrm{Im} \; d_G$ is $\frac{a_1+4b_1-5-m}{4}$.  It follows that 
\[ \mathrm{coker}\; d_G = \bv_{a_1+4b_1-1} \oplus \bv_{1+m} \oplus \bv_{2+m}. \]
Furthermore, $J_2'' \subseteq \mathrm{ker} \; d_G$ by definition, contributing the $J_2''$ term of (\ref{eq:splitcalc}).  To determine the $q$-action on $\mathrm{ker} \, d_G$, we use Lemma \ref{lem:3g}.  Indeed, $q$ takes elements not in the image of $\phi_2|_{\mathrm{ker}\, d_{S^1}}$ to nontrivial elements of $\mathrm{coker}\, d_G$, and $q$ vanishes on $\mathrm{Im}\, \phi_2|_{\mathrm{ker}\,d_{S^1}}$.  Using Fact \ref{fct:onimageofphi2}, we obtain the $q$-action on $J_2''$ as in the Lemma. The $q$-action on $\mathrm{coker} \, d_G$ is given by that on $H_*(BG)$.   
\end{proof}

We combine Lemmas \ref{lem:deco1}-\ref{lem:deco4} to determine $\tilde{H}^{G}_*(X)$ from $\tilde{H}^{S^1}_*(X)$.  We record this as the following Theorem.

\begin{thm}\label{thm:onjspaces}
Let $X=(X',p,h/4)\in \mathfrak{E}$ and $X'$ be a $j$-split space of type SWF.  Then:
\begin{equation}\label{eq:hffullswfh}
\tilde{H}^{S^1}_*(X)=\bt_{s+d'_1+2n_1-1} \oplus \bigoplus^{N}_{i=1}\bt_{s+d'_i}(\frac{d'_{i+1}+2n_{i+1}-d'_i}{2})\oplus \bigoplus^N_{i=1} \bt_{s+d'_i}(n_i) \oplus J^{\oplus 2}[-s],
\end{equation}
for some constants $s,d'_i,n_i,N$ and some $\f[U]$-module $J$, where $2n_i+d'_i > 2n_{i+1}+d'_{i+1}$ and $d'_{i}<d'_{i+1}$ for all $i$,  $2n_N+d'_N \geq 3$, $d'_N \leq 1$, and $d'_{N+1}=1,n_{N+1}=0$.  Let $\mathcal{J}_0=\{(a_k,b_k)\}_k$ be the collection of pairs containing all $(d'_i,\lfloor \frac{n_i+1}{2} \rfloor)$ for $d'_i \equiv 1 \;\mathrm{mod}\; 4$ and all $(d'_i+2,\lfloor \frac{n_i}{2} \rfloor)$ for $d'_i \equiv 3 \; \mathrm{mod} \; 4$, counting multiplicity.  Let $(a,b) \succeq (c,d)$ if $a+4b \geq c+4d$ and $a \geq c$, and let $\mathcal{J}$ be the subset of $\mathcal{J}_0$ consisting of pairs maximal under $\succeq$ (not counted with multiplicity).  If $(a,b) \in \mathcal{J}$, set $m(a,b)+1$ to be the multiplicity of $(a,b)$ in $\mathcal{J}_0$.  If $(a,b) \not\in \mathcal{J}$, set $m(a,b)$ to be the multiplicity of $(a,b)$ in $\mathcal{J}_0$.  Let $|\mathcal{J}|=N_0$ and order the elements of $\mathcal{J}$ so that $\mathcal{J}=\{(a_i,b_i) \}_i $, with $a_i+4b_i > a_{i+1} +4b_{i+1}$.  We interpret $a_{N_0+1}=1,b_{N_0+1}=0$.  Then: 
\begin{eqnarray} \tilde{H}^G_*(X) & = & (\bv_{4\lfloor \frac{d'_1+2n_1+1}{4} \rfloor} \oplus \bv_{1} \oplus \bv_{2} \\ &&\oplus \bigoplus_{i=1}^{N_0} \bv_{a_i}(\frac{a_{i+1}+4b_{i+1}-a_i}{4}) \oplus \bigoplus_{(a,b) \in \mathcal{J}_0} \bv_a(b)^{\oplus m(a,b)} \oplus \re^{\f[U]}_{\f[v]}J \nonumber\\ & & \oplus \bigoplus_{\{i \mid d'_i \equiv 1\; \mathrm{mod}\;4\} } \bv_{d'_i+2}(\lfloor \frac{n_i}{2} \rfloor) \oplus \bigoplus_{\{i \mid d'_i \equiv 3 \; \mathrm{mod}\; 4 \} } \bv_{d'_i}(\lfloor \frac{n_i+1}{2} \rfloor) )[-s].\nonumber
\end{eqnarray}

The $q$-action is given by the isomorphism $q:\bv_2[-s] \rightarrow \bv_1[-s]$ and the map $q:\bv_1[-s] \rightarrow \bv_{4\lfloor \frac{d'_1+2n_1+1}{4}\rfloor}[-s]$ which is an $\f$-vector space isomorphism in all degrees (in $\bv_1[-s]$) greater than or equal to $4\lfloor\frac{d'_1+2n_1+1}{4}\rfloor+s+1$, and vanishes on elements of $\bv_1[-s]$ of degree less than $4\lfloor\frac{d'_1+2n_1+1}{4}\rfloor+s+1$.

The action of $q$ annihilates $\bigoplus_{i=1}^{N_0} \bv_{a_i}(\frac{a_{i+1}+4b_{i+1}-a_i}{4})[-s]$ and $(\bigoplus_{(a,b) \in \mathcal{J}_0} \bv_a(b)^{\oplus m(a,b)}\oplus \re^{\f[U]}_{\f[v]}J)[-s]$.  

To finish specifying the $q$-action, let $x_i$ be a generator of $\bv_{d_i'+2}(\lfloor \frac{n_i}{2} \rfloor)[-s]$ for $i$ such that $d_i' \equiv 1 \, \mathrm{mod}\, 4$ (respectively, let $x_i$ be a generator of $\bv_{d_i'}(\lfloor \frac{n_i+1}{2} \rfloor)[-s]$ if $d_i' \equiv 3 \, \mathrm{mod}\, 4$).  Then $qx_i$ is the unique nonzero element of $(\bv_{4\lfloor \frac{d'_1+2n_1+1}{4} \rfloor} \oplus \bv_{1} \oplus \bv_{2})[-s]$ in grading $\mathrm{deg} \, x_i -1$, for all $i$.
\end{thm}
\begin{proof}
We show that for $M$ an $\f[U]$-module of the form (\ref{eq:deco2}), the sets $\{n_i\}$,$\{d'_i\}$, and the module $J_2$, are determined by the (graded) isomorphism type of $M$, to establish that all the constants in (\ref{eq:hffullswfh}) are well-defined (independent of the choice of direct sum decomposition of $\tilde{H}^{S^1}_*(X)$).  For a fixed $d$, there are at most two distinct isomorphism classes $\bt_d(x)$, each appearing as summands of $M$ that occur an odd number of times in the decomposition of $M$ into simple submodules (not including the infinite tower).  Such a submodule $\bt_d(x)$ will be called a submodule \emph{occurring with odd multiplicity}.  For any $d$ such that there is at least one isomorphism class $\bt_d(x)$ with odd multiplicity, then $d=s+d'_i$ for some $i$, using (\ref{eq:deco2}).  Consider the case that there are exactly two such isomorphism classes $\bt_d(x_1)$ and $\bt_d(x_2)$ with, say, $x_1 < x_2$.  Setting $d=s+d'_i$ for a fixed $i$, and using (\ref{eq:deco2}), we see that $x_2=n_i$, since $n_i>n_{i+1}+\frac{d'_{i+1}-d'_i}{2}$ for all $i$.  If instead there is one (graded) isomorphism class $T_d(x)$ with odd multiplicity, Lemma \ref{lem:deco2} shows $x=n_N$.  If, for a fixed $d$, there are no isomorphism classes $\bt_d(x)$ occurring with odd multiplicity, then $d \not\in \{s+d'_i\}$.  Thus, we see that $\{d_i\}$ and $\{n_i\}$ are determined by the isomorphism type of $M$ as a graded $\f[U]$-module.  It is then easy to see that $J_2$ is also determined by the isomorphism type of $M$.  

In addition, we find that $s$ in (\ref{eq:hffullswfh}) exists and is uniquely determined.  First, we check that there is an $s$ so that (\ref{eq:hffullswfh}) holds.  Observe that $\tilde{H}^{S^1}_*(X)=\tilde{H}^{S^1}_*(X')[p+h]$.  Say that $X'$ is a space of type SWF at level $m$, and set $d_i'=d_i-m$.  Then Lemma \ref{lem:deco2} shows that (\ref{eq:hffullswfh}) holds for this choice of $d'_i$, and $s=m-p-h$.  We next show that there is a unique $s$ so that (\ref{eq:hffullswfh}) holds.  To see this, observe that $\tilde{H}^{S^1}_{*,\mathrm{red}}(X)$, as in (\ref{eq:hffullswfh}), is an $\f$-module of \emph{odd} rank in degrees $d$ such that $d \equiv s+1 \, \mathrm{mod} \, 2$, with $s<d<s+d'_1+2n_1$, and of even rank (possibly zero) in all other degrees (Recall from (\ref{eq:s1redhom}) the definition of $\tilde{H}^{S^1}_{*,\mathrm{red}}$).  Then, for $M$ an $\f[U]$-module that is the homology of $(X',p,h/4)$ with $X'$ $j$-split, we have that $s=m-p-h$ is determined by $M$.  

As in (\ref{eq:s1split}),
\[ \tilde{H}^{S^1}_*(X)=\mathrm{coker} d_{S^1} \oplus \mathrm{ker} d_{S^1}.  \]
Additionally, given $M$, we have determined the sets $\{d'_i\},\{n_i\}$ appearing in Lemma \ref{lem:deco1}.  Then Lemmas \ref{lem:deco3} and \ref{lem:deco3.5} show that $J_1''=\oplus_{(a_i,b_i) \in \mathcal{J}}\bv_{a_i}(b_i)$, for $a_i,b_i$ as in the statement of the Theorem, and that 
\begin{equation}\label{eq:thj2form}
J_2'' = \re^{\f[U]}_{\f[v]} J \oplus \bigoplus_{\{i| d_i \equiv 1 \,\mathrm{mod}\, 4 \}} \bv_{d_i+2}(\lfloor \frac{n_i}{2} \rfloor) \oplus \bigoplus_{\{i| d_i \equiv 3 \,\mathrm{mod}\, 4 \}} \bv_{d_i}(\lfloor \frac{n_i+1}{2} \rfloor)\oplus \bigoplus_{(a,b) \in \mathcal{J}_0} \bv_a(b)^{\oplus m(a,b)}.
\end{equation}
Here we have replaced the notation $\re^{\f[U]}_{\f[q,v]/(q^3)}$ by $\re^{\f[U]}_{\f[v]}$ since $q$ acts by $0$.  Finally, Lemma \ref{lem:deco4} determines $\tilde{H}^G_*(X)$ given $J_1''$ and $J_2''$.  This completes the proof of the Theorem.

\end{proof}
\begin{rmk}\label{rmk:thomonjspacescomplexes}
Since every $j$-split chain complex of type SWF is the cellular chain complex of some space of type SWF, Theorem \ref{thm:onjspaces} also applies to $j$-split chain complexes. 
\end{rmk}

We give an example illustrating the steps of the proof of Theorem \ref{thm:onjspaces}.  Let $X$ be a $j$-split space, and say that $\tilde{H}^{S^1}_*((X,p,h/4))$ is given as in Figure \ref{fig:6thm1}; that is:
\begin{equation*}
\tilde{H}^{S^1}_*((X,p,h/4)) \simeq \bt_6 \oplus \bt_{-5}(6) \oplus \bt_{-5}(5) \oplus \bt_{-3}(4)\oplus \bt_{-3}(3) \oplus \bt_{-1}(2) \oplus \bt_{-1}(1) .
\end{equation*}

\begin{figure}
\input{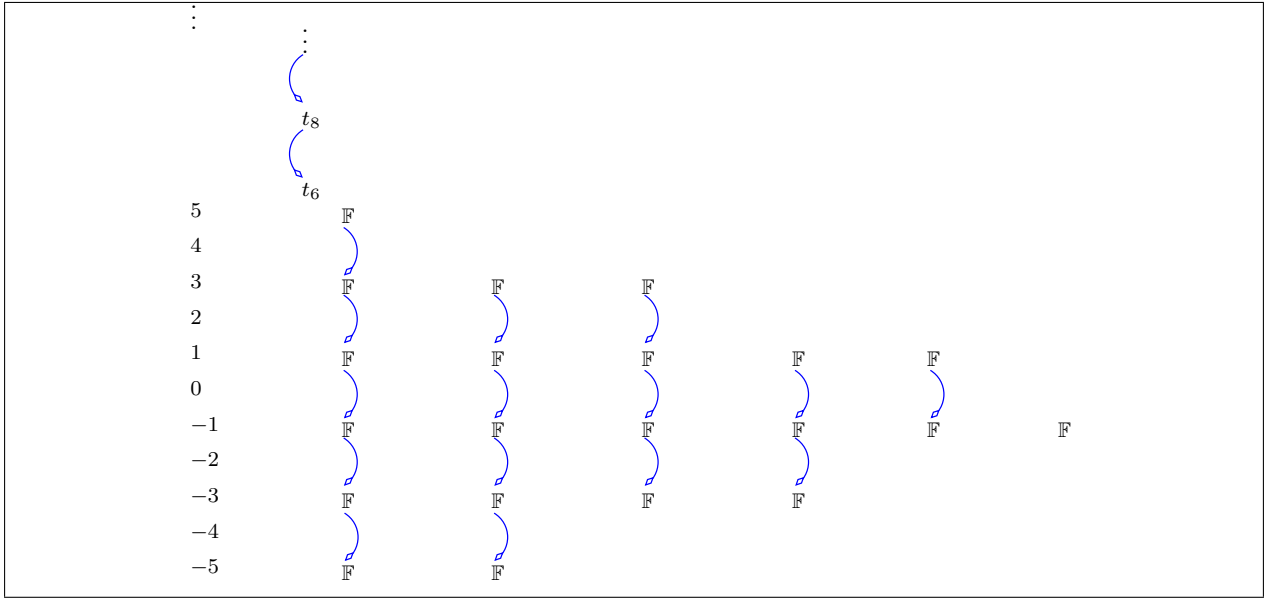}
\caption{The $S^1$-Borel Homology of $(X,p,h/4) \in \E$.  The variables $t_i$ stand for entries of the infinite tower in grading $i$.  }
\label{fig:6thm1}
\end{figure}

We calculate $d'_i,n_i$.  As specified in the proof of Theorem \ref{thm:onjspaces}, we see $\{d'_i+m-p-h\}=\{-5,-3,-1\}$, and $\{n_i\}=\{6,4,2\}$.  We see that $m-p-h=0$ because $\tilde{H}^{S^1}_{-1,\mathrm{red}}((X,p,h/4))$ (i.e. the contribution in degree $-1$ not coming from the tower) is of even rank, while $\tilde{H}^{S^1}_{1,\mathrm{red}}((X,p,h/4))$ has odd rank.  So $s=0$ in Theorem \ref{thm:onjspaces}.  Then $\{d'_i\}=\{-5,-3,-1\}$.  Furthermore, we see $J_2=0$.  Then we recover $( \tilde{H}^{S^1}_*((X/X^{S^1},p,h/4)) \oplus \tilde{H}^{S^1}_*((X^{S^1},p,h/4)),d_{S^1})$, as in Figure \ref{fig:6thm2}.

\begin{figure}
\input{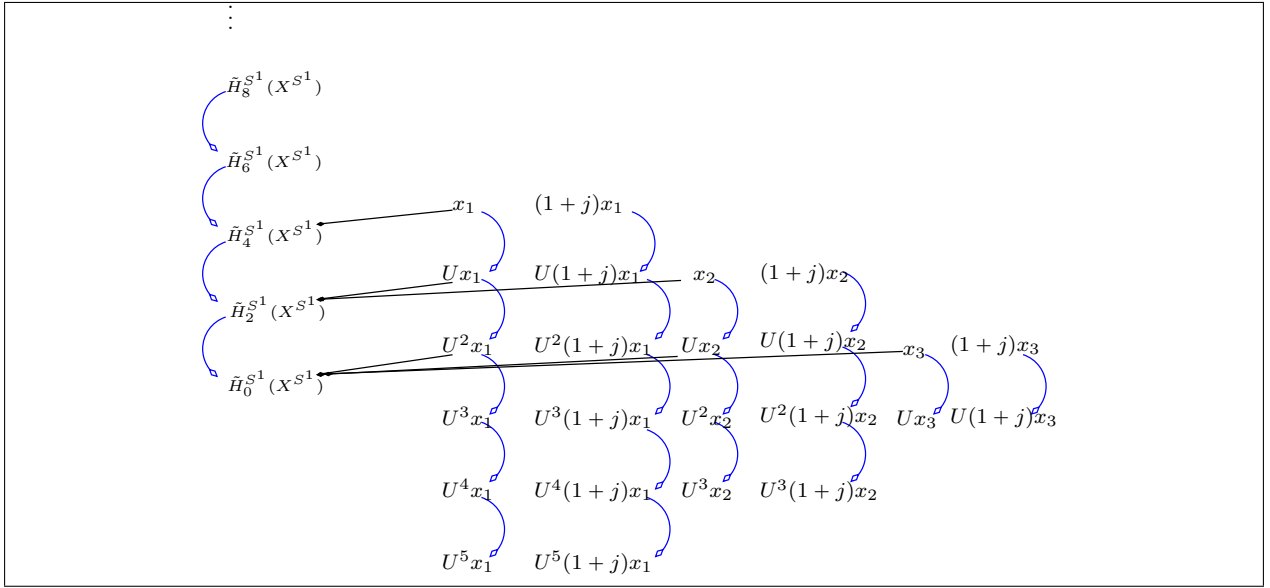}
\caption{The complex $(\tilde{H}^{S^1}_*(X/X^{S^1})[p+h] \oplus \tilde{H}^{S^1}_*((X^{S^1},p,h/4)),d_{S^1})$ corresponding to Figure \ref{fig:6thm1}.}
\label{fig:6thm2}
\end{figure}

Using Lemma \ref{lem:deco3}, we have $J_1'=\bv_{-3}(3)\oplus \bv_{-3}(2) \oplus \bv_{1}(1) $ and $J_2'=\bv_{-5}(3) \oplus \bv_{-1}(2) \oplus \bv_{-1}(1)$, as in Figure \ref{fig:6thm3}.  We see that $\bv_{-3}(2)$ is not maximal in $J_1'$, so $m(-3,2) =1$, while $m(-3,3)=0$, since $\bv_{-3}(3)$ is maximal under $\succeq$.  Similarly, $\bv_1(1)$ is maximal, so $m(1,1)=0$.  Then $J_\mathrm{rep}=\bv_{-3}(2)$, using (\ref{eq:jrep}).  

\begin{figure}
\input{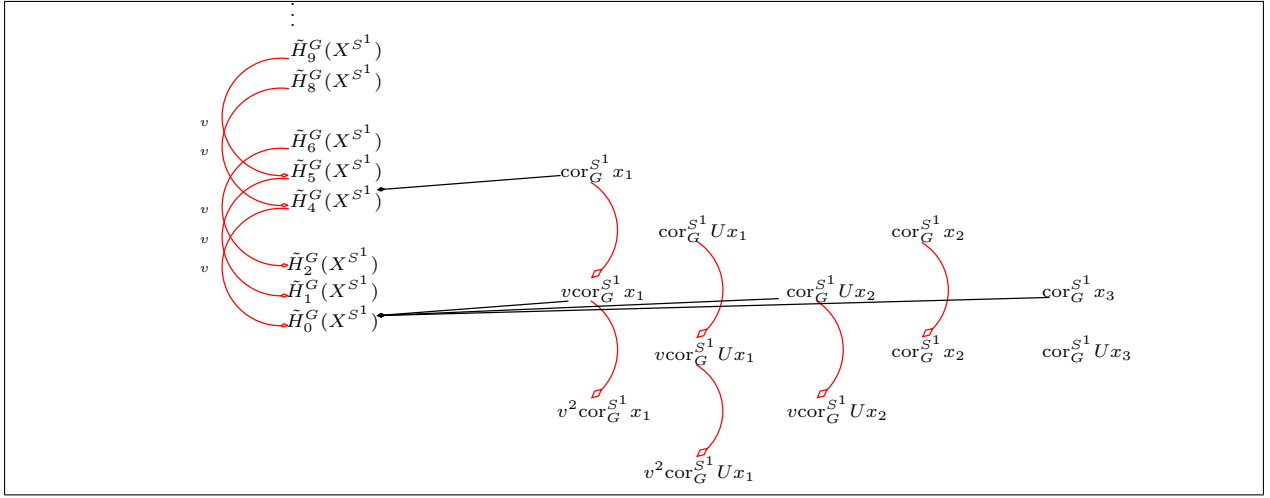}
\caption{The complex $(\tilde{H}^{G}_*(X/X^{S^1})[p+h] \oplus \tilde{H}^{G}_*((X^{S^1},p,h/4)),d_{G})$ corresponding to Figure \ref{fig:6thm1}. } 
\label{fig:6thm3}
\end{figure}

In Figure \ref{fig:6thm3}, $J_1''= \bv_{-3}(3)\oplus \bv_1(1)$.  Then Lemma \ref{lem:deco4} allows us to compute $\tilde{H}^G_*(X)$, as in Figure \ref{fig:6thm4}.

\begin{figure} 
\input{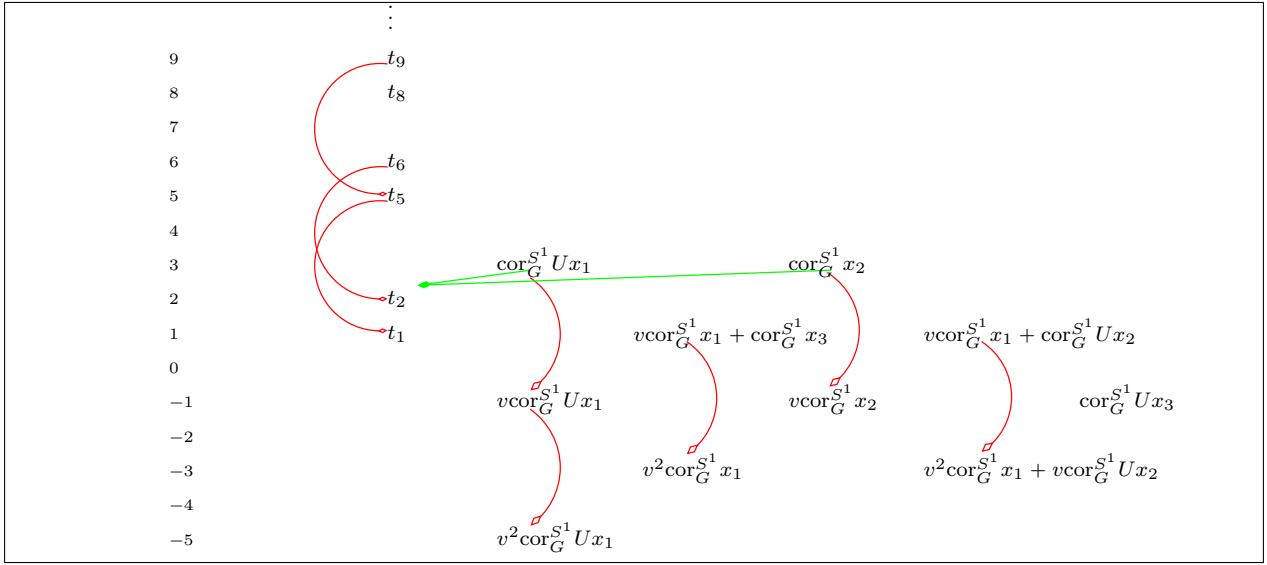}
\caption{Finishing the calculation of $\tilde{H}^G_*(X)$ for the example of Figure \ref{fig:6thm1}.  The curved arrows again represent the $v$-action.  The straight arrows indicate a nontrivial $q$-action.}
\label{fig:6thm4}
\end{figure}

We find $\tilde{H}^G_*(X)=\bv_8 \oplus \bv_1 \oplus \bv_2\oplus \bv_{-5}(3) \oplus \bv_{-3}(2)^{\oplus 2} \oplus \bv_{-1}(2) \oplus \bv_{-1}(1)$, in accordance with Theorem \ref{thm:onjspaces}.

\subsection{Chain local equivalence and $j$-split spaces}\label{subsec:chain}

Using Theorem \ref{thm:onjspaces}, we can determine the chain local equivalence class of $j$-split spaces.  We start with some results on $j$-split chain complexes.  
First, write $\mathcal{S}_{d}(n)$ for the free $\G$-module generated by 
\[\langle x_{d},x_{d+2},...,x_{d+2n-2} \rangle,\]
 with $x_i$ of degree $i$ and $\partial(x_i)=s(1+j^2)x_{i-2}$.  A quick computation gives $H^{S^1}_*(\mathcal{S}_{d}(n))=\bt_{d}(n)^{\oplus 2}$ as $\f[U]$-modules, where $H^{S^1}_*(Z)$ is defined as in (\ref{eq:chains1borel}).  Moreover, for an $\f[U]$-module $J=\bigoplus_i \bt_{e_i}(m_i)$, let $S(J)=\bigoplus_i \mathcal{S}_{e_i}(n_i)$.

\begin{prop}\label{prop:loctyp}
Let $C= \langle \fr \rangle \tilde{\oplus} (C_+ \oplus C_-)$ be a $j$-split chain complex and
\begin{equation}\label{eq:formloctyp}
H^{S^1}_*(C)=\bt_{d_1+2n_1-1} \oplus\bigoplus^{N}_{i=1}\bt_{d_i}(\frac{d_{i+1}+2n_{i+1}-d_i}{2})\oplus \bigoplus^N_{i=1} \bt_{d_i}(n_i) \oplus J^{\oplus 2},
\end{equation} 
where $d_{i+1}> d_i$ and $2n_i+d_i>2n_{i+1}+d_{i+1}$, $2n_{N}+d_{N}\geq 3$, and $d_N \leq 1$.  We interpret $d_{N+1}=1$, $n_{N+1}=0$.  Then $C$ is homotopy equivalent to the chain complex
\begin{equation}\label{eq:locstandardtyp}
( \langle \fr \rangle \tilde{\oplus} (\bigoplus_i \mathcal{S}_{d_i}(n_i))) \oplus S(J), \end{equation} where $\partial(\fr)=0$, $j\fr=\fr,\, s\fr=0$, and $\mathrm{deg}\, (\fr)=0$.  Furthermore, let each factor $\mathcal{S}_{d_i}(n_i)$ have generators $x^i_j$, with $\mathrm{deg} \; x^i_j =j$.  Then $\partial x^i_1=\fr+s(1+j^2)x^i_{-1}$ for all $i$.  
\end{prop}   
\begin{rmk}
By Lemma \ref{lem:deco1}, for $C$ any $j$-split chain complex, a decomposition as in (\ref{eq:formloctyp}) is possible.  
\end{rmk}
Before giving the proof we establish a Lemma.   
\begin{lem}\label{lem:he}
Let $F_1,F_2$ be two free, finite $C^{CW}_*(S^1)$-complexes such that $H^{S^1}_*(F_1) \cong H^{S^1}_*(F_2)$ as $\f[U]$-modules.  Then $F_1 \simeq F_2$, where $\simeq$ denotes homotopy equivalence.   
\end{lem}
\begin{proof}
First, we note that $C^{CW}_*(S^1)$ is chain homotopy equivalent to the algebra $\f[\bar{s}]/(\bar{s}^2)$ where $\mathrm{deg}\,(\bar{s})=1$ and $\partial(\bar{s})=0$.   Koszul Duality \cite{GKM} states that $H_*(F_1)$ and $H_*(F_2)$ are isomorphic as $\f[\bar{s}]/(\bar{s}^2)$ modules if and only if $H^{S^1}_*(F_1)$ and $H^{S^1}_*(F_2)$ are isomorphic as $\f[U]$-modules.  Indeed, our original hypothesis was $H^{S^1}_*(F_1) \simeq H^{S^1}_*(F_2)$, so we see that $H_*(F_1)$ and $H_*(F_2)$ are isomorphic as $\f[\bar{s}]/(\bar{s})^2$-modules.  We next observe that two chain complexes over $\f[\bar{s}]/(\bar{s}^2)$ are quasi-isomorphic if and only if they have isomorphic homology, see Whitehead's Theorem 10.4.5 in \cite{Weibel}.  Thus, $F_1$ and $F_2$ are quasi-isomorphic.  Finally, by Theorem 10.4.8 of \cite{Weibel}, quasi-isomorphic free chain complexes are chain homotopy equivalent, and so $F_1$ and $F_2$ are chain homotopy equivalent.  This establishes the Lemma.  
\end{proof}  

\emph{Proof of Proposition \ref{prop:loctyp}.} The proof is in two steps: first, we show that $C_+$ is chain homotopy equivalent to a chain complex of a certain form, and then we investigate differentials from $C_+$ to $\langle \fr \rangle$.  

Note that the complex $C_+$ is a $C^{CW}_*(S^1)$-complex.  Let $\mathcal{S}^{S^1}_d(n)$ be the $C^{CW}_*(S^1)$-submodule of $\mathcal{S}_d(n)$ generated (as a $C^{CW}_*(S^1)$-module) by $\langle x_d,x_{d+2},...,x_{d+2n-2} \rangle$.  As for $\mathcal{S}_d(n)$, a quick calculation shows $H^{S^1}_*(\mathcal{S}^{S^1}_d(n))=\bt_{d}(n)$.  Similarly, for an $\f[U]$-module $J=\bigoplus_i \bt_{e_i}(m_i)$, let $S^{S^1}(J)=\bigoplus_i \mathcal{S}^{S^1}_{e_i}(n_i)$.  We see: \begin{equation}\label{eq:jproj}S(J)\cong S^{S^1}(J) \oplus S^{S^1}(J),\end{equation} as $\G$-complexes, for all $\f[U]$-modules $J$, where the action of $j$ on the right is given  by interchanging the factors.  

Recall, by the proof of Theorem \ref{thm:onjspaces}, that $H^{S^1}_*(C_+ \oplus C_-)$ is determined by $H^{S^1}_*(C)$ for $C$ a $j$-split chain complex (see Remark \ref{rmk:thomonjspacescomplexes}).  That is, from (\ref{eq:formloctyp}):
\[H^{S^1}_*(C_+)=\bigoplus_{i=1}^N\bt_{d_i}(n_i) \oplus J.\]
Lemma \ref{lem:he} then implies $C_+=\mathcal{S}^{S^1}_d(n) \oplus S^{S^1}(J)$ as a $C^{CW}_*(S^1)$-complex.  Since $j:C_+ \rightarrow C_-$ is an isomorphism, we have from (\ref{eq:jproj}):

\begin{equation}\label{eq:c+form}
C_+\oplus C_- \cong \bigoplus_i \mathcal{S}_{d_i}(n_i) \oplus S(J). 
\end{equation}
Moreover, $H^{S^1}_*(C)$ determines the map $d_{S^1}: H^{S^1}_*(C_+) \rightarrow H^{S^1}_*(\langle \fr \rangle)$.  We compute $d_{S^1}$ a different way, by using the differential from $C_+$ to $\langle \fr \rangle$, and the form of $C_+$ determined by (\ref{eq:c+form}).  Fix a pair of integers $(d,n)$.  If $x_i$ is the generator of a copy of $\mathcal{S}_d(n)$ in degree $i$ and $x_i \in C_+$, then $d_{S^1}: H^{S^1}_*(\mathcal{S}_d(n))\cong \bt_d(n) \rightarrow \bt$ is nontrivial if and only if $\partial(x_1)=\fr+s(1+j^2)x_{-1}$.  Thus, since $d_{S^1}$ is nonvanishing on the factors $\bt_{d_i}(n_i)\subset H^{S^1}_*(C_+)$ and vanishing elsewhere, each generator $x^i_1$, with $\mathrm{deg}\; x^i_1=1$ of $\mathcal{S}_{d_i}(n_i)$ in (\ref{eq:c+form}) must have $\partial(x^i_1)=\fr+s(1+j^2)x^i_{-1}$, and all other differentials $C_+ \rightarrow \langle \fr \rangle$ vanish.  Thus, in particular, $\partial(S(J)) \subset S(J)$.  The decomposition (\ref{eq:locstandardtyp}) follows.  \qed

\begin{prop}\label{prop:eqclassdist}
Let $(X,p,h/4)\in \mathfrak{E}$ with $X$ a $j$-split space of type SWF at level $m$, and 
\begin{equation}\label{eq:propclass}
\tilde{H}^{S^1}_*((X,p,h/4))= \bt_{s+d_1+2n_1+1}\oplus\bigoplus^{N}_{i=1}\bt_{s+d_i}(\frac{d_{i+1}+2n_{i+1}-d_i}{2})\oplus \bigoplus^N_{i=1} \bt_{s+d_i}(n_i) \oplus J^{\oplus 2}[-s],\end{equation}
where $d_{i+1}> d_i$ and $2n_i+d_i>2n_{i+1}+d_{i+1}$, as well as $2n_{N}+d_{N}\geq 3$, and $d_N \leq 1$. Then the chain local equivalence type $[(C^{CW}_*(X,\p),p,h/4)]_{cl} \in \CEL$ is the equivalence class of \begin{equation}\label{eq:locstandardtyp2}
C(p-m,h/4,\{ d_i \}_i,\{ n_i \}_i)= (( \langle \fr \rangle \tilde{\oplus} (\bigoplus_i \mathcal{S}_{d_i}(n_i))),p-m,h/4)\in \CEL. \end{equation}
The connected $S^1$-homology of $(X,p,h/4)$ is given by: \begin{equation}\label{eq:connhom}
\hsc((X,p,h/4))=\bigoplus^{N}_{i=1}\bt_{s+d_i}(\frac{d_{i+1}+2n_{i+1}-d_i}{2})\oplus \bigoplus^N_{i=1} \bt_{s+d_i}(n_i).\end{equation} 
Further, $s$ in (\ref{eq:propclass}) is $m-p-h$.  Moreover, $C(p,h/4,\{ d_i \},\{ n_i \})$ is chain locally equivalent to $C(p',h'/4,\{ d'_i \},\{ n'_i \})$ if and only if $p=p'$, $h=h'$, $\{d_i\} = \{d'_i\}$, and $\{n_i\}=\{n'_i\}$.  
\end{prop}

\begin{proof} 
Let $[(Z,-m,0)]=[C^{CW}_*(X,\p)]\in \ce$ where $Z$ is a $j$-split chain complex, as allowed by Lemma \ref{lem:jsplitcomplexes}.  Using Proposition \ref{prop:loctyp}, we see: \[ [(Z,p,h/4)]_{cl}=((\langle \fr \rangle\tilde{\oplus}\bigoplus \mathcal{S}_{d_i}(n_i) ) ,p,h/4).\]  We have then:
\[C(p-m,h/4,\{ d_i \},\{ n_i \})= [(Z,p-m,h/4)]_{cl}=[(C^{CW}_*(X,\p),p,h/4)]_{cl}, \]
as in (\ref{eq:locstandardtyp2}).

To prove (\ref{eq:connhom}) we consider the complex $\Sigma^{\mathbb{H}^{\lfloor \frac{-d_1+3}{4}\rfloor}}C(0,0,\{d_i\},\{n_i\})[4\lfloor \frac{-d_1+3}{4} \rfloor]$ (we include the grading shift for convenience). We will see that it is a suspensionlike complex, so we may apply the results of Section \ref{subsec:iness}.  
There is a homotopy equivalence:
\begin{equation}\label{eq:lochom} \Sigma^{\mathbb{H}^{\lfloor \frac{-d_1+3}{4}\rfloor}}C(0,0,\{d_i\},\{n_i\})[4\lfloor \frac{-d_1+3}{4} \rfloor ]  \simeq \langle \fr \rangle \tilde{\oplus} \bigoplus_k \langle y_k \rangle \tilde{\oplus} \bigoplus_{i=1}^N \bigoplus_{ \{ k \equiv 1 \; \mathrm{mod} \;2, \; d_i \leq k \leq d_i+2n_i-2\} } \langle z^i_k \rangle, \end{equation}
where 
\begin{equation}\label{eq:locspl} \langle \fr \rangle \tilde{\oplus} \bigoplus_k \langle y_k \rangle \simeq \Sigma^{\mathbb{H}^{\lfloor \frac{-d_1+3}{4} \rfloor }} \langle \fr \rangle, \end{equation}
and $\mathrm{deg}\; z^i_k =\mathrm{deg} \; y_k=k$.  Additionally, $\partial(z^i_k)=s(1+j^2)z^i_{k-2}$ if $k \neq 1$, and $\partial(z^i_1)=s(1+j^2)z^i_{-1}+s(1+j)^3y_{-1}$.  The $y_k$ are defined for $k$ such that $k \not \equiv 3 \; \mathrm{mod} \; 4$ and $-4 \lfloor \frac{-d_1+3}{4} \rfloor +1 \leq k \leq -1$.  Also, 
\begin{eqnarray}
\partial(y_{4k})&=&s(1+j)^3 y_{4k-2}, \\
\partial(y_{4k+1}) & = & (1+j)y_{4k}, \; \; k \neq -\lfloor \frac{-d_1+3}{4} \rfloor ,\\
\partial(y_{4k+2}) & = & (1+j)y_{4k+1} + sy_{4k}, \\
\partial(y_{-4\lfloor \frac{-d_1+3}{4} \rfloor +1}) & = & \fr. 
\end{eqnarray}
According to equation (\ref{eq:locspl}), the first two terms on the right of (\ref{eq:lochom}) account for the suspension of the reducible tower, and the $z^i_k$ correspond to the suspension of the free part.  The $z^i_k$ are suspensions of $x^i_k \in \mathcal{S}_{d_i}(n_i) \subset C(0,0,\{d_i\},\{n_i\})$.  
From this presentation, it is clear that the chain complex $\Sigma^{\mathbb{H}^{\lfloor \frac{-d_1+3}{4}\rfloor}}C(0,0,\{d_i\},\{n_i\})[4\lfloor \frac{-d_1+3}{4} \rfloor]$ is irreducible (that is, it may not be written as a non-trivial direct sum of $\G$-chain complexes).  Then by Lemma \ref{lem:homsums} and Definition \ref{def:connplex}, 
\begin{equation}\label{eq:connexample}
(\Sigma^{\mathbb{H}^{\lfloor \frac{-d_1+3}{4}\rfloor}}C(0,0,\{d_i\},\{n_i\})[4\lfloor \frac{-d_1+3}{4} \rfloor])_{\mathrm{conn}}=\Sigma^{\mathbb{H}^{\lfloor \frac{-d_1+3}{4}\rfloor}}C(0,0,\{d_i\},\{n_i\})[4\lfloor \frac{-d_1+3}{4} \rfloor].
\end{equation}  Then (\ref{eq:connhom}) follows from the definition of $\hsc$, applied to $C(0,0,\{d_i\},\{n_i\})$.  The calculation of $\hsc((X,p,h/4))$ for nonzero $m, p,h$ follows, since
\[C(p-m,h/4,\{d_i \}, \{n_i\})=\Sigma^{(m-p)\tilde{\mathbb{R}}}\Sigma^{-\frac{h}{4}\mathbb{H}}C(0,0,\{d_i\},\{n_i\}).\]
The assertion that $s=m-p-h$ follows from the homology calculation of Theorem \ref{thm:onjspaces}.

Recall that $\hsc$ is a chain local equivalence invariant.  Hence, if $[C(p,h/4,\{d_i\},\{n_i\})]_{cl}=[C(p',h'/4,\{d'_i\},\{n'_i\})]_{cl}$, we see from (\ref{eq:connhom}) that $\{d_i\}=\{d'_i\},$  $\{n_i\}=\{n'_i\}$, and $p+h=p'+h'$.  Furthermore, if $C(p,h/4,\{d_i\},\{n_i\})$ and $C(p',h'/4,\{d'_i\},\{n'_i\})$ are chain locally equivalent, they must have chain homotopy equivalent fixed-point sets.  That is, $p=p'$ and so also $h=h'$, completing the proof.
\end{proof}

\section{Seiberg-Witten Floer spectra and Floer homologies}\label{sec:gauge}
\subsection{Finite-dimensional approximation}\label{eq:subsec:findappx}
In this section we review the finite-dimensional approximation to the Seiberg-Witten equations from Manolescu \cite{ManolescuS1},\cite{ManolescuPin}.    

Let $\mathbb{S}$ be the spinor bundle of the three-manifold with spin structure $(Y,\s)$, and $\Gamma(\mathbb{S})$ its space of sections.  For $\lambda \in (0,\infty)$, the Seiberg-Witten equations of $(Y,\s,g)$ determine a sequence of vector fields $\X^{\gc}_\lambda$ on finite-dimensional vector spaces $W^\lambda$.  Here $W^\lambda$ is the span of eigenvectors of the elliptic operator $*d+D$ acting on the global Coulomb slice of $\Omega^1(Y,i\mathbb{R})\oplus \Gamma(\mathbb{S})$, with eigenvalue in $(-\lambda,\lambda)$, where $D$ is the Dirac operator.  The vector field $\X^{\gc}_\lambda$ on $W^\lambda$ is an approximation of the Seiberg-Witten equations restricted to $W^\lambda$.  The action of $G=\mathrm{Pin}(2)$ restricts to a smooth action on $W^\lambda$ that commutes with the flow defined by $\X^{\gc}_\lambda$.  There is a distinguished subspace $W(-\lambda,0) \subset W^\lambda$ consisting of the span of the eigenvectors with eigenvalue in $(-\lambda,0)$.  Following \cite{ManolescuS1}, we will use the sequence of flows on the spaces $W^\lambda$ to define an invariant of $(Y,\s)$.  

We next recall a few properties of the Conley Index.  For a one-parameter family $\phi_t$ of diffeomorphisms of a manifold $M$ and a compact subset $A \subset M$, we define:
\[ \mathrm{Inv}(A,\phi)=\{x \in A \mid \phi_t(x) \in A \; \mathrm{ for \; all} \; t \in \mathbb{R}\}. \]

Then we say that a set $S \subset M$ is an isolated invariant set if there is some $A$ as above such that 
$S=\mathrm{Inv}(A,\phi)\subset \mathrm{int}(A)$.  Conley proved in \cite{Conley} that one may associate to any isolated invariant set $S$ a pointed homotopy type $I(S)$, an invariant of the triple $(M,\phi_t,S)$.  Floer \cite{FloerConley} and Pruszko \cite{Pruszko} defined an equivariant version, so that if a compact Lie group $K$ acts smoothly on $M$ preserving the flow $\phi_t$, then we may associate a pointed $K$-equivariant homotopy type $I_K(S)$.  The Conley Index, as well as its equivariant refinement, are invariant under continuous changes of the flow, if $S$ is isolated in an appropriate sense.  

Manolescu showed in \cite{ManolescuPin} that $S^\lambda$, the set of all critical points of $\X^{\gc}_\lambda$, along with all trajectories \emph{of finite type} between them contained in a certain sufficiently large ball in $W^{\lambda}$, is an isolated invariant set, and that the flow $\X^{\gc}_\lambda$ is $G$-equivariant.  We then write $I^\lambda(Y,\s,g)=I_G(S^\lambda)$.  To make this construction independent of $\lambda$, we desuspend by $W(-\lambda,0)$.  Then we can define a pointed stable homotopy type associated to a tuple $(Y,\mathfrak{s},g)$: \begin{equation}\label{eq:swfpredef}\mathit{SWF}(Y,\s,g)=\Sigma^{-W(-\lambda,0)}I^\lambda(Y,\s,g).\end{equation}
The desuspension in (\ref{eq:swfpredef}) is interpreted in $\E$.  That is, 
\[\mathit{SWF}(Y,\s,g)=(I^\lambda(Y,\s,g),\mathrm{dim}_{\mathbb{R}} \; W(-\lambda,0)(\tilde{\mathbb{R}}), \mathrm{dim}_{\mathbb{H}} \; W(-\lambda,0)(\mathbb{H})),\]
where $W(-\lambda,0)\cong W(-\lambda,0)(\tilde{\mathbb{R}}) \oplus W(-\lambda,0)(\mathbb{H})$, and $W(-\lambda,0)(\tilde{\mathbb{R}})$ is a direct sum of copies of $\tilde{\mathbb{R}}$.  Similarly, $W(-\lambda,0)(\mathbb{H})$ is a direct sum of copies of $\mathbb{H}$.  

Manolescu showed in \cite{ManolescuPin} that $\mathit{SWF}(Y,\s,g)$ is well-defined, for $\lambda$ sufficiently large.  Further, we must remove the dependence on the choice of metric $g$.  We use $n(Y,\mathfrak{s},g)$, a rational number which controls the spectral flow of the Dirac operator and may be expressed as a sum of eta invariants; for its definition, see \cite{ManolescuS1}.  We have:
\begin{equation}\label{eq:swfdef}
\mathit{SWF}(Y,\mathfrak{s})=\Sigma^{-\frac{1}{2}n(Y,\mathfrak{s},g)\mathbb{H}}\mathit{SWF}(Y,\mathfrak{s},g).
\end{equation}

Interpreted in $\E$, if $\mathit{SWF}(Y,\s,g)=(X,m,n)$, then $\mathit{SWF}(Y,\s)=(X,m,n+\frac{1}{2}n(Y,\s,g))$.  

In addition to the approximate flow above, we may also consider perturbations of the flow as in \cite{KM}.

We call
\[ \mathcal{C}(Y,\s)=\Omega^1(Y,i\mathbb{R}) \oplus \Gamma(\mathbb{S}) \]
the configuration space for the Seiberg-Witten equations, and we let $\mathcal{L}$ denote the Chern-Simons-Dirac functional.  Let $\X$ be the $L^2$-gradient of $\mathcal{L}$ on $\mathcal{C}(Y,\s)$.  We call a map:
\begin{equation}\label{eq:pert}
\q: \mathcal{C}(Y,\mathfrak{s}) \rightarrow \mathcal{T}_0,
\end{equation}
a perturbation, where $\mathcal{T}_j$ denotes the $L^2_j$ completion of the tangent bundle to $\mathcal{C}(Y,\mathfrak{s}).  $
Then we write 
\[\X_\q=\X +\q: \mathcal{C}(Y,\mathfrak{s}) \rightarrow \mathcal{T}_0.\]
  Let $\mathcal{C}^{\mathrm{gC}}(Y,\s)$ denote the global Coulomb slice in $\mathcal{C}(Y,\s)$ and $\mathcal{T}^{\mathrm{gC}}_k$ the $L^2_k$ completion of the tangent bundle to $\mathcal{C}^{\mathrm{gC}}(Y,\s)$.  Lidman and Manolescu also consider a version of $\X_\q$, obtained by projecting trajectories of $\X_\q$ to the global Coulomb slice $\mathcal{C}^{\mathrm{gC}}(Y,\s)$:  
\[\X^{\mathrm{gC}}_\q: \mathcal{C}^{\mathrm{gC}}(Y,\s)\rightarrow \mathcal{T}^{\mathrm{gC}}_0. \]
  Lidman and Manolescu prove that there is a bijective correspondence between finite-energy trajectories of $\X^{\mathrm{gC}}_\q$ and those of $\X_{\q}$, modulo the appropriate gauges.
  
  We write $\X_{\q,\lambda}^{\mathrm{gC}}$ for the finite-dimensional approximation of $\X_{\q}^{\mathrm{gC}}$ in $W^\lambda$ (recalling that $W^\lambda$ are finite-dimensional subspaces of $\mathcal{C}^{\mathrm{gC}}(Y,\s)$).  For tame perturbations in the sense of \cite{KM}, we may define $I^{\lambda}(Y,\mathfrak{s},g,\q)$ as above using $\mathcal{X}^{\mathrm{gC}}_{\q,\lambda}$ in place of $\X^{\gc}_\lambda$.  Furthermore, from $I^{\lambda}(Y,\mathfrak{s},g,\q)$ we may also define $\mathit{SWF}(Y,\mathfrak{s},g,\q)$ analogously to the unperturbed case. Lidman and Manolescu \cite{LM} show that the spectrum is independent of $\q$.  That is:
\[\mathit{SWF}(Y,\mathfrak{s},g,\q)=\mathit{SWF}(Y,\mathfrak{s},g).\]
We also have the attractor-repeller sequence of \cite{ManolescuPin}.  For a generic perturbation $\q$ we may arrange that the reducible critical point of $\X_\q$ is nondegenerate and that there are no irreducible critical points $x$ with $\mathcal{L}(x) \in (0,\epsilon)$ for some $\epsilon>0$. Denote the reducible critical point by $\Theta$.  Let $T=T^\lambda$ be the set of all critical points of $\X_{\q,\lambda}^{\mathrm{gC}}$ and flows of finite type between them.  Then, for all $\omega>0$, we have the following isolated invariant sets:

\begin{itemize}
\item $T^{\mathrm{irr}}_{>\omega}$: the set of irreducible critical points $x$ with $\mathcal{L}_{\q}(x)>\omega$, together with all points on the flows between critical points of this type.
\item $T_{\leq \omega}$ : Same, but with $\mathcal{L}_{\q}(x) \leq \omega$, and allowing $x$ to be reducible.
\end{itemize}

Then we have the exact sequence:
\begin{equation}\label{eq:attrep}
I(T_{\leq \omega})\rightarrow I(T) \rightarrow I(T^{\mathrm{irr}}_{> \omega}) \rightarrow \Sigma I(T_{\leq \omega}) \rightarrow ...
\end{equation}

We record a Theorem of \cite{ManolescuPin}.

\begin{thm}[Manolescu {\cite{ManolescuPin},\cite{ManolescuK}}]\label{thm:ManolescuSWF}
Associated to a three-manifold with $b_1=0$ and a choice of spin structure $(Y,\mathfrak{s})$ there is an invariant $\mathit{SWF}(Y,\mathfrak{s})$, the Seiberg-Witten Floer spectrum class, in $\E$.  A spin cobordism $(W,\mathfrak{t})$ from $Y_1$ to $Y_2$, with $b_2(W)=0$, induces a map $\mathit{SWF}(Y_1, \mathfrak{t}|_{Y_1}) \rightarrow \mathit{SWF}(Y_2,\mathfrak{t}|_{Y_2})$. The induced map is a homotopy-equivalence on the $S^1$-fixed-point set.  
\end{thm}

\begin{rmk}
The three-manifold $Y$ in Theorem \ref{thm:ManolescuSWF} may be disconnected.
\end{rmk}
\begin{defn}
For $[(X,m,n)] \in \E$, we set 
\begin{equation}
\alpha((X,m,n))=\frac{a(X)}{2}-\frac{m}{2}-2n,
\;\beta((X,m,n))=\frac{b(X)}{2}-\frac{m}{2}-2n,
\end{equation}
\[
\gamma((X,m,n))=\frac{c(X)}{2}-\frac{m}{2}-2n.
\]
The invariants $\alpha,\beta$ and $\gamma$ do not depend on the choice of representative of the class $[(X,m,n)]$.  
\end{defn}  

The Manolescu invariants $\alpha(Y,\mathfrak{s}),\beta(Y,\mathfrak{s}),\gamma(Y,\mathfrak{s})$ of a pair $(Y,\mathfrak{s})$ are then given by $\alpha(\mathit{SWF}(Y,\mathfrak{s}))$, $\beta(\mathit{SWF}(Y,\mathfrak{s}))$, and $\gamma(\mathit{SWF}(Y,\mathfrak{s}))$, respectively.  

 From Theorem \ref{thm:ManolescuSWF}, the local and chain local equivalence classes of $\mathit{SWF}(Y,\mathfrak{s})$, $[\mathit{SWF}(Y,\mathfrak{s})]_{l}$ and $[\mathit{SWF}(Y,\mathfrak{s})]_{cl}$, respectively, are homology cobordism invariants of the pair $(Y,\mathfrak{s})$.  Since the $G$-Borel homology of $\mathit{SWF}(Y,\s)$ depends only on $[\mathit{SWF}(Y,\s)]_{cl}$, we have that $\alpha(Y,\mathfrak{s}),\beta(Y,\mathfrak{s}),$ and $\gamma(Y,\mathfrak{s})$ depend only on the chain local equivalence class $[\mathit{SWF}(Y,\mathfrak{s})]_{cl}$.  
\begin{fact}\label{fct:homf}
 Let $Y_1, Y_2$ be rational homology three-spheres with spin structures $\mathfrak{t}_1,\mathfrak{t}_2$ and $(X_i,m_i,n_i)=\mathit{SWF}(Y_i,\mathfrak{t}_i)$ for $i=1,2$.  Then \[\mathit{SWF}(Y_1 \# Y_2,\mathfrak{t}_1 \# \mathfrak{t}_2) \equiv_l (X_1 \wedge X_2, m_1+m_2,n_1+n_2).\]
 \end{fact}
 \begin{proof}
 According to \cite{ManolescuPin}, the Seiberg-Witten Floer spectrum class of the disjoint union $Y_1 \amalg Y_2$ is given by:
 \[\mathit{SWF}(Y_1 \amalg Y_2) \equiv_l (X_1 \wedge X_2, m_1+m_2,n_1+n_2).\]
 On the other hand $Y_1 \amalg Y_2$ is homology cobordant to the connected sum $Y_1 \# Y_2$.  Since the local equivalence class is a homology cobordism invariant, we obtain the claim.  
 \end{proof}
 
 By Theorem \ref{thm:ManolescuSWF} and Fact \ref{fct:homf}, we have a sequence of homomorphisms:
\begin{equation}\label{eq:homse} \theta^H_3 \xrightarrow{\mathit{SWF}} \el \xrightarrow{C_*} \CEL. \end{equation}

\subsection{The Morse-Bott condition and approximate trajectories}\label{sec:tr}

Let $\q$ be an admissible perturbation, as in Definition 22.1.1 of \cite{KM}.  We will need the following results of Lidman-Manolescu \cite{LM}.  
 
\begin{prop}\cite{LM}\label{prop:LM1}
For $\lambda$ sufficiently large, there is a grading-preserving isomorphism between the set of critical points of the finite-dimensional approximation $\X^{\mathrm{gC}}_{\q,\lambda}$ and the set of critical points of $\X_\q$.  
\end{prop}

For $x,y$ critical points of $\X^{\mathrm{gC}}_{\q,\lambda}$, let $M_\lambda(x,y)$ denote the set of unparameterized trajectories of $\X^{\mathrm{gC}}_{\q,\lambda}$ from $x$ to $y$ contained in the ball used to define $S^\lambda$.  Similarly, we let $M(x,y)$ be the set of unparameterized trajectories between critical points of $\X_\q$.  
\begin{prop}\cite{LM}\label{prop:LM2}
There is a correspondence of degree one trajectories compatible with Proposition \ref{prop:LM1}.  That is, if $x_\lambda,y_\lambda$ are critical points, with $\mathrm{gr}  (x_\lambda) = \mathrm{gr}  (y_\lambda) +1$, of $\X^{\mathrm{gC}}_{\q,\lambda}$ corresponding to critical points $x,y $ of $\X_\q$, respectively, then there is an identification
\[ M(x,y) = M_\lambda(x_\lambda,y_\lambda).\]
\end{prop} 

The condition $\mathrm{gr}(x)=\mathrm{gr}(y)+1$ allows the application of an inverse function theorem.  However, without the grading assumption, a compactness result still holds, providing:  

\begin{prop}\cite{LM}\label{prop:LM3}
Let $x$ and $y$ be critical points of $\X_\q$ corresponding to critical points $x_\lambda,y_\lambda$ of $\X^{\mathrm{gC}}_{\q,\lambda}$.  If $M(x,y)=\emptyset$, then $M_\lambda(x_\lambda,y_\lambda)=\emptyset$.  
\end{prop}

We will also need the following Theorem from \cite{LM}.  

\begin{thm}\cite{LM}\label{thm:LM}
Let $(Y,\mathfrak{s})$ be a rational homology three-sphere with spin structure.  Then
\[ \widecheck{HM}(Y,\mathfrak{s})=\mathit{SWFH^{S^1}}(Y,\mathfrak{s}), \]  
as absolutely graded $\f[U]$-modules, where $\widecheck{HM}(Y,\s)$ denotes the ``to" version of monopole Floer homology defined in \cite{KM}.   
\end{thm}

\subsection{Connected Seiberg-Witten Floer homology}

 \begin{defn}
Let $(Y,\s)$ be a rational homology three-sphere with spin structure, and \[[\mathit{SWF}(Y,\s)] =(Z,m,n) \in \ce,\] with $Z$ suspensionlike.  The \emph{connected Seiberg-Witten Floer homology} of $(Y,\mathfrak{s})$, $\hfc(Y,\mathfrak{s})$, is the quotient $(H^{S^1}_*(Z)/(H_*^{S^1}(Z^{S^1})+H^{S^1}_*(\ine)))[m+4n]$, where $\ine\subset Z$ is a maximal inessential subcomplex.  By Theorem \ref{thm:decomp}, the isomorphism class of $\hfc(Y,\mathfrak{s})$ is a homology cobordism invariant.  
\end{defn}
\begin{rmk}
We could have instead considered the quotient $(H^{S^1}_*(Z)/H^{S^1}_*(\ine))[m+4n]$, which is isomorphic to $\hfc(Y,\s)\oplus \bt_d$ where $d$ is the Heegaard Floer correction term of $(Y,\s)$.  As defined above, $\hfc(Y,\s)$ has no infinite $\f[U]$-tower, because of the quotient by $H^{S^1}_*(Z^{S^1})$.  Further, let $Z_{\mathrm{conn}}$ denote the connected complex (Definition \ref{def:connplex}) of $Z$.  It is clear from the construction that 
\[ \hfc(Y,\s)=(H^{S^1}_*(Z_{\mathrm{conn}})/H^{S^1}_*(Z^{S^1}))[m+4n].\]
\end{rmk}
\begin{rmk}
Let $\phi$ be the canonical isomorphism:
\[\phi: H^{S^1}(\mathit{SWF}(Y,\mathfrak{s})) \rightarrow \widecheck{HM}(Y,\mathfrak{s}) \rightarrow \mathit{HF^+}(Y,\mathfrak{s}), \]
provided by, for the first map, \cite{LM}, and for the second, \cite{CGH1} and \cite{KLT1}.  Let $\pi$ be the projection $\pi:\mathit{HF^+}(Y,\s) \rightarrow \mathit{HF_{\mathrm{red}}}(Y,\s)$.  
We note that $\hfc(Y,\mathfrak{s})$ is naturally isomorphic to the quotient 
\[\pi (\phi(H_*^{S^1}(\mathit{SWF}(Y,\mathfrak{s})))/\phi(H_*^{S^1}(\ine)).\]
Then $\hfc(Y,\s)$ can be viewed as a summand of $HF_{\mathrm{red}}(Y,\s)$.    
\end{rmk}

\section{Floer spectra of Seifert fiber spaces}\label{sec:seifmod}
\subsection{The Seiberg-Witten equations on Seifert spaces}\label{sec:moy}
In this section we record some results of \cite{MOY} to describe explicitly the monopole moduli space on Seifert fiber spaces.  First we recall some notation associated with Seifert fiber spaces.

The standard fibered torus corresponding to a pair of integers $(a,b)$, for $a>0$, is the mapping torus of the automorphism of the disk $D^2$ given by rotation by $2\pi b/a$.  Let $D^2_{a}$ be the standard disk, given an orbifold structure by letting $\mathbb{Z}/a$ act by rotation by $2\pi/a$; the origin is then an orbifold point, with multiplicity $a$.  The standard fibered torus is naturally a circle bundle over the orbifold $D^2_a$.  

Let $f: Y \rightarrow P$ be a circle bundle over an orbifold $P$, and $x \in P$ an orbifold point with multiplicity $a$.  If a neighborhood of the fiber over $x$ is equivalent, as an orbifold circle bundle, to the standard fibered torus corresponding to $(a,b)$, we say that $Y$ has \emph{local invariant} $b$ at $x$.  

For $a_i \in \mathbb{Z}_{\geq 1}$, let $S(a_1,\dots,a_k)$ denote the orbifold with underlying space $S^2$ and $k$ orbifold points, with corresponding multiplicities $a_1,\dots,a_k$.  Fix $b_i \in \mathbb{Z}$ with $\gcd(a_i,b_i)=1$ for all $i$.  We let $\Sigma(b,(b_1,a_1),\dots, (b_k,a_k))$ denote the circle bundle over $S(a_1,\dots,a_k)$ with first Chern class $b$ and local invariants $b_i$.  We define the \emph{degree} of the Seifert space $\Sigma(b,(b_1,a_1), \dots, (b_k,a_k))$ by $b+\sum \frac{b_i}{a_i}$.  Finally, we call a space $\Sigma(b,(b_1,a_1),\dots,(b_k,a_k))$ negative (positive) if $b+\sum \frac{b_i}{a_i}$ is negative (positive).  The spaces $\Sigma(b,(b_1,a_1), \dots, (b_k,a_k))$ of nonzero degree are rational homology spheres.  As orbifold circle bundles, the orientation reversal $-\Sigma(b,(b_1,a_1), \dots, (b_k,a_k))$ is isomorphic to $\Sigma(-b,(-b_1,a_1), \dots, (-b_k,a_k))$.  We write $\Sigma(a_1,\dots,a_k)$ for the unique negative Seifert integral homology sphere fibering over $S^2(a_1,\dots,a_k)$.

Let $Y$ be a negative Seifert rational homology three-sphere fibering over a base orbifold $P$ with underlying space $S^2$.  Equipping $Y$ with the metric for which $Y$ has the Seifert geometry, Mrowka, Ozsv{\'a}th, and Yu \cite{MOY} show that the Seiberg-Witten moduli space $\mathcal{M}(Y)$ is composed of the following:  
\begin{itemize}
\item A finite set of points forming the reducible critical set, in bijection with $\mathit{Hom}(H_1(Y),S^1)$, and 
\item for each $(k+1)$-tuple of non-negative integers $\textbf{e}=(e,\epsilon_1,...,\epsilon_k)$, such that $ 0 \leq \epsilon_i < a_i$ and

\[e+ \sum^k_{i=1} \frac{\epsilon_i}{a_i} \leq (\frac{k}{2}-1)-\sum^k_{i=1}\frac{1}{2a_i},\]
there are two components, labelled $C^+(\textbf{e})$ and $C^-(\textbf{e})$, in $\mathcal{M}(Y)$.   

\end{itemize}

  Each component $C^+(\textbf{e}), C^-(\textbf{e})$ is a copy of $Sym^e(|\Sigma|)$, where $\Sigma$ is the base orbifold and $|\Sigma|$ its underlying manifold.  Furthermore, $C^+(\textbf{e})$ and  $C^-(\textbf{e})$ are related by the action of $j\in \mathrm{Pin}(2)$.  That is, the restriction of $j$ to $C^+(\textbf{e})$ acts as a diffeomorphism $C^+(\textbf{e}) \rightarrow C^-(\textbf{e})$, and vice versa.  Then, in the quotient of the configuration space by the based gauge group, each $C^\pm(\textbf{e})$ is diffeomorphic to $G \times Sym^e(|\Sigma|)$.  
  
\begin{fact}\label{fct:csddeg}
All reducible critical points $x$ have $\mathcal{L}(x)=0$, where $\mathcal{L}$ is the Chern-Simons-Dirac functional.  All irreducible critical points have $\mathcal{L}>0$.    
\end{fact}

Mrowka, Ozsv{\'a}th, and Yu do not use the Seiberg-Witten equations as in \cite{KM}.  Instead, they replace the Dirac operator $\hat{D}$ associated to the Seifert metric in the equations with $D=\hat{D}-\frac{1}{2}\xi$ for $\xi$ some constant depending on the Seifert fibration.  It is then clear that the Seiberg-Witten equations they consider differ from the usual equations by a tame perturbation $\q_0$ in the sense of \cite{KM}.  Abusing notation somewhat, we call the Seiberg-Witten equations as in \cite{MOY} simply the Seiberg-Witten equations, or the \emph{unperturbed} Seiberg-Witten equations in the sequel.  

In the case of a negative Seifert space $Y$ with four or fewer singular fibers, the Seiberg-Witten equations are transverse in the sense of \cite{KM}, so we may take $\q=\q_0$, as in \cite{MOY}.  

We will further need: 

\begin{fact}\label{fct:crossflows} There are no flows between $C^+(\textbf{e})$ and $C^-(\textbf{f})$ for any $\textbf{e},\textbf{f}$.  The flow of the Seiberg-Witten equations on $Y$ is Morse-Bott, and if $Y$ has four or fewer singular fibers, the perturbation $\q=\q_0$ is \emph{admissible} in the sense of Definition 22.1.1 of \cite{KM}.  
\end{fact} 
Combining Propositions \ref{prop:LM1},\ref{prop:LM2}, and Fact \ref{fct:crossflows}, we have:

\begin{lem}\label{lem:crossflows}
Let $Y=\Sigma(b,(b_1,a_1),\dots,(b_k,a_k))$ be a negative Seifert rational homology three-sphere.  Then $\mathit{SWF}(Y,\s)$ has a representative $(X,m,n) \in \mathfrak{E}$ with $X$ a $j$-split space.  
\end{lem}
\begin{proof} 
We first treat the case where $Y$ has at most four singular fibers. Then the irreducibles are isolated, by Fact \ref{fct:crossflows}.  
 
We recall the attractor-repeller sequence (\ref{eq:attrep}), which shows that $\mathit{SWF}(Y,\s)$ is obtained by successively attaching stable cells $G \times D^{\mathrm{ind}\, C^+(\textbf{e})}$, corresponding to the irreducible critical point $C^+(\textbf{e})$, to the reducible cell in degree $0$.  Let $I_{\leq \omega}$ be the complex obtained by attaching all critical points with $\mathcal{L} \leq \omega$.  We show by induction that $I_{\leq \omega}$ is $j$-split for all $\omega$.  For $\omega=0$, the only critical point is the reducible by Fact \ref{fct:csddeg}, so the statement is vacuous.  Let 
\begin{equation}\label{eq:isplit}  I_{\leq \omega_0}/I^{S^1}=I_{\leq \omega_0}^+ \vee jI_{\leq \omega_0}^+, \end{equation}
for some fixed $\omega_0$, where $I_{\leq \omega_0}^+$ contains all irreducible critical points $C^+(\textbf{e})$ with $\mathcal{L} \leq \omega_0$.  Fix $\textbf{e}_1$ so that $\mathcal{L}(C^+(\textbf{e}_1))> \omega_0$ and $\mathcal{L}(C^+(\textbf{e}_1))$ is minimal among $\mathcal{L}(x)$ for critical points $x$ with $\mathcal{L}(x)>\omega_0$.  By Fact \ref{fct:crossflows}, and Proposition \ref{prop:LM3}, $M_\lambda(x_\lambda,y_\lambda)=0$, where $x_\lambda$ corresponds to $C^{+}(\textbf{e}_1)$, and $y_\lambda$ corresponds to any critical point of $C^-(\textbf{f})$.  Additionally, the Conley Index satisfies:
\[I_{\leq \mathcal{L}(C^+(\textbf{e}_1) \cup jC^+(\textbf{e}_1))}/I_{\leq \omega_0}=G \times D^{\mathrm{ind}\,C^+(\textbf{e}_1)}=S^1\times D^{\mathrm{ind}\, C^+(\textbf{e}_1)} \vee jS^1\times D^{\mathrm{ind}\, C^+(\textbf{e}_1)},\]
as $S^1\times D^{\mathrm{ind}\, C^+(\textbf{e}_1)}$ and $jS^1\times D^{\mathrm{ind}\, C^+(\textbf{e}_1)}$ are disjoint isolated invariant sets.  Since $M_\lambda(x_\lambda,y_\lambda)=0$ for all $y_\lambda \in jI^+_{\leq \omega_0}$ we have that the attaching map of the cell $S^1\times D^{\mathrm{ind}\, C^+(\textbf{e}_1)}$ has target only in $I^+_{\leq \omega_0}$; then we set \[I_{\leq \mathcal{L}(C^+(\textbf{e}_1))}^+=I_{\leq \omega_0}^+ \cup (S^1 \times D^{\mathrm{ind}\, C^+(\textbf{e}_1)}),\] so that the analogue of the splitting (\ref{eq:isplit}) holds:
\begin{equation}\label{eq:analogueispl}
I_{\leq \mathcal{L}(C^+(\textbf{e}_1)\cup C^-(\textbf{e}_1))}/I^{S^1} = I^+_{\leq \mathcal{L}(C^+(\textbf{e}_1)\cup C^-(\textbf{e}_1))} \vee jI^+_{\leq \mathcal{L}(C^+(\textbf{e}_1)\cup C^-(\textbf{e}_1))},
\end{equation}
completing the induction.

In the case of five or more singular fibers, we perturb the Seiberg-Witten equations to be nondegenerate.  We can arrange that for a small perturbation $\q$ the analogue of Fact \ref{fct:crossflows} continues to hold.  That is, there exists some tame admissible perturbation $\q$ such that the set of irreducible critical points of $\X_\q$ may be partitioned into two sets $C^+$ and $C^-$, interchanged by the action of $j$, so that for all $x \in C^+, y \in C^-$, we have $M(x,y)=\emptyset$.  

We show the existence of such a $j$-equivariant perturbation $\q$.  Choose a sequence of small $j$-equivariant tame admissible perturbations $\q_i$, converging to $0$ in $C^\infty$, so that for each $i$ the perturbed Seiberg-Witten equations have non-degenerate irreducible critical points.  Lin establishes the existence of such perturbations in \cite{flin}.  Choose disjoint neighbourhoods $\mathcal{U}^{\pm}(\textbf{e})$ of $C^{\pm}(\textbf{e})$ such that for $i$ sufficiently large all irreducible critical points of $\mathcal{L}_{\q_i}$ lie in \[\bigcup_{\textbf{e}} (\mathcal{U}^{+}(\textbf{e})\cup\mathcal{U}^{-}(\textbf{e})).\]
Let $C^+_i$ denote the set of irreducible critical points of $\mathcal{L}_{\mathfrak{q}_i}$ in $\cup_{\textbf{e}} \mathcal{U}^+(\textbf{e})$ and let $C^-_i$ denote the set of irreducible critical points of $\mathcal{L}_{\mathfrak{q}_i}$ in $\cup_{\textbf{e}}\mathcal{U}^-(\textbf{e})$.  Let $C^{\pm}$ denote the union $\cup_{\textbf{e}} C^{\pm}(\textbf{e})$.  

Say, to obtain a contradiction, that for all $i$ there exists some pair of critical points $x_i \in C^+_i$, $y_i \in C^{-}_{i}$, such that $M(x_i,y_i)$ is nonempty.  The sequences $x_i,y_i$ have limit points $x \in C^+(\textbf{e})$ and $y \in C^-(\textbf{f})$, by Proposition 11.6.4 of \cite{KM}.  Theorem 16.1.3 of \cite{KM} shows that the moduli space of unparameterized broken trajectories (for a fixed perturbation) is compact.  The proof of Theorem 16.1.3 can be applied to a sequence of trajectories $\breve{\gamma}_i$ for perturbations $\q_i$ with $\q_i\rightarrow \q$.  That is, the sequence $\breve{\gamma}_i$ has a limit point a broken trajectory $(\breve{\tau}_1,...,\breve{\tau}_n)$ from $x$ to $y$, for the perturbation $\q$.  Since $x \in C^+, y \in C^-$, there exists a trajectory $\breve{\tau}_k$ from $C^+$ to $C^-$, or there exists a trajectory $\breve{\tau}_k$ from $C^+$ to the reducible and a trajectory $\breve{\tau}_l$ from the reducible to $C^-$.  The first case contradicts Fact \ref{fct:crossflows}.  The second case contradicts the minimality of $\mathcal{L}$ on the reducible (Fact \ref{fct:csddeg}).  Thus, for some perturbation $\q$ as above we have the desired partition.    

The Lemma then follows as in the case of three or four singular fibers. 
\end{proof}   

By Lemma \ref{lem:crossflows}, Theorem \ref{thm:onjspaces} applies to $\mathit{SWF}(Y,\s)$ for $Y$ a Seifert rational homology sphere, and we obtain the following Corollary, from which Theorems \ref{thm:submain} and \ref{thm:main} of the Introduction follow.
\begin{cor}\label{thm:fullswfh}
Let $Y=\Sigma(b,(\beta_1,\alpha_1),\dots,(\beta_k,\alpha_k))$ be a negative Seifert rational homology sphere with a choice of spin structure $\mathfrak{s}$.  Then 
\begin{equation}\label{eq:hffullswfh2}
\mathit{HF^+}(Y,\mathfrak{s})= \bt_{s+d_1+2n_1-1} \oplus \bigoplus^{N}_{i=1}\bt_{s+d_i}(\frac{d_{i+1}+2n_{i+1}-d_i}{2})\oplus \bigoplus^N_{i=1} \bt_{s+d_i}(n_i) \oplus J^{\oplus 2}[-s],
\end{equation}
for some constants $s,d_i,n_i,N$ and some $\f[U]$-module $J$, all determined by $(Y,\mathfrak{s})$.  Furthermore, $2n_i+d_i>2n_{i+1}+d_{i+1}$ for all $i$, $2n_N+d_N \geq 3$, $d_N \leq 1$, and $d_{N+1}=1$, $n_{N+1}=0$.  Let $\mathcal{J}_0=\{(a_k,b_k)\}_k$ be the collection of pairs containing all $(d_i,\lfloor \frac{n_i+1}{2} \rfloor)$ for $d_i \equiv 1 \;\mathrm{mod}\; 4$ and all $(d_i+2,\lfloor \frac{n_i}{2} \rfloor)$ for $d_i \equiv 3 \; \mathrm{mod} \; 4$, counting multiplicity.  Let $(a,b) \succeq (c,d)$ if $a+4b \geq c+4d$ and $a \geq c$, and let $\mathcal{J}$ be the subset of $\mathcal{J}_0$ consisting of pairs maximal under $\succeq$ (not counted with multiplicity).  If $(a,b) \in \mathcal{J}$, set $m(a,b)+1$ to be the multiplicity of $(a,b)$ in $\mathcal{J}_0$.  If $(a,b) \not\in \mathcal{J}$, set $m(a,b)$ to be the multiplicity of $(a,b)$ in $\mathcal{J}_0$.  Let $|\mathcal{J}|=N_0$ and order the elements of $\mathcal{J}$ so that $\mathcal{J}=\{(a_i,b_i) \}_i $, with $a_i+4b_i > a_{i+1} +4b_{i+1}$.   Then: 
\begin{eqnarray} \mathit{SWFH_*^G}(Y,\mathfrak{s}) & = & (\bv_{4\lfloor \frac{d_1+2n_1+1}{4} \rfloor} \oplus \bv_{1} \oplus \bv_{2}\\ && \oplus \bigoplus_{i=1}^{N_0} \bv_{a_i}(\frac{a_{i+1}+4b_{i+1}-a_i}{4}) \oplus \bigoplus_{(a,b) \in \mathcal{J}_0} \bv_a(b)^{\oplus m(a,b)} \oplus \re^{\f[U]}_{\f[v]}J \nonumber\\ & & \oplus \bigoplus_{\{i \mid d_i \equiv 1\; \mathrm{mod}\;4\} } \bv_{d_i+2}(\lfloor \frac{n_i}{2} \rfloor) \oplus \bigoplus_{\{i \mid d_i \equiv 3 \; \mathrm{mod}\; 4 \} } \bv_{d_i}(\lfloor \frac{n_i+1}{2} \rfloor) )[-s]. \nonumber
\end{eqnarray}

The $q$-action is given by the isomorphism $\bv_2[-s] \rightarrow \bv_1[-s]$ and the map $\bv_1[-s] \rightarrow \bv_{4\lfloor\frac{d_1+2n_1+1}{4}\rfloor}[-s]$ which is an $\f$-vector space isomorphism in all degrees (in $\bv_1[-s]$) greater than or equal to $4\lfloor\frac{d_1+2n_1+1}{4}\rfloor+s+1$, and vanishes on elements of $\bv_1[-s]$ of degree less than $4\lfloor\frac{d_1+2n_1+1}{4}\rfloor+s+1$.  We interpret $a_{N_0+1}=1,b_{N_0+1}=0$.

The action of $q$ annihilates $\bigoplus_{i=1}^{N_0} \bv_{a_i}(\frac{a_{i+1}+4b_{i+1}-a_i}{4})[-s]$ and $(\bigoplus_{(a,b) \in \mathcal{J}_0} \bv_a(b)^{\oplus m(a,b)}\oplus \re^{\f[U]}_{\f[v]}J)[-s]$.  

To finish specifying the $q$-action, let $x_i$ be a generator of $\bv_{d_i+2}(\lfloor \frac{n_i}{2} \rfloor)[-s]$ for $i$ such that $d_i \equiv 1 \, \mathrm{mod}\, 4$ (respectively, let $x_i$ be a generator of $\bv_{d_i}(\lfloor \frac{n_i+1}{2} \rfloor)[-s]$ if $d_i \equiv 3 \, \mathrm{mod}\, 4$).  Then $qx_i$ is the unique nonzero element of $(\bv_{4\lfloor \frac{d_1+2n_1+1}{4} \rfloor} \oplus \bv_{1} \oplus \bv_{2})[-s]$ in grading $\mathrm{deg} \, x_i -1$, for all $i$.
\end{cor}
Theorem \ref{thm:main} follows by setting $N=1$ and $d_1=1$; these conditions imply that $d_2+2n_2-d_1=0$, and so the term $\bigoplus^{N}_{i=1}\bt_{s+d_i}(\frac{d_{i+1}+2n_{i+1}-d_i}{2})$ in (\ref{eq:hffullswfh2}) is the zero module in this case.  

The constant $s$ is the grading of the reducible critical point, where the metric on $Y$ is that associated to the Seifert geometry on $Y$.    

\begin{proof}
Let $(X',p,h/4)$ be a $j$-split representative for $\mathit{SWF}(Y,\s)$ at level $m$, and let $s=m-p-h$.  We may choose such a representative for $\mathit{SWF}(Y,\s)$ by Lemma \ref{lem:crossflows}.  Then, using Lemma \ref{lem:deco2}, we have: 
\[
\mathit{SWFH^{S^1}_*}(Y,\s)=\tilde{H}^{S^1}_*(X')[-p-h]=(\bigoplus^{N}_{i=1}\bt_{d_i}(\frac{d_{i+1}+2n_{i+1}-d_i}{2})\oplus \bigoplus^{N}_{i=1}\bt_{d_i}(n_i) \oplus J_2^{\oplus 2} \oplus \bt_{d_1+2n_1-1})[-s].
\]
Applying the equivalence of $\widecheck{HM}$ and $\mathit{SWFH^{S^1}}$ of \cite{LM}, and the equivalence of $\widecheck{HM}$ and $\mathit{HF^+}$ of \cite{CGH1} and \cite{KLT1}, we obtain the expression (\ref{eq:hffullswfh2}).  Then we apply Theorem \ref{thm:onjspaces} to obtain the calculation of $\mathit{SWFH^G_*}$ of the Corollary.  
\end{proof}

Further, using the results of Section \ref{subsec:chain}, we prove the results of the Introduction on homology cobordisms of Seifert spaces.  Corollaries \ref{cor:dini} and \ref{cor:dini2} of the Introduction follow from Proposition \ref{prop:eqclassdistseif} below.  
\begin{prop}\label{prop:eqclassdistseif}
Let $Y=\Sigma(b,(b_1,a_1),\dots,(b_k,a_k))$ be a negative Seifert rational homology three-sphere with a choice of spin structure $\s$, and
\begin{equation}\label{eq:propclassseif}
\mathit{HF^+}(Y,\mathfrak{s})=\bt_{s+d_1+1} \oplus \bigoplus^{N}_{i=1}\bt_{s+d_i}(\frac{d_{i+1}+2n_{i+1}-d_i}{2})\oplus \bigoplus^N_{i=1} \bt_{s+d_i}(n_i) \oplus J^{\oplus 2}[-s] ,\end{equation}
where $d_{i+1}> d_i$ and $2n_i+d_i>2n_{i+1}+d_{i+1}$, as well as $2n_{N}+d_{N}\geq 3$ and $d_N \leq 1$. Then the chain local equivalence type $[\mathit{SWF}(Y,\mathfrak{s})]_{cl} \in \CEL$ is the equivalence class of \begin{equation}\label{eq:locstandardtyp2seif}
C(s,\{ d_i \}_i,\{ n_i \}_i)= (( \langle \fr \rangle \tilde{\oplus} (\bigoplus_i \mathcal{S}_{d_i}(n_i))),0,-s/4)\in \CEL. \end{equation}
Further, the connected Seiberg-Witten Floer homology of $(Y,\s)$ is:    
\begin{equation}\label{eq:hfcconncalc}\hfc(Y,\s)=\bigoplus^{N}_{i=1}\bt_{s+d_i}(\frac{d_{i+1}+2n_{i+1}-d_i}{2})\oplus \bigoplus^N_{i=1} \bt_{s+d_i}(n_i).\end{equation}
Moreover, if $s \neq t$, or $\{d_i\}_i \neq \{ e_i \}_i $, or $\{n_i\}_i \neq \{ m_i\}_i$, the complexes $C(s,\{ d_i \}_i,\{ n_i \}_i)$ and $C(t,\{ e_i \}_i,\{ m_i \}_i)$ are not locally equivalent.  
\end{prop}
\begin{proof}
Let $\mathit{SWF}(Y,\s)=(X,p,h/4)\in \E$ with $X$ a $j$-split space of type SWF.  By the construction of $\mathit{SWF}(Y,\s)$, $X^{S^1} \simeq (\tilde{\mathbb{R}}^p)^+$.  By Lemma \ref{lem:jsplitcomplexes}, $[(X,p,h/4)] \in \ce$ admits a representative $(Z,p',h'/4)$ with $Z$ a $j$-split chain complex, for some $p',h'$.  Since $[(X,p,h/4)]\in \ce$ and $(Z,p',h'/4)$ must have chain homotopy equivalent fixed-point sets, we have:
\[\Sigma^{-\tilde{\mathbb{R}}^p}((\tilde{\mathbb{R}}^p)^+)=[(X^{S^1},p,0)]=(Z^{S^1},p',0)\in \ce.\]
 However, by the requirement that $Z$ is $j$-split, $Z^{S^1} \simeq \langle \fr \rangle$, where $j\fr=s\fr=\partial(\fr)=0$.  Thus, $p'=0$.  Furthermore, by the proof of Corollary \ref{thm:fullswfh}, $-p'-h'=-h'=s$.  Proposition \ref{prop:eqclassdist} applied to $(Z,0,-s/4)$ yields (\ref{eq:locstandardtyp2seif}) from (\ref{eq:locstandardtyp2}) and (\ref{eq:hfcconncalc}) from (\ref{eq:connhom}).
\end{proof}

\subsection{Spaces of projective type}\label{subsec:projtyp}
Let $Y=\Sigma(b,(b_1,a_1),\dots,(b_k,a_k))$ be a negative Seifert rational homology three-sphere.  Consider the case that $\mathit{HF^+}(Y,\mathfrak{s})$ is given by: 
\begin{equation}\label{eq:projhf}
\mathit{HF^+}(Y,\mathfrak{s})=\bt_{2\delta} \oplus \bt_{d}(n) \oplus J^{\oplus 2},
\end{equation} 
for some $\f[U]$-module $J$, where possibly $n=0$.  In particular, by Corollary \ref{thm:fullswfh}, this implies $d+2n-1=2\delta$.  Let $(Z,0,-s/4)=\mathit{SWF}(Y,\s) \in \ce$. Then by Proposition \ref{prop:loctyp}, we may write:  
\begin{equation}\label{eq:prjtyp0}
Z = (\langle \fr \rangle \tilde{\oplus} \mathcal{S}_1(n) ) \oplus S(J)
\end{equation} 
as a direct sum of $C^{CW}_*(S^1)$-chain complexes, with $\partial(x_1)=\fr$, $\partial(x_{2i+1})=s(1+j^2)x_{2i-1}$ for $i=1,...,n-1$.  Here $d=s+1$, by Corollary \ref{thm:fullswfh}.  The complex $Z$ is evidently chain locally equivalent to $\langle \fr \rangle \tilde{\oplus} \mathcal{S}_d(n) $.  For $X$ a space of type SWF, let $\tilde{\Sigma}X$ denote the unreduced suspension of $X$.  The complex (\ref{eq:prjtyp0}), for $\delta>0$, may be realized as the $G$-CW complex associated to 
\begin{equation*}
(\tilde{\Sigma}(S^{2n-1} \amalg S^{2n-1}),0,-s/4),
\end{equation*}
where $S^1$ acts by complex multiplication on each of the two factors, and $j$ interchanges the factors.  Then 
\begin{equation}\label{eq:prjtyp}
[\mathit{SWF}(Y,\mathfrak{s})]_{cl}\equiv [(\tilde{\Sigma}(S^{2n-1} \amalg S^{2n-1}),0,-s/4)]_{cl}.
\end{equation}
We call a negative Seifert rational homology sphere with spin structure $(Y,\mathfrak{s})$ \emph{of projective type} if (\ref{eq:prjtyp}) holds or if the chain local equivalence class of $\mathit{SWF}(Y,\s)$ is $[\langle \fr \rangle]_{cl}$.  Indeed, we have established that $(Y,\s)$ is of projective type if and only if $\mathit{HF^+}(Y,\s)$ takes the form (\ref{eq:projhf}).  The term \emph{of projective type} refers to the fact:
\begin{equation*}
(S^{2n-1} \amalg S^{2n-1})/G \simeq \mathbb{C}P^{n-1}.
\end{equation*}
We can rephrase the projective type condition (\ref{eq:projhf}) in terms of the \emph{graded roots} of \cite{Nemethigr}.  A graded root $(\Gamma,\chi)$ is an infinite tree $\Gamma$ with an action of $\f[U]$, together with a grading function $\chi:\Gamma \rightarrow \mathbb{Z}$.  Associated to any positive Seifert rational homology sphere with spin structure there is a graded root, which, additionally, has an involution $\iota:\Gamma \rightarrow \Gamma$ that preserves the grading.  

We have the following characterization of spaces of projective type in terms of graded roots as a consequence of Corollary \ref{thm:fullswfh}.

\begin{fact}\label{fct:gradedrootcriterion}  
Let $Y=\Sigma(b,(b_1,a_1),\dots,(b_k,a_k))$ be a negative Seifert rational homology sphere with spin structure $\s$.  Let $(\Gamma_{Y},\chi)$ be the graded root associated to $(-Y,\mathfrak{s})$, and let $\iota$ be the associated involution of $\Gamma_{Y}$.  Let $v \in \Gamma_Y$ be the vertex of minimal grading which is invariant under $\iota$.  The space $(Y,\s)$ is of projective type if and only if there exists a vertex $w$, and a path from $v$ to $w$ in $\Gamma_Y$ which is grading-decreasing at each step, with $\chi(w)=\mathrm{min}_{x\in \Gamma_Y} \chi(x)$.  Moreover, $\delta(Y,\s)-\beta(Y,\s)=\chi(v)-\chi(w)$.  \end{fact}

For instance, we refer to Figure \ref{fig:grad1}.  We call a graded root \emph{of projective type} if its homology is of the form (\ref{eq:projhf}), so that a Seifert integral homology sphere is of projective type if and only if its graded root is.

\begin{figure}
\input{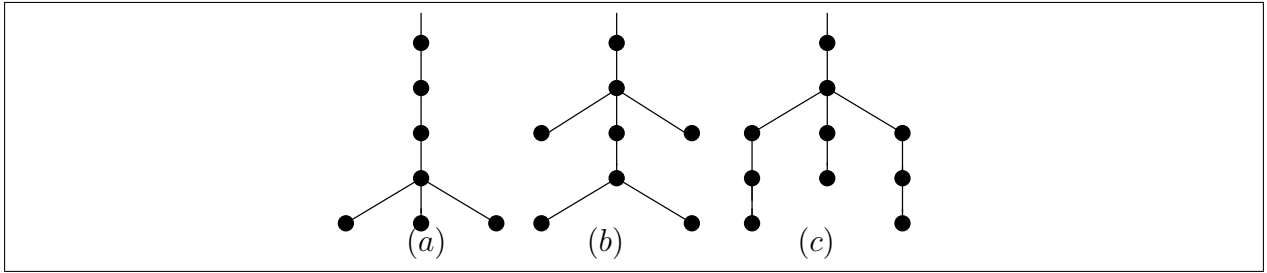}
\caption{Three Graded Roots.  The roots $(a)$ and $(b)$ are of projective type, while $(c)$ is not.  }
\label{fig:grad1}
\end{figure}

More generally, the sets $\{d_i\}$ and $\{n_i\}$ may be read from the graded root, in terms of the minimal grading elements $w$ that are leaves of vertices $v$ that are invariant under $\iota$.  

For spaces $Y$ of projective type, the homology cobordism invariants $(d_i,n_i)$ are determined by $d(Y),\bar{\mu}(Y)$.  The nice topological description of the Seiberg-Witten Floer spectrum of spaces of projective type simplifies calculations.  

The spaces $\Sigma(p,q,pqn+1)$ and $\Sigma(p,q,pqn-1)$ are of projective type for all $p,q,n$, as shown by N\'{e}methi \cite{Nem07} and Tweedy \cite{Tweedy}, respectively, building on work of Borodzik and N\'{e}methi \cite{BN}.  

However, not all Seifert fiber spaces are of projective type.  The Brieskorn sphere $\Sigma(5,8,13)$ is a Seifert space not of projective type, for instance, as one may confirm using graded roots.  Indeed, $\hfc(\Sigma(5,8,13))=\bt_1(2) \oplus \bt_1(1)$.  By Corollary \ref{cor:dini}, any space not of projective type is not homology cobordant to a space of projective type.  In particular, $\Sigma(5,8,13)$ is not homology cobordant to any $\Sigma(p,q,pqn \pm 1)$. 

\subsection{Calculation of Beta}
By the construction of $\mathit{SWF}(Y,\s)$, the grading of the reducible element is $-2n(Y,\s,g)$.  We also saw that the constant $s$ (depending on $(Y,\s)$) in Corollary \ref{thm:fullswfh} is the grading of the reducible (with respect to the Seifert metric).  Also in Corollary \ref{thm:fullswfh}, we saw $s/2=\beta(Y,\s)$ for Seifert rational homology spheres.  We then obtain:
\begin{cor} Let $Y=\Sigma(b,(b_1,a_1),\dots,(b_k,a_k))$ be a negative Seifert rational homology sphere and $\mathfrak{s}$ a spin structure on $Y$.  Then $\beta(Y,\mathfrak{s})=-n(Y,\mathfrak{s},g)$, where $g$ is a metric for which $Y$ has the Seifert geometry. \end{cor}
Ruberman and Saveliev \cite{RubermanSaveliev10} show $n(Y,g) = \bar{\mu}(Y)$ for Seifert integral homology spheres for the Seifert metric, from which we establish Theorem \ref{thm:betamu}.

We have established that $\bar{\mu}$ restricted to Seifert integral homology three-spheres extends to a homology cobordism invariant, but not necessarily that $\bar{\mu}$ extends to a homology cobordism invariant.  In \cite{Manolescusurv} it is shown that $\beta$ is not additive; on the other hand, $\bar{\mu}$ is additive.  Similarly, $\beta$ does not agree with the Saveliev $\nu$ invariant of \cite{SavelievHomCob},\cite{SavelievBrieskorn}, although the two agree on Seifert fiber spaces. 

\section{Applications and Examples}\label{sec:example}
First, we see that Corollary \ref{cor:betamucor} follows from Corollary \ref{thm:fullswfh} and Theorem \ref{thm:betamu}.  Indeed, the negative fibration case follows immediately, and the positive fibration statement follows by using the properties of $\alpha, \beta, \gamma, \bar{\mu}$, and $d$ under orientation reversal.  

We also obtain: 
\begin{thm}  Let $Y$ be a Seifert integral homology sphere.  If $-\bar{\mu}(Y)/2 \neq d(Y)$, then $Y$ is not homology cobordant to any Seifert integral homology sphere with fibration of sign opposite that of $Y$.  
\end{thm} \begin{proof} If $Y$ is a negative Seifert fibration, and $-\bar{\mu}(Y)/2 \neq d(Y)$, then $\alpha(Y) \neq \beta(Y)$, but for all positive fibrations $\alpha=\beta$.  One performs a similar check for positive fibrations. 
\end{proof}

This statement is expressed only in terms of $\bar{\mu}$ and $d$, but the proof comes from the properties of $\alpha, \beta, \gamma$.  As a particular example, we have $\Sigma(2,3,12k-5)$ and $\Sigma(2,3,12k-1)$, for all $k \geq 1$, have $\alpha \neq \beta$ and so are not homology cobordant to any positive Seifert fibration.  

We remark that N{\'e}methi's algorithm \cite{Nemethigr} for Heegaard Floer homology of Seifert fiber spaces makes $\mathit{SWFH^G_*}$ of Seifert spaces computable.  Using Tweedy's computations in \cite{Tweedy}, we provide calculations of $\mathit{SWFH^G_*}$ for the following infinite families as an example.  In the following tables, there are nontrivial $q$-actions between infinite towers.  The only other nontrivial $q$-actions are for $\Sigma(2,7,28k-1)$ and $\Sigma(2,7,28k+15)$, where $q$ sends each summand of $\bv_3(1)^{\oplus k}$ (respectively $\bv_{-1}(1)^{\oplus k+1}$) to $\bv_2$ (respectively $\bv_{-2}$).

\begin{center}
\begin{tabular}{|l|l|r|r|r|r|r|}
\hline
$Y$&  $\mathit{SWFH^G_*}(Y)$ & $\alpha$ & $\beta$ & $\gamma$ &  $\delta$  \\ \hline
$\Sigma(2,5,20k+11)$ & $\mathcal{V}^+_2 \oplus \mathcal{V}^+_{-1} \oplus \mathcal{V}^+_0 \oplus \bv_{-1}(1)^{\oplus k}\oplus \bigoplus^{2k+1}_{i=1}\bv_{-1-2i}(1)$ &1&-1&-1&0\\ \hline
$\Sigma(2,5,20k+1)$ & $\mathcal{V}^+_0 \oplus \mathcal{V}^+_{1} \oplus \mathcal{V}^+_2 \oplus\bv_{-1}(1)^{\oplus k}\oplus \bigoplus^{2k}_{i=1}\bv_{-1-2i}(1)$ &0&0&0&0\\ \hline
$\Sigma(2,5,20k-11)$ & $\mathcal{V}^+_2 \oplus \mathcal{V}^+_{3} \oplus \mathcal{V}^+_4 \oplus \bv_{1}(1)^{\oplus k-1}\oplus \bigoplus^{2k-2}_{i=0}\bv_{-1-2i}(1)$ &1&1&1&1\\ \hline
$\Sigma(2,5,20k-1)$ & $\mathcal{V}^+_4 \oplus \mathcal{V}^+_{1} \oplus \mathcal{V}^+_2 \oplus \bv_1(1)^{\oplus k-1}\oplus \bigoplus^{2k-1}_{i=0}\bv_{-1-2i}(1)$ &2&0&0&1\\ \hline
$\Sigma(2,5,20k-13)$ & $\mathcal{V}^+_0 \oplus \mathcal{V}^+_{1} \oplus \mathcal{V}^+_2 \oplus \bv_{-1}(1)^{\oplus k-1}\oplus \bigoplus^{2k-2}_{i=0}\bv_{-1-2i}(1)$ &0&0&0&0\\ \hline
$\Sigma(2,5,20k-3)$ & $\mathcal{V}^+_2 \oplus \mathcal{V}^+_{-1} \oplus \mathcal{V}^+_0 \oplus \bv_{-1}(1)^{\oplus k-1}\oplus \bigoplus^{2k-1}_{i=0}\bv_{-1-2i}(1)$ &1&-1&-1&0\\ \hline
$\Sigma(2,5,20k+3)$ & $\mathcal{V}^+_2 \oplus \mathcal{V}^+_{3} \oplus \mathcal{V}^+_4 \oplus \bv_{1}(1)^{\oplus k}\oplus \bigoplus^{2k-1}_{i=0}\bv_{-1-2i}(1)$ &1&1&1&1\\ \hline
$\Sigma(2,5,20k+13)$ & $\mathcal{V}^+_4 \oplus \mathcal{V}^+_{1} \oplus \mathcal{V}^+_2 \oplus \bv_{1}(1)^{\oplus k}\oplus \bigoplus^{2k}_{i=0}\bv_{-1-2i}(1)$ &2&0&0&1\\
\hline
\end{tabular}
\end{center}

\begin{flushleft}
{\small
\begin{tabular}{|l|l|r|r|r|r|r|}
\hline
$Y$&  $\mathit{SWFH^G_*}(Y)$ & $\alpha$ & $\beta$ & $\gamma$ &  $\delta$  \\ \hline
$\Sigma(2,7,28k-1)$ & $\mathcal{V}^+_4 \oplus \mathcal{V}^+_{1} \oplus \mathcal{V}^+_2 \oplus \bv_3(1)^{\oplus k}\oplus \bv_1(1)^{\oplus k-1}\oplus\bigoplus^{2k-1}_{i=0}\bv_{-1-2i}(1) \oplus\bigoplus^{2k-1}_{i=0}\bv_{-1-4k-4i}(1)$ &2&0&0&2\\ \hline
$\Sigma(2,7,28k-15)$ & $\mathcal{V}^+_4 \oplus \mathcal{V}^+_{5} \oplus \mathcal{V}^+_6 \oplus \bv_3(1)^{\oplus k-1}\oplus \bv_1(1)^{\oplus k-1}\oplus\bigoplus^{2k-2}_{i=0}\bv_{-1-2i}(1) \oplus\bigoplus^{2k-2}_{i=0}\bv_{1-4k-4i}(1)$ &2&2&2&2\\ \hline
$\Sigma(2,7,28k+1)$ & $\mathcal{V}^+_0 \oplus \mathcal{V}^+_{1} \oplus \mathcal{V}^+_2 \oplus \bv_{-3}(1)^{\oplus k}\oplus \bv_{-1}(1)^{\oplus k}\oplus\bigoplus^{2k}_{i=1}\bv_{-1-2i}(1) \oplus\bigoplus^{2k}_{i=1}\bv_{-1-4k-4i}(1)$ &0&0&0&0\\ \hline
$\Sigma(2,7,28k+15)$ & $\mathcal{V}^+_0 \oplus \mathcal{V}^+_{-3} \oplus \mathcal{V}^+_{-2} \oplus \bv_{-3}(1)^{\oplus k}\oplus \bv_{-1}(1)^{k+1} \oplus\bigoplus^{2k+1}_{i=1}\bv_{-1-2i}(1) \oplus\bigoplus^{2k+1}_{i=1}\bv_{-3-4k-4i}(1)$ &0&-2&-2&0\\ \hline
$\Sigma(2,7,14k-3)$ & $\mathcal{V}^+_2 \oplus \mathcal{V}^+_{3} \oplus \mathcal{V}^+_4 \oplus \bv_1(1)^{\oplus k-1}\oplus \bigoplus^{k-1}_{i=0}\bv_{1-2i}(1)\oplus\bigoplus^{k-1}_{i=0}\bv_{1-2k-4i}(1)$ &1&1&1&1\\ \hline
$\Sigma(2,7,14k+3)$ & $\mathcal{V}^+_2 \oplus \mathcal{V}^+_{-1} \oplus \mathcal{V}^+_0 \oplus \bv_{-1}(1)^{\oplus k}\oplus \bigoplus^{k}_{i=1}\bv_{-1-2i}(1)\oplus\bigoplus^{k}_{i=1}\bv_{-1-2k-4i}(1)$ &1&-1&-1&0\\ \hline
$\Sigma(2,7,14k-5)$ & $\mathcal{V}^+_4 \oplus \mathcal{V}^+_{1} \oplus \mathcal{V}^+_2 \oplus \bv_1(1)^{\oplus k-2}\oplus \bigoplus^{k-1}_{i=0}\bv_{1-2i}(1)\oplus\bigoplus^{k-1}_{i=0}\bv_{1-2k-4i}(1)$ &2&0&0&1\\ \hline
$\Sigma(2,7,14k+5)$ & $\mathcal{V}^+_0 \oplus \mathcal{V}^+_{1} \oplus \mathcal{V}^+_2 \oplus \bv_{-1}(1)^{\oplus k+1}\oplus \bigoplus^{k}_{i=1}\bv_{-1-2i}(1)\oplus\bigoplus^{k}_{i=1}\bv_{-1-2k-4i}(1)$ &0&0&0&0\\
\hline
\end{tabular}
\par}
\end{flushleft}

\bibliography{betaseifert2.bib}
\bibliographystyle{plain}

\end{document}